\documentclass[11pt,reqno]{amsart}
\usepackage{amsfonts, amssymb}
\usepackage{graphicx}
\usepackage[hidelinks]{hyperref}
\usepackage{lipsum,amsmath,amsthm}
\usepackage[numbers,sort&compress]{natbib}
\usepackage{ dsfont }
\usepackage{appendix}
\usepackage{color}
\usepackage{mathrsfs}
\usepackage{esint}
\usepackage{hyperref}
\usepackage{caption}
\usepackage{mathtools}
\usepackage{scalerel}
\usepackage{comment}
\mathtoolsset{showonlyrefs}
\usepackage{todonotes}
\usepackage[right = 2.5cm, left=2.5cm, top = 2.5cm, bottom =2.5cm]{geometry}
\usepackage[utf8]{inputenc}
\usepackage[english]{babel}

\theoremstyle{plain}
\newtheorem{theorem}{Theorem}[section]

\newtheorem{proposition}{Proposition}[section]
\newtheorem{lemma}[proposition]{Lemma}

\theoremstyle{definition}
\newtheorem{remark}{Remark}[section]

\usepackage{todonotes}
\numberwithin{equation}{section}
\allowdisplaybreaks

\newcommand{\leqc}{\lesssim}

\newcommand{\grad}{\nabla}

\newcommand{\tnorm}[1]{\left\lvert\!\left\lvert\!\left\lvert #1
	\right\rvert\!\right\rvert\!\right\rvert}

\newcommand{\R}{\mathbb{R}}
\newcommand{\N}{\mathbb{N}}

\newcommand{\C}{\mathbb{C}}

\newcommand{\Z}{\mathbb{Z}}
\newcommand{\T}{\mathbb{T}}

\newcommand{\loc}{\mathrm{loc}}

\newcommand{\dee}{\mathrm{d}}

\DeclareMathOperator{\supp}{\mathrm{supp}}

\usepackage{bbm}

\newcommand{\E}{\mathbf{E}}

\newcommand{\MM}{\mathcal{M}}

\newtheorem*{lemma*}{Lemma}

\theoremstyle{definition}

\def\Xint#1{\mathchoice
	{\XXint\displaystyle\textstyle{#1}}%
	{\XXint\textstyle\scriptstyle{#1}}%
	{\XXint\scriptstyle\scriptscriptstyle{#1}}%
	{\XXint\scriptscriptstyle\scriptscriptstyle{#1}}%
	\!\int}
\def\XXint#1#2#3{{\setbox0=\hbox{$#1{#2#3}{\int}$}
		\vcenter{\hbox{$#2#3$}}\kern-.5\wd0}}

\def\dashint{\Xint-}

\DeclareUnicodeCharacter{2BC}{'}

\usepackage{zref-clever}

\begin{document}
	
	\title{The Batchelor spectrum for a deterministically driven\\ passive scalar} 
	
	\author{Kyle L. Liss and Jonathan C. Mattingly}
	
	\maketitle
	
	\begin{abstract}
		We study the long-time behavior of a passive scalar transported by an incompressible flow in the presence of smooth, deterministic forcing. For a specific spatially Lipschitz and time-periodic velocity field, we prove that all sufficiently smooth initial data is attracted to a limiting solution that satisfies a cumulative form of Batchelor's law. To our knowledge, this provides the first example for which a version of Batchelor's law can be established with deterministic forcing.
	\end{abstract}
	
	\setcounter{tocdepth}{1}
	{\small\tableofcontents}
	
	\section{Introduction} \label{sec:intro}
	
	The behavior of scalar quantities transported by incompressible flows is a fundamental problem in fluid dynamics, with applications ranging from temperature and salinity transport in geophysical flows to industrial mixing processes. A model problem for the study of such phenomena is the advection-diffusion equation
	\begin{equation}\label{eq:ForcedADE}
		\begin{cases}	\partial_t \rho^\kappa + u \cdot \grad \rho^\kappa = \kappa \Delta \rho^\kappa + F, \\ 
			\rho^\kappa|_{t=0} = \rho_0
		\end{cases}
	\end{equation}
	posed on the $d$-dimensional torus $\T^d \cong [0,1)^d$. This equation describes the evolution of a scalar $\rho^\kappa\colon[0,\infty) \times \T^d \to \R$ that is passively transported by a given divergence-free velocity field $u\colon[0,\infty) \times \T^d \to \R^d$ and undergoes molecular diffusion with strength $\kappa \ge 0$.  The function $F\colon[0,\infty) \times \T^d \to \R$ represents a prescribed external source. 
	
	We will be concerned with the special case of \eqref{eq:ForcedADE} where $\kappa = 0$. The scalar then evolves according to the forced transport equation
	\begin{equation} \label{eq:transport-general}
		\begin{cases} 
			\partial_t \rho + u \cdot \grad \rho = F, \\ 
			\rho|_{t=0} = \rho_0.
		\end{cases}
	\end{equation}
	The main problem of interest in this paper is the long-time Fourier distribution of solutions to \eqref{eq:transport-general} when $u$ is sufficiently smooth and induces exponential growth of scalar gradients, and the source $F$ injects energy at large spatial scales (i.e., low Fourier modes). In 1959, Batchelor \cite{Batchelor1959} argued that in this setting the persistent transfer of energy to higher frequencies produced by the transport term causes the scalar to develop a characteristic power-law distribution in Fourier space. More precisely, his argument predicts that
	\begin{equation} \label{eq:Batchelor1959}
		|\hat{\rho}(t,k)|^2 \approx |k|^{-d}
	\end{equation}
	for large times $t > 0$, where $\hat{\rho}(t,k)$ denotes the Fourier coefficient of $\rho(t)$ for $k \in \Z^d$. This spectral power law is referred to as \textit{Batchelor's law}.
	
	In dimension $d=2$, for a specific deterministic, time-periodic velocity field $u$ that is spatially Lipschitz and for rather general smooth, deterministic, and time-periodic forcing $F$, we prove that
	\eqref{eq:transport-general} admits a time-periodic solution that satisfies a cumulative form of Batchelor's law and attracts all
	solutions with sufficiently smooth initial data. While much
	progress has been made on demonstrating the existence of stationary measures for stochastically forced versions of
	\eqref{eq:ForcedADE} that satisfy some form of
	Batchelor's law, much less is understood in the
	context of deterministically forced advection-diffusion
	equations. We know of no other
	examples of deterministic, spatially regular velocity fields for which a version of Batchelor's law has been proven when the forcing is smooth and deterministic.
	
	Since our velocity field $u$ is Lipschitz, the pure transport equation 
	$$\partial_t \rho + u \cdot \grad \rho = 0$$
	conserves all $L^p$ norms.  Consequently, \eqref{eq:transport-general} is a forced equation with no
	obvious dissipation mechanism, and it is not immediately clear how a time-stationary regime could emerge. Onsager's
	conjecture \cite{onsager1949} in the critical regularity case precisely addresses
	this point. He argued that a forced system that otherwise
	formally conserves energy would evolve, in the infinite-time limit, to a
	solution just rough enough to dissipate energy through an anomalous flux.  Batchelor's prediction fits directly into this perspective. Indeed, the scaling \eqref{eq:Batchelor1959} predicts that solutions of \eqref{eq:transport-general} barely leave $L^2$ as $t \to \infty$, suggesting that regularity just below $L^2$ plays the role of an Onsager-critical threshold at which the advection term can sustain a nontrivial energy flux toward small scales. We show that the
	time-periodic solution constructed in this paper exhibits precisely this behavior: it is contained in $H^{-s}$ for all $s>0$, but does not belong to $L^2$, and sustains a non-vanishing advective energy flux out of all sufficiently small scales. See Remark~\ref{rem:flux} for additional context and discussion.
	
	\subsection{Background and previous works on Batchelor's Law} \label{sec:Batchelor}
	Batchelor's original work was carried out in the context of the forced advection-diffusion equation \eqref{eq:ForcedADE} with weak molecular diffusion, as a model for turbulent scalar advection at length scales below the Kolmogorov scale but above the diffusive cutoff $\sqrt{\kappa}$. In this regime, the velocity field appears effectively smooth and transfers scalar variance injected at large scales by external forcing to progressively finer spatial structures. He predicted that the spectrum \eqref{eq:Batchelor1959} holds over the viscous-convective range of Fourier modes satisfying 
	\begin{equation} \label{eq:BatchelorRange}
		1 \ll |k| \ll \kappa^{-1/2}.
	\end{equation} 
	Batchelor's law is one analogue in passive scalar turbulence of the Kolmogorov-5/3 energy spectrum of hydrodynamic turbulence \cite{K41A, Frisch}. It should be contrasted with the Obukhov-Corrsin theory \cite{Obukhov,Corrsin, KeeferObukhov}, which describes scalar statistics at inertial range scales where the advecting velocity field remains rough and the scalar spectrum is dictated by the regularity of the turbulent flow itself.
	
	Since his original prediction, power-law spectra consistent with Batchelor's law have been observed in passive scalar turbulence over suitable parameter regimes and intermediate ranges of scales in a variety of settings, including natural flows \cite{nye1967scalar,dillon1980batchelor,grant1968spectrum}, laboratory experiments \cite{jullien2000experimental,amarouchene2004batchelor,gibson1963universal}, and direct numerical simulations \cite{bogucki1998direct,yuan2000power,donzis2010batchelor}. Despite the extensive experimental and numerical investigation, the first mathematically rigorous proof of a version of Batchelor's law was given only recently by Bedrossian, Blumenthal, and Punshon-Smith in \cite{BBPS4}, building upon their earlier works \cite{BBPS1,BBPS2,BBPS3} on scalar mixing. They consider \eqref{eq:ForcedADE} with the velocity field $u$ governed by a separate stochastic fluid equation and $F$ a white-in-time stochastic forcing of the form $F = b \dot{W}_t$, where $b$ is a smooth function and $W_t$ is a standard Brownian motion. For a large class of models defining $u$, including the stochastic Navier-Stokes equations with suitably non-degenerate noise, the following ``cumulative'' version of Batchelor's law is proven:
	\begin{equation} \label{eq:BBPScumulative}
		\lim_{t \to \infty} \E \sum_{|k| \le N}|\hat{\rho}^\kappa(t,k)|^2 \approx \log N, \qquad 1 \ll N \ll \kappa^{-1/2},
	\end{equation}
	where $\E$ denotes averaging over the randomness in both the velocity field and stochastic source term. This logarithmic growth is consistent with the ``mode-wise'' law \eqref{eq:Batchelor1959}, though it is a strictly weaker formulation.
	
	It was later shown in \cite{CoopermanRowan} that the cumulative
	estimate \eqref{eq:BBPScumulative} can be strengthened to a version of
	Batchelor’s law that sums over Fourier modes in annuli of fixed width,
	which may be viewed as an intermediate between
	\eqref{eq:Batchelor1959} and \eqref{eq:BBPScumulative}. The same
	authors also recently extended \eqref{eq:BBPScumulative}, in the
	$\kappa = 0$ case, to settings where the random velocity field is
	driven by $C^\infty$ or  degenerate forcing \cite{CoopermanRowanMixing}. The mode-wise law has yet to be proven in any rigorous mathematical setting. Nonetheless, the distinction between the cumulative and mode-wise formulations has been clarified somewhat in the recent work \cite{BlumenthalHuynh}. There, it was shown that in discrete-time analogues of \eqref{eq:ForcedADE}, where the small-scale formation induced by the advection term is highly anisotropic, the cumulative law may hold even when the mode-wise estimate fails badly. This highlights the necessity of some form of isotropy in the velocity field if \eqref{eq:Batchelor1959} is to be expected.
	
	The positive results concerning Batchelor's law mentioned above all rely crucially on the use of a white-in-time stochastic source in \eqref{eq:ForcedADE}. The proof of \eqref{eq:BBPScumulative} in \cite{BBPS4} shows that the cumulative formulation of Batchelor’s law is quite robust in such stochastically forced settings, requiring essentially just suitable exponential mixing properties and sufficient regularity of the velocity field. It is natural to ask whether Batchelor's prediction persists for broader classes of forcing, for instance more physically realistic injection mechanisms like boundary sources. However, even in the comparatively simpler setting of scalar transport on a periodic box in the presence of a deterministic and smooth source term, the validity of the cumulative law already becomes unclear and requires information beyond exponential mixing of the velocity field. We will see that the deterministic setting considered here requires a much more detailed understanding of the velocity than in the white-in-time case.
	
	\vspace{0.2cm}
	\noindent \textbf{Acknowledgments}: The authors thank the NSF for its
	support through the RTG grant DMS-2038056. The authors are also
	grateful for helpful discussions with Alex Blumenthal, Mark Demers, Dmitry
	Dolgopyat, Tarek Elgindi, Carlangelo
	Liverani, Ayman Rimah Said, Vlad Vicol, and Lai-Sang Young.  JCM thanks the Simons Foundation for partial support during the academic year 2024-2025 through a Simons Fellowship. He also thanks the
	Courant Institute at NYU, EPFL in Lausanne, and the SLMath Institute in Berkeley for  their
	hospitality and support during extended stays.
	
	\section{Results and discussion} \label{sec:results}
	
	Let 
	\begin{equation}
		U_1(x,y) = \begin{pmatrix}
			0 \\ 2|x-1/2|
		\end{pmatrix}
		\quad \text{and} \quad U_2(x,y) = \begin{pmatrix}
			2|y-1/2| \\ 0
		\end{pmatrix}.
	\end{equation}
	We consider the alternating ``sawtooth'' shear flow $u_\alpha\colon[0,\infty)\times \T^2 \to \R^2$, defined for an amplitude $\alpha > 0$ by 
	\begin{equation}\label{eq:ualphadef}
		u_\alpha(t,x,y)
		=
		\begin{cases}
			\alpha U_1(x,y) & t \in [0,1/2), \\
			\alpha U_2(x,y) & t \in [1/2,1)
		\end{cases}
	\end{equation}
	and extended periodically in time for $t \ge 1$. This velocity field is divergence free and Lipschitz continuous in space, uniformly in time. It therefore generates a well-defined Lebesgue measure-preserving flow map $\Phi^{\alpha}_{s,t}:\T^2 \to \T^2$, which solves
	\begin{equation}\label{eq:flowmapdef}
		\begin{cases}	\dfrac{\dee}{\dee t} \Phi^\alpha_{s,t}(z) = u_\alpha(t,\Phi^\alpha_{s,t}(z)) & t > s, \\ 
			\Phi^\alpha_{s,s}(z) = z
		\end{cases}
	\end{equation}
	for $z = (x,y) \in \T^2$. The flow generated by $u_\alpha$ was
	analyzed recently by the authors and their collaborator in \cite{ELM25}, where it was shown to be exponentially mixing for all $\alpha$ sufficiently large. Precisely, for the time-one flow map $T_\alpha:=\Phi_{0,1}^\alpha$, there exist constants $c,C > 0$ depending on $\alpha$ such that 
	\begin{equation} \label{eq:umixing}
		\left|\int_{\T^2} (f \circ T_\alpha^{-n})(z) g(z)  \, \dee z \right| \le C e^{-cn}\|f\|_{C^1}\|g\|_{C^1}
	\end{equation}
	for every mean-zero $f,g \in C^1(\T^2)$ and $n \in \N$. Other works that produce exponentially mixing flows on $\T^2$ include \cite{ElgindiZlatos,ACM,YaoZlatos,blumenthal2023exponential, CoopermanRowanMixing}. The only known example besides $u_\alpha$ within the class of time-periodic, Lipschitz velocity fields comes from \cite{Hill2022}, which also considers a family of piecewise linear alternating shears. 
	
	In this paper, we study \eqref{eq:transport-general} with $u = u_\alpha$, namely the forced transport equation
	\begin{equation} \label{eq:transport}
		\begin{cases} 
			\partial_t \rho + u_\alpha \cdot \grad \rho = F, \\ 
			\rho|_{t=0} = \rho_0,
		\end{cases}
	\end{equation}
	where $F\colon[0,\infty) \times \T^2 \to \R$ is deterministic, spatially smooth, and satisfies the mean-zero condition
	\begin{equation} \label{eq:meanzero}
		\int_{\T^2}F(\cdot,z) \,\dee z \equiv 0.
	\end{equation}
	As mentioned earlier, the scaling \eqref{eq:Batchelor1959} suggests that solutions of \eqref{eq:transport} should remain uniformly bounded in every fixed negative Sobolev space $H^{-s}$, but leave $L^2$ in the infinite-time limit. Uniform control of negative Sobolev norms follows readily from Duhamel's formula and the exponential mixing estimate \eqref{eq:umixing} under a mild regularity assumption on $F$, for instance
	$$\sup_{t\ge 0}\|F(t,\cdot)\|_{C^1} < \infty,$$
	which may be interpreted as ensuring that energy injection occurs only at the large scales. On the other hand, the Batchelor scaling and loss of $L^2$ regularity are substantially more delicate phenomena and do not follow in our setting from exponential mixing alone. Addressing these effects is the main contribution of this work. 
	
	Our results are stated precisely in the following two sections. In Section~\ref{sec:discrete}, we consider a discrete-time analogue of \eqref{eq:transport} associated with $T_\alpha$. The discrete-time model captures the essential mechanisms in a simpler setting and allows us to treat arbitrary smooth forcing. Our continuous-time results, stated in Section~\ref{sec:continuous}, follow primarily from the discrete-time analysis, with an additional structural assumption on $F$ required to obtain the full cumulative Batchelor's law.
	
	\subsection{Discrete-time model} \label{sec:discrete}
	
	For a mean-zero function $f \in C^\infty(\T^2)$ with $\|f\|_{L^2} = 1$, define the map $\mathcal{K}_\alpha: H^{1}(\T^2) \to H^{1}(\T^2)$ by
	\begin{equation} \label{eq:ForcedMap}
		\mathcal{K}_\alpha(\rho)= \rho \circ T_\alpha^{-1} + f. 
	\end{equation}
	Note that $\mathcal{K}_\alpha$ indeed maps $H^1$ into itself because $T_\alpha$ and $T_\alpha^{-1}$ are both Lipschitz continuous, and hence composition with either function defines a bounded operator on $H^1$. Iterates of $\mathcal{K}_\alpha$ may be viewed as a discrete-time analogue of \eqref{eq:transport} in the following sense. Given an initial scalar $\rho_0$, define for $n \ge 1$
	\begin{equation} \label{eq:rhon}
		\rho_n = 	\mathcal{K}_\alpha(\rho_{n-1}).
	\end{equation}
	Then, $\rho_n$ coincides with the solution of \eqref{eq:transport} at the integer time $t=n$, with $F$ given by the ``pulsed'' forcing
	\begin{equation} \label{eq:pulsed}
		F(t,z) = \sum_{k=1}^\infty  \delta(t-k)f(z),
	\end{equation}
	where $\delta(\cdot)$ denotes the Dirac delta distribution on $\R$.
	
	In the discrete-time setting above, the natural way to interpret the emergence of Batchelor's law is through the convergence of the iterates $\rho_n$ toward a fixed point of $\mathcal{K}_\alpha$ with the predicted Fourier distribution. This is in contrast to the stochastic settings of \cite{BBPS4,BlumenthalHuynh,CoopermanRowan}, where long-time behavior is described instead in terms of an ergodic stationary measure. Iterating \eqref{eq:ForcedMap}, we see that the limiting object to consider is given by the formal series
	\begin{equation} \label{eq:rhoinftydef}
		\rho_\infty = \sum_{n=0}^\infty f \circ T_\alpha^{-n}.
	\end{equation}
	Our result below shows that for $\alpha$ sufficiently large, $\rho_\infty$ given by the formula above is well defined, satisfies the cumulative Batchelor law, and attracts all sufficiently smooth initial data. 
	
	\begin{theorem} \label{thm:discrete}
		Fix any mean-zero $f \in C^\infty(\T^2)$ with $\|f\|_{L^2} = 1$. For all $\alpha$ sufficiently large, the following hold:
		\begin{itemize}
			\item[(1)] The series \eqref{eq:rhoinftydef} defines a distribution $\rho_\infty \in \left(\bigcap_{s > 0}H^{-s}\right) \setminus L^2$ satisfying $\mathcal{K}_\alpha(\rho_\infty) = \rho_\infty$, where $\mathcal{K}_\alpha$ is extended to a map on $H^{-1}(\T^2)$ by duality. 
			\vspace{0.1cm}
			\item[(2)] For every mean-zero $\rho_0 \in C^1(\T^2)$, we have $\lim_{n \to \infty}\|\mathcal{K}_\alpha^n (\rho_0) - \rho_\infty\|_{H^{-s}} = 0$ for all $s>0$.
			\vspace{0.1cm}
			\item[(3)] The cumulative Batchelor law holds for $\rho_\infty$ in that there exists $C \ge 1$ independent of $\alpha$ and $N_\alpha \ge 1$ such that 
			\begin{equation} \label{eq:Batchelor}
				\frac{1}{C}\frac{\log N}{\log \alpha} \le \|\Pi_{\le N} \rho_\infty\|_{L^2}^2 \le  C\frac{\log N}{\log \alpha}
			\end{equation}
			for every $N \ge N_\alpha$, where $\Pi_{\le N}$ denotes the projection onto wavenumbers $k \in \Z^2$ with $|k| \le N$.
		\end{itemize}
	\end{theorem}
	
	Let us now remark on the challenges that arise with deterministic forcing and the role of large $\alpha$. Further comments regarding Theorem~\ref{thm:discrete} are made below in Remarks~\ref{rem:cumulative}-\ref{rem:generalize}. Aside from the statement that $\rho_\infty \not \in L^2$, parts (1) and (2) above only require that $\alpha$ be large enough for the exponential mixing estimate \eqref{eq:umixing} to hold. The sharp threshold for mixing is not known, but it is likely that $\alpha \ge 4$ is sufficient. Indeed, although the proof of mixing in \cite{ELM25} requires $\alpha$ to be somewhat larger, it relies primarily on the uniform hyperbolicity of $T_\alpha$, and this condition is known to be satisfied as soon as $\alpha \ge 4$. 
	
	The more fundamental application of ``large'' amplitude is in the proof of \eqref{eq:Batchelor}. Here, we are required to study
	\begin{equation} \label{eq:expand}
		\|\Pi_{\le N} \rho_\infty\|_{L^2}^2 = \sum_{n=0}^\infty \|\Pi_{\le N} (f\circ T_\alpha^{-n})\|_{L^2}^2 + 2\sum_{n=1}^\infty \sum_{m=0}^{n-1} \langle \Pi_{\le N} (f\circ T_\alpha^{-n}), \Pi_{\le N}(f\circ T_\alpha^{-m})\rangle_{L^2}.
	\end{equation}
	The general arguments in \cite{BBPS4} (see also \cite[Section 2]{BlumenthalHuynh}) show that the Lipschitz regularity of $T_\alpha$ alone is sufficient to establish
	$$ \sum_{n=0}^\infty  \|\Pi_{\le N} (f\circ T_\alpha^{-n})\|_{L^2}^2 \gtrsim_\alpha \log N, \quad N \gg 1,$$
	while exponential mixing implies the corresponding upper bound. Therefore, proving \eqref{eq:Batchelor} essentially amounts to controlling the off-diagonal terms
	\begin{equation}\label{eq:cross}
		\langle \Pi_{\le N} (f\circ T_\alpha^{-n}), \Pi_{\le N}(f\circ T_\alpha^{-m})\rangle_{L^2}
	\end{equation}
	in \eqref{eq:expand}, which encode interactions between the scalar input at distinct time steps. Such terms play no role in white-in-time stochastically forced settings because temporal decorrelations cause them to vanish in expectation. For this reason, mixing estimates are needed in \cite{BBPS4,CoopermanRowan,BlumenthalHuynh} only to obtain the upper bound portion of Batchelor's law. In our setting here, obtaining a lower bound on $\|\Pi_{\le N} \rho_\infty\|_{L^2}$ already becomes challenging and is in fact the more difficult inequality to prove. Although \eqref{eq:umixing} implies that \eqref{eq:cross} decays as $m \to \infty$ for fixed $n$, without capturing subtle cancellations, treating the off-diagonal contribution in \eqref{eq:expand} as an error relative to the diagonal term requires this decay to be sufficiently strong even for moderate time separations. Taking $\alpha$ large makes this possible by accelerating the mixing rate in \eqref{eq:umixing}. However, the low-frequency cutoff $\Pi_{\le N}$ introduces significant technical difficulties that make the necessary bound on \eqref{eq:cross}, which we prove in Theorem~\ref{prop:Projected}, require an understanding of $T_\alpha$ much deeper than quantitative-in-$\alpha$ exponential mixing. This challenge will be discussed more in Section~\ref{sec:projecteddiscussion}.
	
	\begin{remark}[Cumulative vs. Mode-wise] \label{rem:cumulative}
		The uniform hyperbolicity of $T_\alpha$ has the feature that its directions of expansion and contraction are nearly constant over $\T^2$. As a result, the small scale formation induced by $T_\alpha$ is highly anisotropic. In the same spirit as \cite{BlumenthalHuynh}, it is fairly clear then that \eqref{eq:Batchelor} cannot be strengthened to a mode-wise scaling law, though we do not pursue rigorous statements in this direction. 
	\end{remark}
	
	\begin{remark}[$\log \alpha$ scaling in Batchelor's law] \label{rem:alphascaling}
		The inverse powers of $\log \alpha$ appearing in \eqref{eq:Batchelor} are natural and admit a simple heuristic explanation. Each term $f \circ T_\alpha^{-n}$ in the series defining $\rho_\infty$ has unit $L^2$ mass, which is transported to progressively higher frequencies as $n$ increases. The quantitative-in-$\alpha$ exponential mixing estimate \eqref{eq:uniformmix}, stated in Theorem~\ref{prop:UniformMixing}, implies that the amount of Fourier mass that $f \circ T_\alpha^{-n}$ has below frequency $\alpha^{qn}$ decays rapidly in $n$ for some $q > 0$. Consequently, only on the order $\log N/ \log \alpha$ terms in \eqref{eq:rhoinftydef} contribute non-negligibly at frequencies below $N$.
	\end{remark}
	
	\begin{remark}[Extensions to $\kappa > 0$]
		It would be interesting to extend Theorem~\ref{thm:discrete} to the $\kappa > 0$ setting. In this case, one would consider the pulsed diffusion operator $$\mathcal{K}_\alpha^{\kappa}(\rho) = e^{\kappa \Delta} (\rho \circ T_\alpha^{-1}) + f:= T_\alpha^{(\kappa)}(\rho) + f$$ 
		and hope to prove that the limiting object $\rho_\infty^\kappa$ (which now belongs to $L^2$ for each fixed $\kappa > 0$) satisfies \eqref{eq:Batchelor} for $1 \ll N \ll \kappa^{-1/2}$. Establishing this would require the uniform-in-$\alpha$ mixing estimates presented in Section~\ref{sec:mixingstatements} for $T_\alpha$, specifically \eqref{eq:uniformmix} and \eqref{eq:projectedmix}, to be carried out for $T_\alpha^{(\kappa)}$ uniformly in $\kappa$.  The extension of \eqref{eq:uniformmix} should follow from minor adaptations of our proof by using the stochastic representation of $T_\alpha^{(\kappa)}$ as in \cite{ELM25} (see also \cite[Section A.4]{BlumenthalHuynh}), since our argument relies only on properties of $T_\alpha$ that generalize naturally to the corresponding random maps. It is less clear whether \eqref{eq:projectedmix} can be extended uniformly in $\kappa$, as our proof hinges on finer structural properties of the higher iterates $T_\alpha^n$ (see Lemma~\ref{lem:multipoints}) that do not hold for the random maps . Nevertheless, we expect our proof contains the essential ideas necessary to obtain such an estimate. By contrast, extending the continuous-time results of the next subsection to $\kappa > 0$ may involve new difficulties, as even standard uniform-in-$\kappa$ mixing estimates seem to be unclear in that setting.
	\end{remark}
	
	\begin{remark}[Time-dependent discrete forcing] \label{rem:generalize}
		Our analysis extends with no essential changes to establish Batchelor's law for the iterates
		\[
		\rho_n = \rho_{n-1}\circ T_\alpha^{-1} + f_n,
		\]
		where $\{f_n\}_{n\ge1} \subseteq C^\infty(\T^2)$ is a sequence of mean-zero functions with unit $L^2$ norm and uniformly bounded $C^1$ norms. In this slightly more general setting, Batchelor's law holds for $\rho_n$ only over an intermediate range of spatial scales depending on $n$. We stated Theorem~\ref{thm:discrete} with constant forcing to simplify the presentation by allowing for a natural fixed limiting object that satisfies Batchelor's law at all scales.
	\end{remark}
	
	\subsection{Continuous-time results} \label{sec:continuous}
	
	We now return to the transport equation
	$$\begin{cases}
		\partial_t \rho + u_\alpha \cdot \grad \rho = F \\ 
		\rho|_{t=0} = \rho_0
	\end{cases} 
	$$
	introduced earlier in \eqref{eq:transport}. 
	
	\subsubsection{Cumulative Batchelor's law} 
	
	In the same vein as Remark~\ref{rem:generalize}, we restrict for simplicity to time-one periodic forcing. One then expects the long-time behavior of solutions to be described by a limit cycle that satisfies the cumulative Batchelor law. The result below establishes the existence of such a limiting distribution and an upper bound on its Fourier mass for general smooth forcing functions.
	
	\begin{theorem} \label{thm:continuous0}
		Let $F:[0,\infty) \times \T^2 \to \R$ be a smooth, time-one periodic function satisfying the mean-zero condition \eqref{eq:meanzero}. For $\alpha$ sufficiently large, there exists a time-one periodic distribution $$\rho_\infty \in L^\infty([0,\infty);H^{-s}) \qquad \forall s > 0$$ that solves \eqref{eq:transport} and attracts all sufficiently smooth initial data. That is, for every mean-zero $\rho_0 \in C^1(\T^2)$, the unique solution $\rho$ of \eqref{eq:transport} satisfies $\lim_{t \to \infty}\|\rho(t)- \rho_\infty(t)\|_{H^{-s}} = 0$ for all $s > 0$. Moreover, there exists $C_\alpha \ge 1$ such that 
		\begin{equation} \label{eq:continuousUpper}
			\|\Pi_{\le N} \rho_\infty(t)\|_{L^2}^2 \le  C_\alpha \log N
		\end{equation}
		for every $t \ge 0$ and $N \ge 2$.
	\end{theorem}
	
	Under an additional structural assumption on the forcing, we are able to upgrade \eqref{eq:continuousUpper} to the full cumulative law with the expected scaling in $\alpha$.
	
	\begin{theorem} \label{thm:continuous}
		Let $F:[0,\infty) \times \T^2 \to \R$ be a time-one periodic function, defined for $t \in [0,1)$ by
		\begin{equation} \label{eq:forcespecial}
			F(t,x,y) = \eta(t)h(y),
		\end{equation}
		where $h \in C^\infty(\T)$ is a mean-zero function with $\|h\|_{L^2(\T)} = 1$ and $\eta \in C^\infty([0,1];[0,\infty))$ is a nonzero function supported in $(1/2,1)$. Then, the limit cycle $\rho_\infty$ guaranteed by Theorem~\ref{thm:continuous0} satisfies the cumulative Batchelor law: there exists $C \ge 1$ independent of $\alpha$ and $N_\alpha \ge 1$ such that 
		\begin{equation} \label{eq:continuousBatchelor}
			\frac{1}{C}\frac{\log N}{\log \alpha} \le \|\Pi_{\le N} \rho_\infty(t)\|_{L^2}^2 \le  C\frac{\log N}{\log \alpha}
		\end{equation}
		for all $t \ge 0$ and $N \ge N_\alpha$. In particular, $\rho_\infty(t) \not \in L^2$ for every $t \ge 0$.
	\end{theorem}
	
	We now briefly explain the role of the particular forcing in Theorem~\ref{thm:continuous} and its connection with establishing a lower bound on $\|\Pi_{\le N} \rho_\infty\|_{L^2}$. Let $(\Phi_{s,t}^\alpha)^{-1} = \phi_{t,s}^
	\alpha$, where $\Phi_{s,t}^\alpha$ is the flow map defined in \eqref{eq:flowmapdef}. By Duhamel's formula, the unique solution of \eqref{eq:transport} is then given by 
	\begin{equation}
		\rho(t) = \rho_0 \circ \phi_{t,0}^{\alpha} + \int_0^t F_s \circ \phi_{t,s}^\alpha \, \dee s,
	\end{equation}
	where we have defined $F_s(z) = F(s,z)$. A computation exploiting the joint periodicity of $u_\alpha$ and $F$ (see Lemma~\ref{lem:discretereduction}) shows that the above can be rewritten at integer times as 
	\begin{equation} \label{eq:effectivediscrete}
		\rho(n) = \rho_0 \circ T_\alpha^{-n}+ \sum_{k=0}^{n-1}f_\alpha \circ T_\alpha^{-k},
	\end{equation}
	where 
	\begin{equation}\label{eq:feff}
		f_\alpha = \int_0^1 F_s \circ \phi_{1,s}^\alpha \, \dee s.
	\end{equation}
	One should interpret $f_\alpha$ as the net scalar input accumulated over each period. For typical choices of $F$, say, a generic autonomous function with unit $L^2$ norm, the time averaging in \eqref{eq:feff} produces an effective $f_\alpha$ whose $L^2$ norm is small when $\alpha$ is large. Thus, the large-amplitude regime in continuous time typically corresponds to a small-input regime. At the same time, pushing through the map $\phi_{1,s}^\alpha$ typically makes $f_\alpha$ increasingly irregular as $\alpha$ grows. This combination of roughness and small energy input prevents the isolation of a positive diagonal component capable of dominating the off-diagonal interactions in the continuous-time analogue of \eqref{eq:expand}. 
	
	In view of the discussion above, to obtain lower bounds in the large-amplitude regime it is natural to require $F$ be such that
	\begin{equation} \label{eq:falphaconditions}
		\inf_{\alpha \ge 1}\|f_\alpha\|_{L^2} > 0 \qquad \text{and} \qquad \sup_{\alpha \ge 1} \|\grad f_\alpha\|_{L^\infty} < \infty,
	\end{equation}
	which guarantees that scalar mass accumulates at low frequencies uniformly in $\alpha$. Under these conditions, the representation \eqref{eq:effectivediscrete} falls directly within the scope of Theorem~\ref{thm:discrete}. Our proof of Theorem~\ref{thm:continuous} is based on this reduction and applies equally well to any spatially smooth, time-one periodic $F$ satisfying \eqref{eq:falphaconditions}. The specific choice \eqref{eq:forcespecial} is just one simple class of functions for which \eqref{eq:falphaconditions} holds, as it is easy to see that 
	$$f_\alpha(x,y) = h(y) \int_{1/2}^1 \eta(t) \, \dee t.$$  
	Other examples include the ideal pulsed forcing \eqref{eq:pulsed} and the temporally smoothed version 
	$$F(t,z) = \alpha\sum_{k=1}^\infty \psi(\alpha(t-k))f(z),$$
	where $\psi$ is a bump function centered at $t = 0$.
	
	\subsubsection{Non-vanishing energy flux}\label{sec:Flux}
	
	In the zero-diffusivity setting of \eqref{eq:transport}, the time-periodic behavior of $\rho_\infty$ requires that energy injected at low frequencies by the forcing be exactly balanced, on average, by a transfer of scalar mass to higher frequencies from advection. Indeed, taking the $L^2$ inner product of \eqref{eq:transport} with $\Pi_{\le N} \rho_\infty$ yields 
	$$\frac{1}{2}\frac{\dee}{\dee t} \|\Pi_{\le N} \rho_\infty(t)\|_{L^2}^2 +  \langle \Pi_{\le N} \rho_\infty(t), \Pi_{\le N}\bigl(u_\alpha \cdot \nabla \rho_\infty(t)\bigr) \rangle_{L^2} = \langle \Pi_{\le N} \rho_\infty(t), F_t \rangle_{L^2}.$$
	Integrating over $t \in [0,1]$ and using the time-periodicity of $\rho_\infty$ then gives the balance
	\begin{equation}\label{eq:flux}
		\mathcal{T}_N := \int_0^1 \langle \Pi_{\le N} \rho_\infty(t), \Pi_{\le N}\bigl(u_\alpha \cdot \nabla \rho_\infty(t)\bigr) \rangle_{L^2}\,\dee t = \int_0^1 \langle \Pi_{\le N} \rho_\infty(t), F_t \rangle_{L^2} \, \dee t.
	\end{equation}
	Physically, $\mathcal{T}_N$ represents time-averaged advective $L^2$-energy flux out of Fourier modes with  $|k| \le N$. In white-in-time stochastically forced settings, the right-hand side of \eqref{eq:flux} comes from an It\^{o} correction and is independent of the solution in expectation. In such cases, positivity of the corresponding flux is automatic once the existence of a stationary measure is established (see, e.g., \cite[Theorem 1.12]{BBPS4}). This is no longer true in the deterministic setting here. Although seemingly quite unlikely, it is a priori possible for the structure of $\rho_\infty$ to result in cancellations that cause the net energy input over each period to vanish or even be negative. In the setting of Theorem~\ref{thm:continuous}, we show that this does not occur, and that $\mathcal{T}_N$ remains strictly positive at sufficiently small scales. 
	\begin{theorem} \label{thm:flux}
		Let $F$ and $\rho_\infty$ be as in Theorem~\ref{thm:continuous}. Then, the limit $\lim_{N\to \infty} \mathcal{T}_N$ exists and is strictly positive.
	\end{theorem} 
	
	\begin{remark}\label{rem:flux}
		Our proof relies on the particular structure of \eqref{eq:forcespecial} and does not extend directly to arbitrary forcing satisfying \eqref{eq:falphaconditions}. More general sufficient conditions on $F$ ensuring positivity of $\mathcal{T}_N$ could be obtained by expanding out the right-hand side of \eqref{eq:flux} and using the explicit formula for $\rho_\infty(t)$, but for simplicity we avoid such generality.
	\end{remark}
	
	\begin{remark}
		In hydrodynamic turbulence, a nonvanishing flux of energy through arbitrarily small scales is closely tied to the persistence of energy dissipation in the inviscid limit, a phenomenon typically referred to as ``anomalous'' dissipation. As described by Onsager's 
		Conjecture \cite{onsager1949}, such anomalous dissipation can occur
		only for Euler solutions that are sufficiently rough, lying at or just
		below a critical regularity threshold
		\cite{Eyink_1994,ConstETtit, Eyink_Sreenivasan_2006,isettOnsager}. Related anomalous dissipation
		mechanisms also arise in passive scalar turbulence when the advecting
		velocity field is rough, and have been studied extensively recently in
		the mathematics literature
		\cite{DEIJ,VicolArmstrong,EL,CGM,johansson2024anomalous,burczak2023anomalous,RowanHess,
			BrueDeLellis, KeeferObukhov}. The regime considered here is slightly different from advection by a rough velocity field, where it is possible to construct $L^2$ solutions of the transport equation that dissipate energy. In the Batchelor setting, smoothness of the velocity field implies that energy dissipation at $\kappa = 0$ requires a loss of $L^2$ regularity of the scalar itself. Batchelor’s prediction may still be viewed as an Onsager-type statement predicting anomalous dissipation, in which
		scalar regularity just below $L^2$ represents a critical threshold for
		the flux \eqref{eq:flux} to remain positive as $N \to \infty$, effectively dissipating energy out of any fixed box in Fourier space. A
		precise characterization of the Onsager-critical regularity class for
		a passive scalar  was identified in
		\cite[Theorem 1.19 and Lemma 5.4]{BBPS4}, and we do not pursue
		refinements of the regularity of $\rho_\infty$ here. 
		
		Compared to the explorations of non-uniqueness of solutions below the
		Onsager regularity threshold, there are fewer explorations of how a system
		with  forcing evolves from smooth initial data to a rough solution
		with the critical Onsager regularity, where the anomalous dissipation
		generated by the advection term balances the energy injected by the forcing. Other examples of mathematical
		explorations of this 
		phenomenon include simplified linear shell models with deterministic
		and stochastic forcing
		\cite{MattinglySuidanVanden-Eijnden2007a,MattinglySuidanVanden-Eijnden2007b}, the stochastically forced
		Burgers equation \cite{EKhaninMazelSinai,EVandenEijnden}, and some stochastically
		forced nonlinear shell
		models \cite{Barbato,Romito,FriedlanderGlatt-HoltzVicol}.
	\end{remark}

	\section{Proof of main theorems} \label{sec:outline}
	
	In this section, we state the quantitative-in-$\alpha$ exponential mixing estimates that we require and then use them to complete the proofs of the theorems presented in Section~\ref{sec:results}. 
	The mixing estimates will be proven in Sections~\ref{sec:uniform} and~\ref{sec:projected}, and comprise the main technical work of the paper.
	
	\subsection{Quantitative exponential mixing} \label{sec:mixingstatements}
	
	We begin with a quantitative version of \eqref{eq:umixing}. Recall that $\Phi_{s,t}^\alpha$ denotes the flow map associated with $u_\alpha$.
	
	\begin{theorem}[Quantitative-in-$\alpha$ exponential mixing] \label{prop:UniformMixing}
		There exist constants $c, C > 0$ such that for all $\alpha$ sufficiently large, mean-zero $f,g \in W^{1,\infty}(\T^2)$, $n \in \N$, and $t \in [\tfrac{1}{2} + \alpha^{-1/2},1]$ the following
		estimates hold:
		\begin{subequations}
			\begin{align}
				\left|\int_{\T^2} f\circ T_\alpha^{-n}(z)\, g(z)\,\dee z\right|
				&\le C e^{-c n\log\alpha}\,\|f\|_{W^{1,\infty}}\|g\|_{W^{1,\infty}}, \label{eq:uniformmix} \\[0.5em]
				\left|\int_{\T^2} f\circ T_\alpha^{-n}(z)\, g\circ \Phi_{0,t}^\alpha(z)\,\dee z\right|
				&\le C e^{-c n\log\alpha}\,\|f\|_{W^{1,\infty}}\|g\|_{W^{1,\infty}}. \label{eq:uniformmix2}
			\end{align}
		\end{subequations}
	\end{theorem}
	
	\begin{remark} \label{rem:Sobolevdecay}
		It follows from \eqref{eq:uniformmix} and the continuous embedding $H^3(\T^2) \hookrightarrow 
		W^{1,\infty}(\T^2)$ that
		\begin{equation} \label{eq:Sobolevdecay}
			\|f\circ T_\alpha^{-n}\|_{H^{-3}} \le C e^{-cn\log \alpha}\|f\|_{W^{1,\infty}}.
		\end{equation}
		This version of \eqref{eq:uniformmix} will be useful in the proof of Theorem~\ref{thm:discrete} and our extensions to continuous time. The bound \eqref{eq:uniformmix2} involving $\Phi_{0,t}^\alpha$ will be needed only in the proof of Theorem~\ref{thm:flux} to estimate certain terms that arise after expanding out the right-hand side of \eqref{eq:flux}.
	\end{remark}
	
	The proof of Theorem~\ref{prop:UniformMixing} is carried out in Section~\ref{sec:uniform} and based on proving a finite-step contraction estimate for the \textit{transfer operator} $\mathcal{L}_\alpha$, defined by $\mathcal{L}_\alpha f = f \circ T_\alpha^{-1}$, in a suitable \textit{anisotropic norm}. For context and a discussion of our approach, see Section~\ref{sec:anisotropicoutline}.
	
	Next, we present the bound that will be needed to control the off-diagonal terms \eqref{eq:cross}, discussed earlier. Notice that it follows just from Theorem~\ref{prop:UniformMixing} and a simple $W^{1,\infty}$ growth estimate that 
	\begin{align*}\left|\int_{\T^2} \Pi_{\le N}(f \circ T_\alpha^{-n})(z) \Pi_{\le N}(g\circ T_\alpha^{-m})(z) \, \dee z\right| \le C (4\alpha^2)^n \alpha^{-cm}\|f\|_{W^{1,\infty}} \|g\|_{W^{1,\infty}}
	\end{align*}
	for $m > n$. It is unlikely (and probably false) that the exponential decay rate \(\alpha^{-c}\) can be shown to match the growth rate \(4\alpha^2\). Thus, the above only provides strong decay for $0 \le n \ll m$, which, as will be clear from Lemma~\ref{lem:offdiagonal}, is far from sufficient to prove Theorem~\ref{thm:discrete}. The key feature of Theorem~\ref{prop:Projected} below is that it forces \eqref{eq:cross} to vanish as $\alpha \to \infty$ as soon as $|m-n| > 0$, even when $n \approx m$. For purely technical reasons, it will be convenient to work with a regularized version of $\Pi_{\le N}$ which we now introduce.
	
	Let $\varphi:\R^2 \to [0,\infty)$ be an even, smooth bump function supported in $[-1,1]^2$ with $\int_{\R^2} \varphi(z)\dee z = 1$, and define the rescaled and periodized kernel 
	$$ \varphi_N(z) = \sum_{m \in \Z^2}N^2 \varphi(N(z+m)).$$
	For $f:\T^2 \to \R$, we let 
	\begin{equation} \label{eq:convolution}
		P_{\le N}f(z) = \int_{\T^2} \varphi_N(z-\bar{z}) f(\bar{z}) \, \dee \bar{z}
	\end{equation}
	denote the convolution of $f$ and the kernel $\varphi_N$. 
	
	\begin{theorem}[Projected exponential mixing] \label{prop:Projected}
		There exist constants $c, C > 0$ such that for all $\alpha$ sufficiently large, mean-zero $f,g \in W^{1,\infty}(\T^2)$, and $m,n \in \N$ we have
		\begin{equation}\label{eq:projectedmix}
			\left|\int_{\T^2} P_{\le N}(f \circ T_\alpha^{-n})(z) P_{\le N}(g\circ T_\alpha^{-m})(z) \, \dee z\right| \le C e^{-c|m-n|\log \alpha} \|f\|_{W^{1,\infty}}\|g\|_{W^{1,\infty}}
		\end{equation}
		for every $N \ge 1$.
	\end{theorem}
	
	The proof of Theorem~\ref{prop:Projected} is given in Section~\ref{sec:projected}. It relies on the anisotropic estimates used to prove Theorem~\ref{prop:UniformMixing} as a key ingredient, but also requires several additional ideas.
	
	\subsection{Proof of Theorem~\ref{thm:discrete}}
	
	Throughout this entire section, $f \in C^\infty(\T^2)$ is a fixed mean-zero function with $\|f\|_{L^2} = 1$ and 
	$$\rho_\infty = \sum_{n=0}^\infty f \circ T_\alpha^{-n}$$ 
	is as defined in \eqref{eq:rhoinftydef}. Here and in the remainder of the paper, we write $a \leqc b$ when $a \le Cb$ for a constant $C > 0$ that is independent of $\alpha$ for all $\alpha$ sufficiently large. We write $a \leqc_p b$ when we want to emphasize that the constant $C$ depends on some parameter $p$.
	
	Before proving Theorem~\ref{thm:discrete}, we establish three preliminary lemmas. We begin with uniform-in-$\alpha$ control on $\rho_\infty$ in negative regularity spaces, which is an easy consequence of Theorem~\ref{prop:UniformMixing}.
	
	\begin{lemma} \label{lem:negativecontrol}
		For $\alpha$ sufficiently large, the series defining $\rho_\infty$ converges absolutely in $H^{-s}$ for every $s > 0$ and we have
		\begin{equation} \label{eq:negativecontrol}
			\|\rho_\infty\|_{H^{-s}}\leqc_s 1.
		\end{equation}
		In particular, $\rho_\infty \in \bigcap_{s > 0} H^{-s}$.
	\end{lemma}
	
	\begin{proof}
		It suffices to prove that 
		\begin{equation} \label{eq:negativegoal}
			\sum_{n=0}^\infty \|f \circ T_\alpha^{-n}\|_{H^{-s}} \leqc_s 1
		\end{equation}
		for $0 < s< 3$. By \eqref{eq:Sobolevdecay}, for any $n \in \N$ and $\alpha$ sufficiently large there holds
		\begin{equation} \label{eq:Existence1}
			\|f \circ T_\alpha^{-n}\|_{H^{-3}} \leqc e^{-c n \log \alpha}\|f\|_{C^1} \leqc e^{-c n \log \alpha}.
		\end{equation}
		By interpolation, \eqref{eq:Existence1}, and the fact that composition with $T_\alpha^{-1}$ preserves the $L^2$ norm, we have
		\begin{equation}\label{eq:interpolation}
			\|f\circ T_\alpha^{-n}\|_{H^{-s}} \le \|f\circ T_{\alpha}^{-n}\|_{H^{-3}}^{\frac{s}{3}}\|f \circ T_\alpha^{-n}\|^{1-\frac{s}{3}}_{L^2} \leqc e^{-\frac{s}{3} c n \log \alpha}.
		\end{equation} 
		Thus,
		\begin{equation} \label{eq:Existence2} 
			\sum_{n=0}^\infty \|f\circ T_\alpha^{-n}\|_{H^{-s}} \leqc \sum_{n=0}^\infty e^{-\frac{s}{3} c n \log \alpha} \leqc_s 1,
		\end{equation}
		which proves \eqref{eq:negativegoal} and completes the proof.
	\end{proof}
	
	Next, we show that it is sufficient to prove the cumulative law with $\Pi_{\le N}$ replaced by the regularized version $P_{\le N}$. Below and in all that follows, for $g:\T^2 \to \R$ we write $\hat{g}: \Z^2 \to \C$ to denote its Fourier coefficients, defined for $k \in \Z^2$ by 
	$$\hat{g}(k) = \int_{\T^2} e^{-2\pi i z\cdot k} g(z) \, \dee z.$$  
	\begin{lemma} \label{lem:Kernel}
		Suppose that there exists $C_0 \ge 1$ such that for all $\alpha$ sufficiently large and $N \ge 2$ we have 
		\begin{equation} \label{eq:PBatchelor}
			-C_0 + \frac{1}{C_0}\frac{\log N}{\log \alpha} \le \|P_{\le N} \rho_\infty\|_{L^2}^2 \le C_0 \left(\frac{\log N}{\log \alpha}+1\right).
		\end{equation}
		Then, there exists $c_0 \ge 1$, depending only on $C_0$ and the kernel $\varphi$ used to define $P_{\le N}$ in \eqref{eq:convolution}, such that 
		\begin{equation} \label{eq:Kernelgoal}
			-c_0 + \frac{1}{c_0}\frac{\log N}{\log \alpha} \le \|\Pi_{\le N} \rho_\infty\|_{L^2}^2 \le c_0 \left(\frac{\log N}{\log \alpha}+1\right)
		\end{equation}
		for all $N \ge 2$. In particular, \eqref{eq:Batchelor} holds with $C = 2c_0$ and $N_\alpha = \alpha^{2c_0^2}$.
	\end{lemma}
	
	\begin{proof}
		We first prove that there is a constant $C_1 \ge 1$ depending only on $\varphi$ such that for all $g \in H^{-1}(\T^2)$ and $N \ge 1$ we have
		\begin{equation} \label{eq:Pequiv}
			\frac{1}{C_1}\|P_{\le \sqrt{N}} g\|^2_{L^2} - \frac{\|g\|^2_{H^{-1}}}{N} \le \|\Pi_{\le N} g\|^2_{L^2} \le C_1 \|P_{\le C_1 N} g\|^2_{L^2}. 
		\end{equation}
		For $\xi \in \R^2$, let 
		$$\hat{\varphi}(\xi) = \int_{\R^2} e^{-2\pi i z \cdot \xi} \varphi(z) \, \dee z$$
		denote the Fourier transform of $\varphi$. Then, the Fourier coefficients of the periodized and rescaled kernel are given as $\hat{\varphi}_N(k) = \hat{\varphi}(k/N)$. Since $\hat{\varphi}(0) = \int_{\R^2} \varphi(z) \dee z = 1$, it follows by the continuity of $\hat{\varphi}$ that there exists $\delta \in (0,1)$ such that $\hat{\varphi}_{N/\delta}(k) = \hat{\varphi}(\delta k/N) \ge 1/2$ for $|k| \le N$. Therefore,
		\begin{equation} 
			\|\Pi_{\le N} g\|_{L^2}^2 = \sum_{|k| \le N} |\hat{g}(k)|^2 \le 4 \sum_{|k| \le N} |\hat{\varphi}_{N/\delta}(k)|^2 |\hat{g}(k)|^2 \le 4 \|P_{\le N/\delta} g\|_{L^2}^2,
		\end{equation}
		which proves the upper bound in \eqref{eq:Pequiv}. For the lower bound, first note that since $\hat{\varphi}: \R^2 \to \C$ is Schwartz class, we have
		\begin{equation}
			|\hat{\varphi}(k/\sqrt{N})| \leqc \left(\frac{|k|}{\sqrt{N}}\right)^{-3},
		\end{equation}
		and thus
		\begin{align}
			\|P_{\le \sqrt{N}}g\|_{L^2}^2 & = \sum_{k \in \Z^2} |\hat{\varphi}_{\sqrt{N}}(k)|^2 |\hat{g}(k)|^2 
			= \sum_{k\in \Z^2} |\hat{\varphi}(k/\sqrt{N})|^2 |\hat{g}(k)|^2 \\ 
			& \le \|\Pi_{\le N} g\|_{L^2}^2 + N \sum_{ |k| > N} \left(\frac{|k|}{\sqrt{N}}\right)^2 |\hat{\varphi}(k/\sqrt{N})|^2 |k|^{-2}|\hat{g}(k)|^2 \\
			& \le \|\Pi_{\le N} g\|_{L^2}^2 +  \|g\|_{H^{-1}}^2 \sup_{|k| > N} N\left(\frac{|k|}{\sqrt{N}}\right)^2 |\hat{\varphi}(k/\sqrt{N})|^2 \\ 
			& \leqc \|\Pi_{\le N} g\|_{L^2}^2 +
			N^{-1}\|g\|_{H^{-1}}^2.
		\end{align}
		This implies the desired lower bound and hence completes the proof of \eqref{eq:Pequiv}.
		
		We now turn to the proof of \eqref{eq:Kernelgoal}. 
		By Lemma~\ref{lem:negativecontrol}, $\|\rho_\infty\|_{H^{-1}}\le C_2$ for some $0 < C_2 < \infty$ independent of $\alpha$. Applying \eqref{eq:Pequiv} with $g = \rho_\infty$ gives 
		$$ \frac{1}{C_1} \|P_{\le \sqrt{N}}\rho_\infty\|_{L^2}^2 - \frac{C_2}{N} \le \|\Pi_{\le N} \rho_\infty\|_{L^2}^2 \le C_1 \|P_{\le C_1 N} \rho_\infty\|_{L^2}^2. $$
		Using \eqref{eq:PBatchelor} then yields 
		\begin{equation}
			- \frac{C_2}{N} - \frac{C_0}{C_1} + \frac{1}{2C_0 C_1} \frac{\log N}{\log \alpha} \le \|\Pi_{\le N} \rho_\infty\|_{L^2}^2 \le C_0 C_1 \left(\frac{\log N}{\log \alpha} + \frac{\log C_1}{\log \alpha} + 1 \right),
		\end{equation}
		which gives \eqref{eq:Kernelgoal} and completes the proof.
	\end{proof}
	
	Lastly, it will be convenient to establish a lemma that estimates the off-diagonal terms of the infinite series defining $\|P_{\le N} \rho_\infty\|_{L^2}^2$.
	
	\begin{lemma} \label{lem:offdiagonal}
		There exists $C \ge 1$ such that for every $N \ge 2$ and $\alpha$ sufficiently large we have 
		$$ \left|\sum_{n=1}^\infty \sum_{m=0}^{n-1} \int_{\T^2} P_{\le N}(f \circ T^{-n})(z) P_{\le N} (f\circ T^{-m})(z) \, \dee z\right| \le \frac{ C\log N}{\alpha^c}, $$
		where $c > 0$ is as in Theorem~\ref{prop:UniformMixing}.
	\end{lemma}
	
	\begin{proof}
		Note that $\|P_{\le N}\|_{H^{-3} \to L^2} \leqc N^3$ since $\hat{\varphi}$ is Schwartz class.
		Thus, by \eqref{eq:Sobolevdecay} we have 
		\begin{align} \left|\int_{\T^2} P_{\le N}(f \circ T_\alpha^{-n})(z) P_{\le N}(f \circ T_\alpha^{-m})(z) \dee z\right| &\le \|P_{\le N} (f \circ T_\alpha^{-n})\|_{L^2} \|P_{\le N} (f \circ T_\alpha^{-m})\|_{L^2} \\ 
			& \leqc N^6 \|f \circ T_\alpha^{-n}\|_{H^{-3}}\|f \circ T_\alpha^{-m}\|_{H^{-3}}\\ 
			& \leqc  N^6 e^{-c(m+n)\log\alpha}. \label{eq:diagonal1}
		\end{align}
		Combining \eqref{eq:diagonal1} with Theorem~\ref{prop:Projected}, we obtain 
		\begin{equation}\label{eq:diagonal2} \left|\int_{\T^2} P_{\le N}(f \circ T_\alpha^{-n})(z) P_{\le N}(f \circ T_\alpha^{-m})(z) \, \dee z\right|\leqc \min\left(e^{-c|n-m|\log\alpha},N^6 e^{-c(m+n)\log\alpha}\right). 
		\end{equation}
		Define
		\begin{equation} \label{eq:nstar}
			n_* = \frac{3 \log N}{c \log \alpha} 
		\end{equation}
		and let $n_\alpha$ denote the first natural number greater than or equal to $n_*$. Using the first bound in \eqref{eq:diagonal2} for $m \le n_\alpha$ and the second for $m > n_\alpha$, we have 
		\begin{align} &\left|\sum_{n=1}^\infty \sum_{m=0}^{n-1} \int_{\T^2} P_{\le N}(f \circ T^{-n})(z) P_{\le N} (f\circ T^{-m})(z) \, \dee z\right| \leqc  \sum_{n=1}^\infty \sum_{m=0}^{n-1}\min\left(e^{-c|n-m|\log\alpha},N^6 e^{-c(m+n)\log\alpha}\right) \\ 
			& \quad \le \sum_{n=1}^{n_\alpha} \sum_{m=0}^{n-1} e^{-c(n-m)\log \alpha} + \sum_{n=n_\alpha + 1}^\infty \sum_{m=0}^{n_\alpha} e^{-c(n-m)\log \alpha} + N^6 \sum_{n=n_\alpha + 2}^\infty \sum_{m=n_\alpha + 1}^{n-1} e^{-c(m+n)\log \alpha}. \label{eq:diagonal3}
		\end{align}
		For the three sums in \eqref{eq:diagonal3}, assuming $\alpha$ is large enough so that $\alpha^c \ge 4$, we have the estimates
		$$ \sum_{n=1}^{n_\alpha} \sum_{m=0}^{n-1} e^{-c(n-m)\log \alpha} = \sum_{n=1}^{n_\alpha} \alpha^{-cn} \sum_{m=0}^{n-1} \alpha^{cm} = \frac{1}{\alpha^c - 1}\sum_{n=1}^{n_\alpha} \alpha^{-cn} (\alpha^{cn}-1)\leqc \frac{n_\alpha}{\alpha^c} \leqc \frac{\log N}{\alpha^c},  $$
		$$ \sum_{n=n_\alpha + 1}^\infty e^{-cn \log \alpha} \sum_{m=0}^{n_\alpha} e^{c m \log \alpha} = \left(\frac{\alpha^{-c(n_\alpha+1)}}{1-\alpha^{-c}}\right) \left(\frac{\alpha^{c(n_\alpha+1)}-1}{\alpha^c-1}\right) = \frac{1-\alpha^{-c(n_\alpha+1)}}{\alpha^c + \alpha^{-c} -2} \leqc \alpha^{-c},$$
		and
		$$ N^6 \sum_{n=n_\alpha + 2}^\infty e^{-cn \log \alpha} \sum_{m=n_\alpha+1}^{n-1} e^{-mc \log \alpha} \le N^6 \left(\frac{\alpha^{-c(n_\alpha+2)}}{1-\alpha^{-c}}\right) \left(\frac{\alpha^{-c(n_\alpha+1)}}{1-\alpha^{-c}}\right) \leqc N^6 \alpha^{-3c} \alpha^{-2n_\alpha c} \le \alpha^{-3c}.  $$
		In the final inequality above we noted that the definition of $n_*$ is such that 
		\begin{equation} \label{eq:nstarprop}
			\alpha^{-2n_\alpha c} \le \alpha^{-2n_* c} = N^{-6}.
		\end{equation}
		Putting the estimates of the three sums above into \eqref{eq:diagonal3} completes the proof.
	\end{proof}
	
	With Lemmas~\ref{lem:negativecontrol}-\ref{lem:offdiagonal} in hand, we are now ready to prove Theorem~\ref{thm:discrete}. 
	\begin{proof}[Proof of Theorem~\ref{thm:discrete} (1) and (2)] 
		That $\rho_\infty \in \bigcap_{s > 0} H^{-s}$ is a consequence of Lemma~\ref{lem:negativecontrol}, while $\rho_\infty \not \in L^2$ follows from part (3) of Theorem~\ref{thm:discrete}, which is proven independently below. 
		
		To show that $\mathcal{K}_\alpha(\rho_\infty) = \rho_\infty$ in $H^{-1}$, we must show that for all test functions $\psi \in H^1$ there holds 
		\begin{equation} \label{eq:dualfixed}
			\langle \rho_\infty, \psi \rangle_{H^{-1} \times H^1} = \langle \rho_\infty, \psi \circ T_\alpha \rangle_{H^{-1} \times H^1} + \langle f, \psi \rangle_{L^2},
		\end{equation}
		where $\langle \cdot, \cdot \rangle_{H^{-1} \times H^1}$ denotes the dual pairing between $H^{-1}$ and $H^1$. Let 
		$$S_m = \sum_{n=0}^m f \circ T_\alpha^{-n}.$$
		Then, $S_m \in L^2$ for every $m \in \N$ and $\lim_{m \to \infty} S_m = \rho_\infty$ in $H^{-1}$ by Lemma~\ref{lem:negativecontrol}. Consequently,   
		\begin{align*}
			\langle \rho_\infty, \psi \circ T_\alpha \rangle_{H^{-1} \times H^1}  = \lim_{m \to \infty} \langle S_m, \psi \circ T_\alpha \rangle_{L^2} & = -\langle f , \psi \rangle_{L^2} + \lim_{m \to \infty} \langle S_{m+1}, \psi\rangle_{L^2} \\ 
			& = - \langle f, \psi\rangle_{L^2} + \langle \rho_\infty, \psi \rangle_{H^{-1} \times H^1},
		\end{align*}  
		which is \eqref{eq:dualfixed}.
		
		It remains to prove that $\lim_{n \to \infty} \| \mathcal{K}_\alpha^n (\rho_0) - \rho_\infty\|_{H^{-s}} = 0$ for every mean-zero $\rho_0 \in C^1(\T^2)$ and $s > 0$. Iterating \eqref{eq:ForcedMap}, we have 
		$$\mathcal{K}_\alpha^n(\rho_0) = \rho_0 \circ T_\alpha^{-n} + \sum_{k=0}^{n-1} f \circ T_\alpha^{-k}.$$
		Thus,
		$$\|\mathcal{K}_\alpha^n(\rho_0) - \rho_{\infty}\|_{H^{-s}} \le \|\rho_0 \circ T_\alpha^{-n}\|_{H^{-s}} + \sum_{k=n}^\infty \|f \circ T_\alpha^{-k}\|_{H^{-s}}.$$
		As $n \to \infty$, the first term above tends to zero by Theorem~\ref{prop:UniformMixing} and the interpolation argument in \eqref{eq:interpolation}, while the second term tends to zero by Lemma~\ref{lem:negativecontrol}.
	\end{proof}
	
	\begin{proof} [Proof of Theorem~\ref{thm:discrete} (3)] 
		By Lemma~\ref{lem:Kernel}, it is sufficient to show that there exists $C_0 \ge 1$ such that \eqref{eq:PBatchelor} holds for all $\alpha$ sufficiently large. From \eqref{eq:rhoinftydef}, we have 
		\begin{equation} \label{eq:Batchelor1}
			\|P_{\le N} \rho_\infty\|_{L^2}^2 = \sum_{n=0}^\infty \|P_{\le N}(f \circ T_\alpha^{-n})\|_{L^2}^2 + 2 \sum_{n=1}^\infty \sum_{m=0}^{n-1} \int_{\T^2}P_{\le N}(f\circ T^{-n})(z)P_{\le N}(f \circ T^{-m})(z) \, \dee z.
		\end{equation}
		To bound this expression from above, first observe that the estimate $\|P_{\le N}(f \circ T_\alpha^{-n})\|_{L^2} \leqc N^3 e^{-nc \log \alpha}$ from  \eqref{eq:diagonal1}, together with $\|P_{\le N}\|_{L^2 \to L^2} \le 1$, implies
		$$ \|P_{\le N} (f \circ T_\alpha^{-n})\|_{L^2}^2 \leqc \min\left(1,N^6 e^{-2c n \log\alpha}\right). $$  
		Therefore, for $n_*$ as in \eqref{eq:nstar} we have
		\begin{align}
			\sum_{n=0}^\infty \|P_{\le N} (f \circ T_\alpha^{-n})\|_{L^2}^2 & \leqc \sum_{n=0}^\infty \min\left(1,N^6 e^{-2c n \log\alpha}\right)  
			\le \sum_{n \le n_*} 1 + N^6 \sum_{n > n_*} e^{-2cn \log \alpha} \\ 
			& \leqc n_* + N^6 \alpha^{-2 c n_*}  \leqc 1 + \frac{\log N}{\log \alpha}, \label{eq:Batchelor2}
		\end{align}
		where in the last inequality we used \eqref{eq:nstarprop}. Putting \eqref{eq:Batchelor2} into \eqref{eq:Batchelor1} and applying Lemma~\ref{lem:offdiagonal} gives
		\begin{equation}
			\|P_{\le N} \rho_\infty\|_{L^2}^2 \leqc 1 + \frac{\log N}{\log \alpha} + \frac{\log N}{\alpha^c} \leqc 1+\frac{\log N}{\log \alpha},
		\end{equation}
		which is the upper bound in \eqref{eq:PBatchelor}.
		
		We now turn to the lower bound. Let $\delta \in (0,1)$ such that $|1- |\hat{\varphi}(\xi)|^2| \le 1/2$ for all $|\xi| \le \delta$. Then, 
		\begin{align}
			\|P_{\le N}(f \circ T_\alpha^{-n})\|_{L^2}^2 & = \|f\circ T_\alpha^{-n}\|_{L^2}^2 + \sum_{k \in \Z^2}\left(|\hat{\varphi}(k/N)|^2 - 1\right)|\widehat{f\circ T_\alpha^{-n}}(k)|^2 \\ 
			& \ge \frac{1}{2} - \|\Pi_{> \delta N}(f\circ T_\alpha^{-n})\|_{L^2}^2, \label{eq:Batchelor3}
		\end{align}
		where in the inequality we used that $\|f \circ T_\alpha^{-n}\|_{L^2} = \|f\|_{L^2} = 1$. Since $\|DT_\alpha^{-1}\|_{L^\infty} \le 4\alpha^2$ (see Section~\ref{sec:hyperbolic}, and in particular \eqref{eq:matrices}), we have
		$$\|\grad (f\circ T_\alpha^{-n})\|_{L^2} \le (2\alpha)^{2n}\|\grad f\|_{L^2},$$
		and hence
		\begin{equation} \label{eq:Batchelor3.2}
			\|\Pi_{> \delta N} (f \circ T_\alpha^{-n})\|^2 \le (\delta N)^{-2}\|\grad (f \circ T_\alpha^{-n})\|_{L^2}^2 \leqc N^{-2} (2 \alpha)^{4n}. 
		\end{equation}
		It follows from \eqref{eq:Batchelor3} and \eqref{eq:Batchelor3.2} that there is a constant $C_1 \ge 1$ independent of $\alpha$ such that for all $\alpha \gg 1$ we have
		\begin{equation}\label{eq:Batchelor4}
			\|P_{\le N}(f \circ T_\alpha^{-n})\|_{L^2}^2 \ge \frac{1}{4} \qquad \forall n \le \frac{\log N}{4\log \alpha} - C_1.
		\end{equation}
		Using \eqref{eq:Batchelor4} in \eqref{eq:Batchelor1} and appealing to Lemma~\ref{lem:offdiagonal} gives
		\begin{align} 
			\sum_{n=0}^\infty \|P_{\le N}\rho_\infty\|_{L^2}^2 & \ge - \frac{2C \log N}{\alpha^c} + \sum_{n=0}^\infty \|P_{\le N}(f\circ T_\alpha^{-n})\|_{L^2}^2 \\ 
			& \ge - \frac{2C \log N}{\alpha^c} + \sum_{0 \le n \le \frac{\log N}{4\log \alpha} - C_1} \frac{1}{4} 
			\ge - \frac{2C \log N}{\alpha^c} + \frac{\log N}{16 \log \alpha} - \frac{C_1}{4}.
		\end{align}
		When $\alpha$ is large, the first term is absorbed by the second, which gives the lower bound in \eqref{eq:PBatchelor} and completes the proof. 
	\end{proof}
	
	\subsection{Proofs of Theorems~\ref{thm:continuous0},~\ref{thm:continuous}, and~\ref{thm:flux}} 
	
	In this section, we prove Theorems~\ref{thm:continuous0}-\ref{thm:flux} concerning the forced transport equation \eqref{eq:transport}. Throughout, $F:[0,\infty) \times \T^2 \to \R$ denotes a smooth, time-one periodic function satisfying \eqref{eq:meanzero} and we write $F_t(\cdot) = F(t,\cdot)$. Recall that $\Phi_{s,t}^\alpha:\T^2 \to \T^2$ denotes the flow map of $u_\alpha$ and we let $\phi_{t,s}^\alpha = (\Phi_{s,t}^\alpha)^{-1}$, so that the solution of \eqref{eq:transport} is given by
	\begin{equation} \label{eq:gensolution}
		\rho(t) = \rho_0 \circ \phi_{t,0}^{\alpha} + \int_0^t F_s \circ \phi_{t,s}^\alpha \, \dee s.
	\end{equation}
	The ``effective'' forcing $f_\alpha \in W^{1,\infty}(\T^2)$ is as defined in \eqref{eq:feff}. Note that since composition with $\phi_{t,s}^\alpha$ preserves the spatial average, $f_\alpha$ has zero mean by \eqref{eq:meanzero}.
	
	\subsubsection{Proof of Theorems~\ref{thm:continuous0} and~\ref{thm:continuous}} \label{sec:continuousproof}
	
	We first establish a generalization of \eqref{eq:effectivediscrete} that will be needed to leverage Theorem~\ref{thm:discrete} in the continuous-time setting.
	
	\begin{lemma} \label{lem:discretereduction}
		Let $\rho_0 \in W^{1,\infty}$. Then, for any $n \in \N$ and $\tau \in [0,1)$ the solution \eqref{eq:gensolution} can be written as
		\begin{equation}\label{eq:effectivediscrete2}
			\rho(n+\tau) = \rho_0 \circ T_\alpha^{-n} \circ \phi_{\tau,0}^\alpha + \int_0^\tau F_s \circ \phi_{\tau,s}^\alpha \, \dee s + \sum_{k=0}^{n-1} f_\alpha \circ T_\alpha^{-k} \circ \phi_{\tau,0}^\alpha.
		\end{equation}
	\end{lemma}
	
	\begin{proof} 
		Since $u_\alpha$ is time-one periodic, we have $\phi_{n+\tau,0}^\alpha = T_\alpha^{-n} \circ \phi_{\tau,0}^\alpha$, and so to complete the proof it suffices to show that 
		\begin{equation}\label{eq:discretegoal}
			\int_0^{n+\tau} F_s \circ \phi_{n+\tau,s}^\alpha \, \dee s = \int_0^\tau F_s \circ \phi_{\tau,s}^\alpha \, \dee s + \sum_{k=0}^{n-1} f_\alpha \circ T_\alpha^{-k} \circ \phi_{\tau,0}^\alpha.
		\end{equation}
		We have
		\begin{align} \label{eq:discretegoal1}
			\int_0^{n+\tau} F_s \circ \phi_{n+\tau,s}^\alpha \, \dee s = \sum_{k=0}^{n-1} \int_{k}^{k+1} F_s \circ \phi_{n+\tau,s}^\alpha \, \dee s + \int_n^{n+\tau} F_s \circ \phi_{n+\tau,s}^\alpha \, \dee s.
		\end{align}
		By the co-cycle property $\phi_{t_2,t_1}^\alpha = \phi^\alpha_{r,t_1} \circ \phi_{t_2,r}^\alpha$ for $t_1 < r < t_2$ and the periodicity of $u_\alpha$, we have 
		$$ \phi_{n+\tau,s}^\alpha = \phi_{k+1,s}^\alpha \circ \phi_{n,k+1}^\alpha \circ \phi_{n+\tau,n} = \phi_{k+1,s}^\alpha \circ T_{\alpha}^{-(n-k-1)} \circ \phi_{\tau,0}^\alpha. $$
		Using also the periodicity of $F$, each term of the summation in \eqref{eq:discretegoal1} can then be re-written as 
		\begin{equation}\label{eq:discretegoal2}
			\int_{k}^{k+1} F_s \circ \phi_{n+\tau,s}^\alpha \, \dee s = \left(\int_k^{k+1} F_s \circ \phi_{k+1,s}^\alpha\, \dee s\right)\circ T_\alpha^{-(n-k-1)} \circ \phi_{\tau,0}^\alpha = f_\alpha \circ T_\alpha^{-(n-k-1)} \circ \phi_{\tau,0}^\alpha. 
		\end{equation}
		Similarly,
		\begin{equation} \label{eq:discretegoal3}
			\int_n^{n+\tau} F_s \circ \phi_{n+\tau,s}^\alpha \, \dee s = \int_0^\tau F_s \circ \phi_{\tau,s}^\alpha \, \dee s.
		\end{equation}
		Putting \eqref{eq:discretegoal2} and \eqref{eq:discretegoal3} into \eqref{eq:discretegoal1} and relabeling the summation completes the proof of \eqref{eq:discretegoal}.
	\end{proof}
	
	In view of Lemma~\ref{lem:discretereduction}, the natural limiting object to consider is 
	\begin{equation} \label{eq:rhoinftycontinuous}
		\rho_\infty(t) = \int_0^t F_s \circ \phi_{t,s}^\alpha \, \dee s + \sum_{n = 0}^\infty f_\alpha \circ T_\alpha^{-n} \circ \phi_{t,0}^\alpha,
	\end{equation}
	and indeed $\rho_\infty$ in the statement of Theorem~\ref{thm:continuous0} will be as defined above. In order to estimate the summation in \eqref{eq:rhoinftycontinuous}, we will need
	the following lemma, which is essentially an extension of \eqref{eq:interpolation} to continuous time.
	
	\begin{lemma}\label{lem:continuouscorrelation}
		Let $\alpha$ be sufficiently large. For every mean-zero $f \in W^{1,\infty}$, $s \in (0,1)$, $n \in \N$, and $t \ge 0$ we have
		\begin{equation} \label{eq:continuouscorrelation}
			\|f \circ T_\alpha^{-n} \circ \phi_{t,0}^\alpha\|_{H^{-s}} \leqc \alpha^{2s} e^{-\frac{s}{3}c(n+\lfloor t \rfloor)\log \alpha} \|f\|_{W^{1,\infty}}.
		\end{equation}
	\end{lemma}
	
	\begin{proof}
		Since $u_\alpha$ is time-one periodic, we have 
		\begin{equation} \label{eq:floort}
			T_\alpha^{-n} \circ \phi^\alpha_{t,0} = T_\alpha^{-(n+\lfloor t \rfloor)}\circ \phi^\alpha_{\tau, 0},
		\end{equation}
		where $\tau = t -  \lfloor t \rfloor \le 1$. 
		The same argument used to prove \eqref{eq:interpolation} gives
		\begin{equation}\label{eq:interpolation2}
			\|f \circ T_\alpha^{-(n+\lfloor t \rfloor)}\|_{H^{-s}} \leqc e^{-\frac{s}{3}c(n + \lfloor t \rfloor )\log \alpha}\|f\|_{W^{1,\infty}},
		\end{equation}
		and hence by \eqref{eq:floort} we have
		\begin{align}
			\|f \circ T_\alpha^{-n} \circ \phi_{t,0}^\alpha\|_{H^{-s}} & = \sup_{g \in H^s, \|g\|_{H^s}\le 1} \langle f \circ T_\alpha^{-(n+\lfloor t \rfloor)}, g \circ \Phi_{0,\tau}^\alpha \rangle_{L^2} \\ 
			& \leqc e^{-\frac{s}{3}c(n + \lfloor t \rfloor )\log \alpha}\|f\|_{W^{1,\infty}} \sup_{g \in H^s, \|g\|_{H^s} \le 1} \|g \circ \Phi_{0,\tau}^\alpha\|_{H^s}.
		\end{align}
		Therefore, to complete the proof, it suffices to show that 
		\begin{equation}\label{eq:dualitygoal}
			\sup_{g \in H^s, \|g\|_{H^s} \le 1} \|g \circ \Phi_{0,\tau}^\alpha\|_{H^s} \leqc \alpha^{2s}.
		\end{equation}
		Since $\tau \le 1$, we have $\|D \Phi_{0,\tau}^\alpha\|_{L^\infty} \leqc \alpha^2$. Thus, for general $h \in H^1$ there holds
		$$\|\grad (h \circ \Phi_{0,\tau}^\alpha)\|_{L^2}\le \alpha^2 \|h\|_{H^1} \quad \text{and} \quad \|(h \circ \Phi_{0,\tau}^\alpha)\|_{L^2} = \|h\|_{L^2}.$$
		The bound \eqref{eq:dualitygoal} then follows from Sobolev interpolation (see e.g. \cite[Lemma A.1]{EL}).
	\end{proof}
	
	The final preliminary ingredient we require is a tool to show that the upper bound portion of Batchelor's law is propagated under the flow of \eqref{eq:transport}. This is probably not strictly necessary, but it allows us to treat the continuous time setting without proving a version of Theorem~\ref{prop:Projected} that includes an additional composition with $\phi_{t,0}^\alpha$.
	
	\begin{lemma}\label{lem:Fourier}
		For any $0 \le s < t \le 1$, $f \in H^{-1/8}$ and $N \gg \alpha$ there holds
		\begin{equation}
			\|\Pi_{\le N} (f \circ \phi_{t,s}^\alpha)\|_{L^2} \leqc \|\Pi_{\le N^{100}} f\|_{L^2} + \alpha^{1/8}N^{-3/2}\|f\|_{H^{-1/8}}.
		\end{equation}
	\end{lemma}
	
	\begin{proof}
		By density we may assume without loss of generality that $f$ is smooth. Regardless of $0 \le s < t \le 1$ we can write $\phi_{t,s}^\alpha = \phi_1 \circ \phi_2$, where 
		$$ \phi_1(x,y) = (x,y+c_1|x-1/2|) \quad \text{and} \quad \phi_2(x,y) = (x+c_2|y-1/2|,y) $$
		for some constants $-\alpha \le c_1, c_2 \le 0$. We will control the effects of each composition separately and then piece together the resulting estimates. Let $g = f \circ \phi_1$, so that $f \circ \phi_{t,s}^\alpha = g \circ \phi_2$. To bound $\Pi_{\le N} (g \circ \phi_2)$, we begin by writing
		$$ \Pi_{\le N} (g \circ \phi_2) = \Pi_{\le N}(\Pi_{\le N^{10}}g \circ \phi_2) + \Pi_{\le N}(\Pi_{>N^{10}}g \circ \phi)$$
		and using $\|\Pi_{\le N^{10}} g \circ \phi_2\|_{L^2} = \|\Pi_{\le N^{10}} g\|_{L^2}$ to obtain
		\begin{equation} \label{eq:Fourier1}
			\|\Pi_{\le N} (g \circ \phi_2)\|_{L^2} \le \|\Pi_{\le N^{10}} g\|_{L^2} + \|\Pi_{\le N}(\Pi_{>N^{10}} g \circ \phi)\|_{L^2}.
		\end{equation}
		Towards estimating the second term on the right-hand side of \eqref{eq:Fourier1}, for $\ell = (\ell_1,\ell_2) \in \Z^2$ we compute 
		\begin{align}
			(\Pi_{> N^{10}}g \circ \phi_2)^\wedge(\ell) &= \sum_{|k| > N^{10}} \hat{g}(k) \int_{\T^2} e^{2\pi i (k \cdot \phi_2(z) - \ell \cdot z)} \, \dee z  \\ 
			& = \sum_{|k| > N^{10}} \hat{g}(k) \, \mathbf{1}_{k_1 = \ell_1}  \int_{\T} e^{2 \pi i (\ell_1 c_2 |y-1/2|+k_2 y - \ell_2 y)} \, \dee y \\ 
			& = \sum_{\{k_2: |(\ell_1,k_2)| > N^{10}\}} \hat{g}(k) \int_{\T} e^{2 \pi i (\ell_1 c_2 |y-1/2|+k_2 y - \ell_2 y)} \, \dee y  \label{eq:Fourier2}
		\end{align}
		Evaluating the integral above explicitly, we find
		\begin{equation} \label{eq:Fourier3}
			\left|\int_{\T} e^{2 \pi i (\ell_1 c_2 |y-1/2|+k_2 y - \ell_2 y)} \, \dee y\right| \leqc  \frac{1}{|k_2 - \ell_2 - c_2 \ell_1|} + \frac{1}{|k_2 - \ell_2 + c_2\ell_1|}.
		\end{equation}
		When $|\ell| \le N$ and $N \gg \alpha$, the summation in \eqref{eq:Fourier2} is restricted to a region where $|k_2 - \ell_2 \pm c_2 \ell_1| \gtrsim 1+|k_2| \gtrsim N^{10}$. Therefore, putting \eqref{eq:Fourier3} into \eqref{eq:Fourier2}, for $|\ell| \le N$ we have
		\begin{align}
			\left|(\Pi_{> N^{10}}g \circ \phi_2 )^\wedge (\ell)\right| & \leqc \sum_{\{k_2: |(\ell_1,k_2)| > N^{10}\}} \frac{|\hat{g}(k)|}{1+|k_2|} \\ 
			& \le \left(\sum_{|k_2| \gtrsim N^{10}} \frac{|\hat{g}(k)|^2}{(1+|k_2|)^{1/2}}\right)^{1/2}\left(\sum_{k_2 \in \Z} \frac{1}{(1+|k_2|)^{3/2}}\right)^{1/2} \\ 
			& \leqc N^{-5/2}\|g\|_{H^{-1/8}}, \label{eq:Fourier4}
		\end{align}
		where in the final inequality we noted that $(1+|k_2|)^{1/2} \gtrsim N^{-5/2}(1+|k_2|)^{1/4}$. It follows from \eqref{eq:Fourier4} and \eqref{eq:Fourier1} that
		\begin{equation} \label{eq:Fourier5}
			\|\Pi_{\le N}(f \circ \phi_{t,s}^\alpha)\|_{L^2} = \|\Pi_{\le N}(g \circ \phi_2)\|_{L^2} \leqc \|\Pi_{\le N^{10}} g\|_{L^2} + N^{-3/2}\|g\|_{H^{-1/8}}.
		\end{equation} 
		The same computations used to prove \eqref{eq:Fourier5} with $N$ replaced by $N^{10}$ show that 
		\begin{equation} \label{eq:Fourier6}
			\|\Pi_{\le N^{10}} g\|_{L^2} = \|\Pi_{\le N^{10}}(f \circ \phi_1)\|_{L^2} \leqc \|\Pi_{\le N^{100}} f\|_{L^2} + N^{-15}\|f\|_{H^{-1/8}}.
		\end{equation}
		Combining \eqref{eq:Fourier6} with \eqref{eq:Fourier5} completes the proof after noting that $\|g\|_{H^{-1/8}} \leqc \alpha^{1/8}\|f\|_{H^{-1/8}}$ due to the same argument that gave \eqref{eq:dualitygoal}.
	\end{proof}
	
	We are now ready to prove Theorems~\ref{thm:continuous0} and~\ref{thm:continuous}.
	
	\begin{proof}[Proof of Theorem~\ref{thm:continuous0}]
		\hfill \\
		\vspace{-0.3cm}
		
		\textbf{Step 1} ($\rho_\infty$ is well defined): Let $\rho_\infty$ be given by \eqref{eq:rhoinftycontinuous}. By Lemma~\ref{lem:continuouscorrelation}, for any $s \in (0,1)$ we have 
		\begin{equation} \label{eq:continuousabsolute}
			\sum_{n=0}^\infty \sup_{t \ge 0}\|f_\alpha\circ T_\alpha^{-n} \circ \phi_{t,0}^\alpha\|_{H^{-s}} \leqc \alpha^{2s}\|f_\alpha\|_{W^{1,\infty}}\sum_{n=0}^\infty e^{-\frac{s}{3}cn\log\alpha} \leqc_s \alpha^{2s} \|f_\alpha\|_{W^{1,\infty}}.
		\end{equation}
		Thus, the series in \eqref{eq:rhoinftycontinuous} converges absolutely in $L^\infty([0,\infty);H^{-s})$. Combined with the simple estimate
		\begin{equation} \label{eq:L2trivial}
			\left\|\int_0^t F_s \circ \phi_{t,s}^\alpha \, \dee s\right\|_{L^2} \le \int_0^t\|F_s \circ \phi_{t,s}^\alpha\|_{L^2} \, \dee s \leqc t,
		\end{equation}
		it follows that $\rho_\infty \in L^\infty_\loc([0,\infty);H^{-s})$ for every $s > 0$ with the bound
		\begin{equation} \label{eq:continuousnegative}
			\sup_{0 \le t \le t_0}\|\rho_\infty(t)\|_{H^{-s}}  \leqc_{s} t_0 +\alpha^{2s} \|f_\alpha\|_{W^{1,\infty}}.
		\end{equation}
		
		\textbf{Step 2} ($\rho_\infty$ is a time-one periodic solution of \eqref{eq:transport}): Let
		$$\rho_\infty^{(m)}(t):=\int_0^t F_s \circ \phi_{t,s}^\alpha \, \dee s + \sum_{n=0}^m f_\alpha\circ T_\alpha^{-n} \circ \phi_{t,0}^\alpha$$
		denote the unique solution of \eqref{eq:transport} with initial data $\sum_{n=0}^m f_\alpha \circ T_\alpha^{-n}$. By Step 1, we have $\lim_{m\to \infty}\rho_\infty^{(m)} = \rho_\infty$ in $L^\infty_{\loc}([0,\infty);H^{-s})$, and hence we may pass to the limit in the distributional formulation of \eqref{eq:transport} for $\rho_\infty^{(m)}$ to see that  $\rho_\infty$ is also a solution. Regarding the periodicity of $\rho_\infty$, it is straightforward to check using Lemma~\ref{lem:discretereduction} that for every $k \in \N$ and $\tau \in [0,1)$ we have $\rho_\infty^{(m)}(k+\tau) = \rho_\infty^{(m+k)}(\tau)$. Taking $m \to \infty$ in this equality shows that $\rho_\infty(k+\tau) = \rho_\infty(\tau)$ in $H^{-s}$, which implies $\rho_\infty$ is time-one periodic.
		
		\textbf{Step 3} ($\rho_\infty$ attracts smooth data): Fix a mean-zero initial data $\rho_0 \in C^1$ and let $\rho$ denote the associated solution of \eqref{eq:transport}. Since $\rho_\infty$ is time-one periodic, to prove that $\lim_{t \to \infty}\|\rho(t) - \rho_\infty(t)\|_{H^{-s}} = 0$, it is sufficent so show that 
		\begin{equation}\label{eq:convergencegoal}
			\lim_{n \to \infty} \sup_{\tau \in [0,1)}\|\rho(n+\tau) - \rho_\infty(\tau)\|_{H^{-s}} = 0.
		\end{equation}
		By Lemma~\ref{lem:discretereduction}, we have 
		$$\rho(n+\tau) - \rho_\infty(\tau) = \rho_0 \circ T_\alpha^{-n} \circ \phi_{\tau,0}^\alpha + \sum_{k=n}^\infty f_\alpha \circ T_\alpha^{-k} \circ \phi_{\tau,0}^\alpha,$$
		and \eqref{eq:convergencegoal} then follows immediately from Lemma~\ref{eq:continuouscorrelation} and \eqref{eq:continuousabsolute}.
		
		\textbf{Step 4} (Cumulative upper bound): We now prove \eqref{eq:continuousUpper}. Since $\rho_\infty$ is time-one periodic, it is sufficient to show that there exist $N_\alpha, C_\alpha \ge 1$ such that 
		\begin{equation} \label{eq:continuousUpper2}
			\|\Pi_{\le N} \rho_\infty(t)\|_{L^2}^2 \le C_\alpha \log N
		\end{equation}
		for all $t \in [0,1]$ and $N \ge N_\alpha$. Let $\tilde{f}_\alpha = \|f_\alpha\|_{W^{1,\infty}}^{-1}f_\alpha$ and define
		$$\tilde{\rho}_\infty = \sum_{n=0}^\infty \tilde{f}_\alpha \circ T_\alpha^{-n}.$$
		The proof of the upper bound portion of Theorem~\ref{thm:discrete} relied only on the forcing $f$ having zero mean and satisfying $\|f\|_{W^{1,\infty}} \leqc 1$. Both of these conditions are satisfied for $\tilde{f}_\alpha$, and so by Theorem~\ref{thm:discrete} there exists $N_\alpha \ge 1$ such that
		\begin{equation} \label{eq:rhotildeupper}
			\|\Pi_{\le N} \rho_\infty(0)\|_{L^2}^2 = \|f_\alpha\|_{W^{1,\infty}}^2	\|\Pi_{\le N} \tilde{\rho}_\infty\|_{L^2}^2 \leqc \|f_\alpha\|_{W^{1,\infty}}^2\frac{\log N}{\log \alpha}
		\end{equation}
		for all $N \ge N_\alpha$. Since 
		$$\Pi_{\le N} \rho_\infty(t) = \Pi_{\le N} \int_0^t F_s \circ \phi_{t,s}^\alpha \, \dee s + \Pi_{\le N} (\rho_\infty(0) \circ \phi_{t,0}^\alpha),$$
		by Lemma~\ref{lem:Fourier}, \eqref{eq:rhotildeupper}, and \eqref{eq:continuousnegative}, for all $t \in [0,1]$ and $N \ge N_\alpha$ we have
		\begin{align}
			\|\Pi_{\le N} \rho_\infty(t)\|_{L^2}^2
			& \leqc 1 + \|\Pi_{N^{100}} \rho_\infty(0)\|_{L^2}^2 + \alpha^{1/4}N^{-3}\|\rho_\infty(0)\|_{H^{-1/8}}^2 \\ 
			& \leqc 1 + \|f_\alpha\|_{W^{1,\infty}}^2 \frac{\log N}{\log \alpha} + \alpha^{1/4}N^{-3}(1+\alpha^{1/2}\|f_\alpha\|_{W^{1,\infty}}^2), \label{eq:continuousUpper3}
		\end{align}
		which implies \eqref{eq:continuousUpper2}.
	\end{proof}
	
	\begin{proof}[Proof of Theorem~\ref{thm:continuous}]
		Suppose now that $f_\alpha$ satisfies \eqref{eq:falphaconditions}. Then, it is immediate from \eqref{eq:continuousUpper3} that 
		$$\|\Pi_{\le N} \rho_\infty(t)\|^2_{L^2} \leqc \frac{\log N}{\log \alpha}$$
		for all $t \in [0,1]$ and $N$ sufficiently large. It remains to prove the lower bound in \eqref{eq:continuousBatchelor}. We first note that since 
		$$\rho_\infty(0) = \sum_{n=0}^\infty f_\alpha \circ T_\alpha^{-n}$$
		and $\rho_\infty$ is time-one periodic, 
		it follows from \eqref{eq:falphaconditions} and Theorem~\ref{thm:discrete} that there exists $N_\alpha, C_0 \ge 1$ such that 
		\begin{equation} \label{eq:integerlower}
			\|\Pi_{\le N} \rho_\infty(m)\|^2_{L^2} \ge \frac{1}{C_0} \frac{\log N}{\log \alpha} 
		\end{equation}
		for every $N \ge N_\alpha$ and integer time $m \in \N \cup \{0\}$. Suppose now that there exists $t_0 \in (0,1)$ and $N_0 \ge N_\alpha^{100}$ such that 
		\begin{equation}\label{eq:lowercontradiction}
			\|\Pi_{\le N_0} \rho_\infty(t_0)\|_{L^2}^2 \le c_0 \frac{\log N_0}{\log \alpha}
		\end{equation}
		for some $c_0 > 0$. We will show that \eqref{eq:lowercontradiction} combined with Lemma~\ref{lem:Fourier} and the lower bound in \eqref{eq:integerlower} at $m = 1$ implies an $\alpha$-independent lower bound on $c_0$. First, since $\rho_\infty$ solves \eqref{eq:transport}, we have
		$$\rho_\infty(1) = \int_{t_0}^1 F_s \circ \phi_{1,s}^\alpha \, \dee s + \rho_\infty(t_0) \circ \phi_{1,t_0}^\alpha.$$
		Then, using computations identical to those giving \eqref{eq:continuousUpper3}, it follows from Lemma~\ref{lem:Fourier}, \eqref{eq:lowercontradiction}, and \eqref{eq:continuousnegative} that 
		\begin{equation} \label{eq:lowercontradiction1}
			\|\Pi_{N_0^{1/100}} \rho_\infty(1)\|_{L^2}^2 \leqc 1 + c_0 \frac{\log N_0}{\log \alpha} + \alpha^{3/4} N_0^{-3/100}.
		\end{equation}
		As $N_0^{1/100} \ge N_\alpha$, combining \eqref{eq:lowercontradiction1} and \eqref{eq:integerlower} implies that there exists $C_1 \ge 1$ independent of $\alpha$ such that 
		\begin{equation}
			\frac{1}{100 C_0} \frac{\log N_0}{\log \alpha} \le C_1\left(1 + c_0 \frac{\log N_0}{\log \alpha} + \alpha^{3/4} N_0^{-1/50}N_\alpha^{-1}\right).
		\end{equation}
		Provided that $N_\alpha$ is sufficiently large, the above results in 
		$$c_0 \ge \frac{1}{300 C_0 C_1}$$
		and we conclude that 
		$$\|\Pi_{\le N} \rho_\infty(t)\|_{L^2}^2 \ge \frac{1}{300 C_0 C_1} \frac{\log N}{\log \alpha}$$ 
		for every $t \in [0,1]$ and $N \ge N_\alpha^{100}$, which completes the proof. 
	\end{proof}

	\subsubsection{Proof of Theorem~\ref{thm:flux}}
	
	In this section, we consider the energy flux 
	$$\mathcal{T}_N = \int_0^1 \langle \rho_\infty(t), \Pi_{\le N} F_t\rangle_{L^2} \, \dee t, $$ 
	introduced in \eqref{eq:flux}, when $F$ is given by \eqref{eq:forcespecial}. Recall that in this case, 
	$$ f_\alpha(x,y) = h(y) \int_{1/2}^1 \eta(t) \, \dee t.$$
	We will assume without loss of generality that $\int_{1/2}^1 \eta(t) \, \dee t = 1$, so that $f_\alpha(x,y) = h(y)$.
	
	Theorem~\ref{thm:flux} asserts that $\lim_{N \to \infty} \mathcal{T}_N$ exists and is strictly positive. The former is immediate from the results proven in Section~\ref{sec:continuousproof} above. Indeed, since $\rho_\infty \in L^\infty([0,\infty);H^{-1})$ and $F$ is smooth, we have 
	\begin{equation}
		\lim_{N \to \infty} \mathcal{T}_N:= \mathcal{T} = \int_0^1  \langle \rho_\infty(t), F_t\rangle_{L^2} \, \dee t.
	\end{equation}
	Our goal here is thus to show that $\mathcal{T} > 0$. In view of \eqref{eq:rhoinftycontinuous}, the expression above can be expanded as 
	\begin{equation} \label{eq:fluxlimit}
		\mathcal{T} = \int_0^1 \int_0^t \langle F_s \circ \phi_{t,s}^\alpha, F_t \rangle_{L^2} \, \dee s \, \dee t + \sum_{n=0}^\infty \int_0^1 \langle f_\alpha \circ T_\alpha^{-n} \circ \phi_{t,0}^\alpha, F_t\rangle_{L^2} \, \dee t.
	\end{equation}
	The positivity of $\mathcal{T}$ will come from the first term, while the summation will be treated as an error using \eqref{eq:uniformmix2} of Theorem~\ref{prop:UniformMixing}. Since \eqref{eq:uniformmix2} gives no decay for $n = 0$, we first treat that term separately. 
	
	\begin{lemma} \label{lem:n=0}
		Let $F$ be given by \eqref{eq:forcespecial}. Then, 
		\begin{equation} 
			\left|\int_0^1 \langle f_\alpha \circ \phi_{t,0}^\alpha, F_t \rangle_{L^2} \, \dee t \right| \leqc \frac{\|h\|_{L^2(\T)}^2}{\alpha}.
		\end{equation}
	\end{lemma}
	
	\begin{proof}
		A direct computation using \eqref{eq:forcespecial} and $f_\alpha(x,y) = h(y)$ shows that 
		\begin{equation} \label{eq:n=0integral}
			\int_0^1 \langle f_\alpha \circ \phi_{t,0}^\alpha, F_t \rangle_{L^2} \, \dee t = \int_{\T^2} h(y-\alpha |x-1/2|)h(y) \, \dee y \, \dee x.
		\end{equation}
		Expressing $h(y) = \sum_{k \in \Z\setminus \{0\}} h_k e^{2\pi i k y}$ as a Fourier series, we have 
		\begin{align}
			\left|\int_{\T^2} h(y-\alpha |x-1/2|)h(y) \, \dee y \, \dee x\right| & = \left| \sum_{k \in \Z \setminus \{0\}}|h_k|^2 \int_{0}^1 e^{-2\pi i k\alpha |x-1/2|} \, \dee x\right| \\ 
			& \leqc\frac{1}{\alpha} \sum_{k \in \Z\setminus \{0\}} |h_k|^2 = \frac{1}{\alpha} \|h\|_{L^2(\T)}^2.
		\end{align}
		The result then follows by putting the bound above into \eqref{eq:n=0integral}.
	\end{proof}
	
	\begin{proof}[Proof of Theorem~\ref{thm:flux}]
		Throughout the proof we allow for implicit constants to depend on both $F$ and $f_\alpha$, as in the present setting they are both smooth functions independent of $\alpha$. Let $I_\alpha = \{t \in [1/2,1]:|t-1/2| < \alpha^{-1/2}\}$ and $I_\alpha^c = [1/2,1]\setminus I_\alpha$. We have 
		$$\left|\langle f_\alpha \circ T_\alpha^{-n} \circ \phi_{t,0}^\alpha, F_t\rangle_{L^2}\right| \leqc \mathbf{1}_{I_\alpha}(t) \|f_\alpha \circ T_\alpha^{-n} \circ \phi_{t,0}^\alpha\|_{H^{-1/8}} + \mathbf{1}_{I_\alpha^c}(t)\left|\langle f_\alpha \circ T_\alpha^{-n}, F_t \circ \Phi_{0,t}^\alpha\rangle_{L^2}\right|.$$
		It then follows from \eqref{eq:continuousabsolute} applied with $s = 1/8$ and Theorem~\ref{prop:UniformMixing} that
		\begin{align} \sum_{n=1}^\infty \int_0^1 \left|\langle f_\alpha \circ T_\alpha^{-n} \circ \phi_{t,0}^\alpha, F_t\rangle_{L^2}\right| \, \dee t & \leqc \alpha^{1/4}|I_\alpha|  + |I_\alpha^c| \sum_{n=1}^\infty e^{-n c \log \alpha} \\ 
			& \leqc \alpha^{-1/4} + \alpha^{-c}. \label{eq:fluxproof1}
		\end{align}
		From a direct computation using \eqref{eq:forcespecial}, we find 
		\begin{equation} \label{eq:fluxproof2}
			\int_0^1 \int_0^t \langle F_s \circ \phi_{t,s}^\alpha, F_t \rangle_{L^2} \, \dee s \, \dee t = \|h\|_{L^2(\T)}^2 \int_{1/2}^1 \eta(t)(t-1/2) \dee t = c_0 
		\end{equation}
		for some constant $c_0 > 0$ that is independent of $\alpha$. Putting the previous two estimates into \eqref{eq:fluxlimit} and using Lemma~\ref{lem:n=0} to bound the $n=0$ term of the sum yields 
		$$\mathcal{T} \ge c_0 - C_0\left(\alpha^{-1}+\alpha^{-1/4} + \alpha^{-c}\right)$$
		for some constant $C_0 \ge 1$ independent of $\alpha$. This proves that $\mathcal{T} \ge c_0/2 > 0$ for $\alpha$ sufficiently large. 
	\end{proof}
	
	\section{Uniform-in-$\alpha$ exponential mixing} \label{sec:uniform}
	
	This section is devoted to the proof of Theorem~\ref{prop:UniformMixing}. In Section~\ref{sec:hyperbolic}, we recall from \cite{ELM25} the uniform hyperbolicity of $T_\alpha$ and the structure of its singularity sets. In Section~\ref{sec:anisotropicoutline}, we begin with a brief overview of anisotropic Banach space methods for correlation decay and discuss the basic strategy of our proof. We then introduce our specific anisotropic framework and reduce the proof of Theorem~\ref{prop:UniformMixing} to the contraction estimates stated in Propositions~\ref{prop:StableUnstable} and~\ref{prop:OneStep}. Section~\ref{sec:Sminus} collects preliminary lemmas regarding the backward singularity set, and the proofs of Propositions~\ref{prop:StableUnstable} and~\ref{prop:OneStep} are given in Section~\ref{sec:contractionproofs}.
	
	\subsection{Uniform hyperbolicity of $T_\alpha$} \label{sec:hyperbolic}
	Define $T_{1,\alpha},T_{2,\alpha}:\T^2 \to \T^2$ by 
	\begin{equation} \label{eq:shears}
		T_{1,\alpha}(x,y) = \begin{pmatrix} x \\ 
			y+\alpha|x-1/2|\end{pmatrix} \mod 1, \quad T_{2,\alpha}(x,y) = \begin{pmatrix} x+\alpha|y-1/2| \\ 
			y\end{pmatrix} \mod 1.
	\end{equation}
	Then, $T_\alpha$ is given by the composition. 
	\begin{equation}\label{eq:wedge}
		T_\alpha = T_{2,\alpha} \circ T_{1,\alpha}.
	\end{equation}
	Both $T_{1,\alpha}$ and $T_{2,\alpha}$ are continuous, piecewise affine maps on $\T^2$. Specifically, $T_{1,\alpha}$ is an affine map modulo 1 on $(0,1/2)\times \T$ and $(1/2,1)\times\T$, while $T_{2,\alpha}$ is affine modulo 1 on $\T \times (0,1/2)$ and $\T \times (1/2,1)$. It follows that $T_\alpha = T_{2,\alpha} \circ T_{1,\alpha}$ is piecewise smooth on $\T^2$ and has constant derivative on each of  
	\begin{align}
		\mathcal{C}_1^+&=((0,1/2)\times \T) \cap T_{1,\alpha}^{-1}
		(\T \times (0,1/2)),
		&
		\mathcal{C}_2^+&= ((1/2,1)\times \T)
		\cap T_{1,\alpha}^{-1}
		(\T\times
		(0,1/2)),\\
		\mathcal{C}_3^+&= ((0,1/2)\times \T )\cap T_{1,\alpha}^{-1}
		(\T \times (1/2,1)),
		&
		\mathcal{C}_4^+&=( (1/2,1)\times \T) \cap T_{1,\alpha}^{-1}
		(\T \times (1/2,1)).
	\end{align}
	Notice that each $\mathcal{C}_i^+$ can be written as $T_{1,\alpha}^{-1}(R)$ for some square $R$, and hence is connected. For example, $\mathcal{C}_1^+ = T_{1,\alpha}^{-1}((0,1/2)\times (0,1/2))$. We define the \textit{singularity set} of $T_\alpha$ by 
	\begin{equation} \label{eq:Splusdef}
		\mathcal{S}^+ = \bigcup_{i=1}^4 \partial \mathcal{C}_i^+,
	\end{equation}
	so that  $\{\mathcal{C}_i^+\}_{i=1}^4$ partitions $\T^2 \setminus \mathcal{S}^{+}$ into open, connected components on which $T_\alpha$ is smooth with $D T_\alpha$ constant. Here, we have written $\partial \mathcal{C}$ to denote the boundary of a set $\mathcal{C} \subseteq \T^2$. Direct computation shows that the possible values of $D T_\alpha$ are the elements of
	\begin{equation} \label{eq:matrices}
		\mathcal{A}:= \left\{
		\begin{pmatrix}
			1 + \alpha^2 & \alpha \\
			- \alpha & 1
		\end{pmatrix},
		\begin{pmatrix}
			1 + \alpha^2 & -\alpha \\
			- \alpha & 1
		\end{pmatrix},
		\begin{pmatrix}
			1 - \alpha^2 & \alpha \\
			\alpha & 1
		\end{pmatrix},
		\begin{pmatrix}
			1 - \alpha^2 & -\alpha \\
			\alpha & 1
		\end{pmatrix}
		\right\}.
	\end{equation}
	Naturally, analogous properties hold for the inverse map $T_\alpha^{-1}$. In particular, defining $\mathcal{C}_i^- = T_\alpha(\mathcal{C}_i^+)$ and $\mathcal{S}^- = T_\alpha(\mathcal{S}^+)$, we have that $\{\mathcal{C}_i^{-1}\}_{i=1}^4$ partitions $\T^2 \setminus \mathcal{S}^-$ into connected components where $D T_\alpha^{-1} = A^{-1}$ for some $A \in \mathcal{A}$.

	Each matrix $A \in \mathcal{A}$ can be written as 
	\begin{equation} \label{eq:factor}
		A = \alpha^2\left[\begin{pmatrix} \pm 1 & 0 \\ 0 & 0\end{pmatrix} + \alpha^{-1} M\right] 
	\end{equation}
	for some $M \in \R^{2\times 2}$ with $\max_{1 \le i,j \le 2}|M_{ij}| \le 1$, and hence for large $\alpha$ one expects the eigenvalues and eigenvectors of $A$ to be close to those of the matrix obtained by setting $M$ to zero in \eqref{eq:factor}. Indeed, direct computation shows that for $\alpha \gg 1$, every $A \in \mathcal{A}$ has two distinct eigenvalues $\lambda_u, \lambda_s$ that satisfy 
	\begin{equation}\label{eq:eigenvalues} \alpha^2 - 3 \le |\lambda_u| \le \alpha^2 +3, \quad |\lambda_s \lambda_u| = 1
	\end{equation}
	and associated eigenvectors $e_u$ and $e_s$ such that 
	\begin{equation}
		\tan[\angle(e_u,\mathrm{Span}\{e_x\})] \le 2\alpha^{-1}, \quad \tan[\angle(e_s, \mathrm{Span}\{e_y\})] \le 2\alpha^{-1}, \label{eq:eigenvectors}
	\end{equation}
	where $e_x$ and $e_y$ are the standard basis vectors of $\R^2$. The bounds \eqref{eq:eigenvalues} and \eqref{eq:eigenvectors} imply that for $\alpha$ large, $T_\alpha$ is a \textit{uniformly hyperbolic} map with singularities that possesses a pair of constant, $D T_\alpha$-invariant cone fields. Precisely, if we define the cones
	\begin{equation} \label{eq:cones}
		C_s = \left\{(x,y) \in \R^2: |x| \le \frac{4|y|}{\sqrt{\alpha}}\right\} \quad \text{and} \quad C_u = \left\{(x,y) \in \R^2: |y| \le \frac{4|x|}{\sqrt{\alpha}}\right\},
	\end{equation}
	then for all $\alpha$ sufficiently large we have 
	\begin{equation} \label{eq:forwardinvariant}
		(D T_\alpha)(z)C_u \subseteq C_u \quad \forall z \in \T^2 \setminus \mathcal{S}^+, \qquad \inf_{v \in C_u} \inf_{z \in \T^2 \setminus \mathcal{S}^+} \frac{|(D T_\alpha)(z)v|}{|v|} \ge \frac{\alpha^2}{2}
	\end{equation}
	and
	\begin{equation} \label{eq:backwardinvariant}
		(DT^{-1}_\alpha)(z)C_s \subseteq C_s \quad \forall z \in \T^2 \setminus \mathcal{S}^-, \qquad \inf_{v \in C_s} \inf_{z \in \T^2 \setminus \mathcal{S}^-} \frac{|(D T^{-1}_\alpha)(z)v|}{|v|} \ge \frac{\alpha^2}{2}.
	\end{equation}
	In fact, computing explicitly with \eqref{eq:matrices}, one can check that 
	\begin{equation} \label{eq:leadingorder}
		\alpha^2 - 10 \alpha^{3/2} \le \frac{|(D T_\alpha)(z) v|}{|v|} \le \alpha^2 + 10\alpha^{3/2} \qquad \forall v \in C_u, \quad z \in \T^2 \setminus \mathcal{S}^+,
	\end{equation}
	as well as corresponding inequalities for the inverse map
	\begin{equation} \label{eq:leadingorder2}
		\alpha^2 - 10 \alpha^{3/2} \le \frac{|(D T^{-1}_\alpha)(z) v|}{|v|} \le \alpha^2 + 10\alpha^{3/2} \qquad \forall v \in C_s, \quad z \in \T^2 \setminus \mathcal{S}^-.
	\end{equation}
	That the expansion rate satisfies upper and lower bounds matching to leading order in $\alpha$ will be important in the proof of Theorem~\ref{prop:Projected}, but \eqref{eq:forwardinvariant} and \eqref{eq:backwardinvariant} will be sufficient for our purposes in the present section. In what follows, we will repeatedly use that by \eqref{eq:backwardinvariant}, for any $z \in \T^2 \setminus \mathcal{S}^-$ and $v \in C_s$ we have 
	\begin{equation} \label{eq:forwardcontract0}
		\frac{|(D T_\alpha)(z')v'|}{|v'|} \le \left(\frac{\alpha^2}{2}\right)^{-1},
	\end{equation}
	where $ z'= T_\alpha^{-1}(z) \in \T^2 \setminus \mathcal{S}^+$ and $v' = (D T_\alpha^{-1})(z) v \in C_s.$
	
	Detailed properties of the singularity sets $\mathcal{S}^{\pm}$ (as well as those of higher iterates of $T_\alpha$) play an important role in the proofs of our mixing estimates and will be discussed later in Sections~\ref{sec:Sminus} and~\ref{sec:nstep}. For now, we only note that $\mathcal{S}^-$ can be written as
	\begin{equation}\label{eq:ExplicitS1}
		\mathcal{S}^- = \{y=0\} \cup \{y=1/2\} \cup T_{2,\alpha}(\{x=0\}) \cup T_{2,\alpha}(\{x=1/2\}).
	\end{equation}
	This follows from the decomposition $T_\alpha^{-1} = T_{1,\alpha}^{-1} \circ T_{2,\alpha}^{-1}$: the singularity set of $T_\alpha^{-1}$ consists of the singularity set of $T_{2,\alpha}^{-1}$ together with the preimage under $T_{2,\alpha}^{-1}$ of the singularity set of $T_{1,\alpha}^{-1}$. By \eqref{eq:ExplicitS1} and the definition of $T_{2,\alpha}$, we see that $\mathcal{S}^-$ is a finite union of line segments whose tangent vectors lie in the unstable cone $C_u$ (see Figure~\ref{fig:backward8}), and hence are bounded away from the stable cone $C_s$.
	
	\begin{remark} \label{rem:largecones}
		The cones $C_u$ and $C_s$ can actually be taken much smaller and still be $DT_\alpha$-invariant. One can check that \eqref{eq:forwardinvariant} and \eqref{eq:backwardinvariant} hold with $C_s$ and $C_u$ replaced by
		\begin{equation} \label{eq:smallcones}
			\widetilde{C}_s = \left\{(x,y) \in \R^2: |x| \le \frac{4|y|}{\alpha}\right\} \quad \text{and} \quad \widetilde{C}_u = \left\{(x,y) \in \R^2: |y| \le \frac{4|x|}{\alpha}\right\}.
		\end{equation}
		In fact, it holds that $(DT_\alpha)(z)C_u \subseteq \tilde{C}_u$, as well as the corresponding backwards statement. While $\widetilde{C}_s$ and $\widetilde{C}_u$ have the natural scaling in $\alpha$, we work with the slightly larger cones defined in \eqref{eq:cones} because they are simultaneously invariant for the derivative of $\Phi_{0,t}^\alpha$, as it appears in \eqref{eq:uniformmix2}, which will be necessary for a small portion of the proof.  
	\end{remark}

	\subsection{Anisotropic framework and reduction of the proof} \label{sec:anisotropicoutline}
	
	Our proof of Theorem~\ref{prop:UniformMixing} is inspired by the functional--analytic approach to the statistical properties of uniformly hyperbolic systems, which has been developed extensively over the past few decades. The basic goal of the method is to construct a Banach space on which the associated transfer operator has good spectral properties (in our setting, the transfer operator $\mathcal{L}_\alpha$ acts on functions by $\mathcal{L}_\alpha f = f \circ T_\alpha^{-1}$). The fundamental principle in building such spaces is to work with norms adapted to the hyperbolicity: under iteration, the dynamics create fine scales along stable directions while producing an effective smoothing along unstable directions. This leads naturally to anisotropic norms that measure weak (negative) regularity in stable directions and strong (positive) regularity, or smoothing, in unstable directions. On such spaces the transfer operator can be quasi-compact, yielding a spectral gap (in the absence of peripheral eigenvalues) and hence exponential decay of correlations. Two broad classes of constructions have emerged in the literature. One is a geometric approach defining anisotropic norms through integration against families of test functions supported on approximate stable curves (see, e.g., \cite{DemersLiverani2d,BKL02,GouezelLiverani06,BaladiDemersLiverani}). The other is based on anisotropic Sobolev-type norms defined via Fourier multipliers or pseudo-differential operators adapted to the hyperbolic splitting (see, e.g., \cite{BaladiTsujii,Baladi2005Anisotropic,FaureRoySjoestrand2008,baladi2009good}). Further discussion and broader surveys can be found in \cite{DemersGentle,DLNbook,baladi2000positive}.
	
	The map $T_\alpha$ falls within the general class of two-dimensional
	piecewise hyperbolic maps with singularities studied in
	\cite{DemersLiverani2d}. There, anisotropic Banach spaces defined via
	geometric norms were constructed on which the transfer operator acts
	quasi-compactly. The exponential mixing estimate \eqref{eq:umixing}
	was established in \cite{ELM25} by ruling out the existence of
	peripheral spectrum and appealing to the resulting spectral gap
	provided by \cite{DemersLiverani2d}. While this is a powerful approach
	to establish exponential mixing, it is not sufficient for our purposes
	here. In the present work, it is crucial to obtain quantitative growth
	of the mixing rate as the amplitude parameter $\alpha$
	increases. Abstract quasi-compactness arguments yield estimates on the
	\textit{essential} spectral radius, but generally provide no
	information on the discrete spectrum and therefore no quantitative
	control on the spectral gap. Our proof of
	Theorem~\ref{prop:UniformMixing} instead exploits the large-$\alpha$
	regime and our detailed understanding of the singularity sets of
	$T_\alpha$ to prove a finite-step contraction estimate for the
	transfer operator in an anisotropic norm adapted from
	\cite{DemersLiverani2d}. This allows us to directly estimate the
	exponential decay rate. It is likely possible to quantify the mixing
	rate using other approaches to correlation decay, such as Young Towers
	or Standard Pairs \cite{Young1998,Young1999,Dolgopyat}. However, the
	detailed dynamical information provided by decay in an anisotropic
	norm will play a crucial role later in the paper, serving as a key
	lemma in our proof of Theorem~\ref{prop:Projected} (see
	Section~\ref{sec:projecteddiscussion} for more discussion).
	
	\subsubsection{Definition of the anisotropic norms} \label{sec:normdefs}
	We now define the specific anisotropic norms that we employ, which follow closely the geometric norms from \cite{DemersLiverani2d}, but with some modifications and simplications available in our specific setting. Let $\mathcal{W}_s$ denote the set of line segments in $\T^2$ whose tangent vectors lie in $C_s$. We define the set of \textit{admissible stable curves}
	\begin{equation} \label{eq:admissible}
		\Sigma = \{W \in \mathcal{W}_s: |W| \le 2\},
	\end{equation}
	where $|W|$ denotes the length of $W$. For $W \in \mathcal{W}_s$ and $\gamma \in (0,1]$, let $C^{\gamma}(W)$ denote the set of H\"{o}lder continuous functions on $W$ with exponent $\gamma$. For $\varphi \in C^\gamma(W)$, we define 
	\begin{equation}\label{eq:Holderdef}
		[\varphi]_{\gamma,W} = \sup_{z_1,z_2 \in W} \frac{|\varphi(z_1) - \varphi(z_2)|}{d_{W}(z_1,z_2)^\gamma}, \quad |\varphi|_{C^\gamma(W)} = [\varphi]_{\gamma,W} + \sup_{z \in W}|\varphi(z)|, 
	\end{equation}
	where $d_W$ denotes the arclength distance along $W$. Throughout this entire section, we ignore endpoint conventions for the line segments $W \in \Sigma$, as they never play a role in the arguments. For example, segments that differ only by endpoints are identified and statements concerning the partition of a set into line segments are understood up to finitely many omitted points.
	
	Choose  $p, \beta > 0$ satisfying $0 < p \le 1-\beta$ and let $\dee m_W$ denote the arclength measure on $W \in \mathcal{W}_s$. Given $f \in W^{1,\infty}(\T^2)$, we define its \textit{weak stable norm} by
	\begin{equation}\label{eq:weakdef}
		|f|_w = \sup_{W \in \Sigma} \Big(\sup_{|\varphi|_{C^1(W)}\le 1} \int_W f \varphi \, \dee m_W\Big)
	\end{equation}
	and its \textit{strong stable norm} by
	\begin{equation}\label{eq:strongstabledef}
		\|f\|_s = \sup_{W \in \Sigma} \Big(\sup_{|\varphi|_{C^\beta(W)}\le |W|^{-p}} \int_W f \varphi \, \dee m_W\Big).
	\end{equation}
	Note these definitions are such that $|f|_w \leqc \|f\|_s$ for any choice of $p, \beta > 0$. Roughly speaking, the norms above measure the negative regularity of $f$ along stable curves. To define a measure of positive regularity in unstable directions, we introduce a family $\Sigma_u \subseteq \Sigma \times \Sigma$ of \textit{admissible pairs} and a notion of distance between them. Let
	\begin{equation} \label{eq:pairdef}
		\Sigma_u = \{(W_1,W_2) \in \Sigma \times \Sigma: |W_1| = |W_2| \text{ and } W_1 \parallel W_2\}
	\end{equation}
	denote the pairs of admissible curves that are parallel and have equal length. For any $(W_1,W_2) \in \Sigma_u$, there exists $h \in \T^2$ such that 
	\begin{equation} \label{eq:translation}
		W_2 = W_1+h,
	\end{equation} 
	and therefore we may define
	\begin{equation} 
		d_{\Sigma}(W_1,W_2) = |h|_{\T^2},
	\end{equation}
	where $h$ satisfies \eqref{eq:translation} and $|h|_{\T^2} := d_{\T^2}(0,h)$. Here, $d_{\T^2}$ denotes the standard distance function on $\T^2$, defined for $z_1,z_2 \in \T^2$ by $d_{\T^2}(z_1,z_2) = \min_{k \in \Z^2}|z_1 - z_2 + k|$. For notational convenience, when $z \in \T^2$ we abbreviate $|z|_{\T^2} = |z|$. The \textit{strong unstable norm} of $f \in W^{1,\infty}(\T^2)$ is defined as 
	\begin{equation}\label{eq:unstablenorm}
		\|f\|_{u} = \sup_{\substack{(W_1,W_2) \in \Sigma_u \\ d_{\Sigma}(W_1,W_2) \le \alpha^{-1}}} d_{\Sigma}(W_1,W_2)^{-p} \left|\int_{W_1} f \, \dee m_{W_1} - \int_{W_2} f \, \dee m_{W_2}\right|. 
	\end{equation} 
	We also combine the strong stable/unstable norms and define the \textit{strong norm} of $f$ by 
	\begin{equation} \label{eq:strongdef}
		\tnorm{f}:= \|f\|_u + \|f\|_s.
	\end{equation}
	
	\begin{remark}
		It is possible to work with more general classes of admissible curves, but restricting to parallel, equal-length segments is natural in our setting: this property is preserved by each affine part of the dynamics and simplifies certain computations.
	\end{remark}
	
	\begin{remark}
		Typically, in geometric constructions of anisotropic norms, the strong unstable norm is defined by testing against general H\"{o}lder continuous functions on nearby admissible curves that are close in a suitably defined metric. By contrast, $\|\cdot\|_u$ introduced above restricts to constant test functions, and hence only measures differences of averages along stable leaves. This forfeits key features of the usual functional analytic framework. In particular, the standard pair of ``strong'' and ``weak'' Banach spaces associated with our norms do not clearly satisfy the compact embedding necessary to deduce quasi-compactness of $\mathcal{L}_\alpha$ from a set of Lasota-Yorke inequalities (see, e.g., \cite{DemersGentle} for context). This, however, is irrelevant to our argument since we proceed by establishing a finite-step contraction estimate for $\mathcal{L}_\alpha$ directly in the strong norm, and therefore never rely on quasi-compactness or properties of the associated Banach spaces.
	\end{remark}
	
	\subsubsection{Contraction estimates and the proof of Theorem~\ref{prop:UniformMixing}}
	
	The main estimates that we will prove for $\mathcal{L}_\alpha:W^{1,\infty}(\T^2) \to W^{1,\infty}(\T^2)$ are stated in the following two propositions. Proposition~\ref{prop:StableUnstable} collects bounds similar in spirit to those obained in typical functional--analytic approaches, while Proposition~\ref{prop:OneStep} contains the key one-step estimate that will allow us to avoid arguing via quasi-compactness.
	
	\begin{proposition} \label{prop:StableUnstable}
		For all $\alpha \gg 1$ and $f \in W^{1,\infty}(\T^2)$ we have the estimates:
		\vspace{0.2cm}
		\begin{itemize}
			\item[(a)] $\|\mathcal{L}_\alpha f\|_s \leqc \alpha^{-2 \beta}\|f\|_s + |f|_w$,
			\vspace{0.2cm}
			\item[(b)] $\left|f \circ \phi_{t,0}^\alpha\right|_w \leqc |f|_w \quad \forall \, t \in [\tfrac{1}{2} + \alpha^{-1/2},1]$,
			\vspace{0.2cm}
			\item[(c)] $\|\mathcal{L}_\alpha f\|_{u} \leqc \alpha^{-p}\|f\|_{u} + \|f\|_s.$
		\end{itemize}
	\end{proposition}
	
	\begin{proposition}\label{prop:OneStep}
		For all mean-zero $f \in W^{1,\infty}(\T^2)$ and $\alpha \gg 1$, we have 
		\begin{equation}\label{eq:OneStep}
			|\mathcal{L}_\alpha f|_{w} \leqc \alpha^{-p}\tnorm{f}.
		\end{equation}
	\end{proposition}
	
	Combining Propositions~\ref{prop:StableUnstable} and~\ref{prop:OneStep} gives an exponential decay estimate of $|\mathcal{L}_\alpha^n f|_w$, from which Theorem~\ref{prop:UniformMixing} follows as an easy corollary.
	
	\begin{theorem} \label{thm:exponential}
		Let 
		\begin{equation}
			q = \frac{1}{6}\min(2\beta,p) > 0.
		\end{equation}
		There exists $C_\alpha \ge 1$ such that for all mean-zero $f \in W^{1,\infty}(\T^2)$, $\alpha \gg 1$, and $n \in \N$ we have
		\begin{equation} \label{eq:spectralgap}
			\tnorm{\mathcal{L}_\alpha^n f} \le C_\alpha \alpha^{-qn} \tnorm{f}
		\end{equation}
		and 
		\begin{equation} \label{eq:weakdecay}
			|\mathcal{L}_\alpha^n f|_w  \leqc \alpha^{-qn} \tnorm{f}.
		\end{equation}
	\end{theorem}
	
	\begin{proof}
		Since $|f|_w \leqc \|f\|_s$, as a basic consequence of
		Proposition~\ref{prop:StableUnstable}, we have
		\begin{equation} \label{eq:boundedoperator}
			\|\mathcal{L}_\alpha f\|_s \leqc \|f\|_s \quad \text{and} \quad \tnorm{\mathcal{L}_\alpha f} \leqc \tnorm{f}.
		\end{equation}
		Applying \eqref{eq:boundedoperator}, Proposition~\ref{prop:StableUnstable}(a), and Proposition~\ref{prop:OneStep} yields
		\begin{equation}\label{eq:iterate1}
			\|\mathcal{L}_\alpha^3 f\|_s \leqc	\|\mathcal{L}_\alpha^2 f\|_s \leqc \alpha^{-2\beta}\|\mathcal{L}_\alpha f\|_s + |\mathcal{L}_\alpha f|_w \leqc (\alpha^{-2\beta} + \alpha^{-p})\tnorm{f}.
		\end{equation}
		Using Proposition~\ref{prop:StableUnstable}(c), \eqref{eq:boundedoperator}, and \eqref{eq:iterate1} then gives
		\begin{equation} \label{eq:iterate2}
			\|\mathcal{L}_\alpha^3 f\|_u \leqc \alpha^{-p}\|\mathcal{L}_\alpha^2 f\|_u + \|\mathcal{L}^2_\alpha f\|_s \leqc (\alpha^{-2\beta} + \alpha^{-p})\tnorm{f}.
		\end{equation}
		Let $$r = \frac{1}{2}\min(2\beta,p) = 3q.$$ Combining \eqref{eq:iterate1} and \eqref{eq:iterate2} shows that for $\alpha$ sufficiently large we have 
		\begin{equation} 
			\tnorm{\mathcal{L}_\alpha^3 f} \le \alpha^{-r}\tnorm{f},
		\end{equation}
		and hence 
		\begin{equation} \label{eq:iterate3}
			\tnorm{\mathcal{L}_\alpha^{3m} f} \le \alpha^{-rm} \tnorm{f} \quad \forall m\in \N.
		\end{equation}
		Moreover, from Proposition~\ref{prop:OneStep} and \eqref{eq:boundedoperator}, it follows that for any $k \in \{0,1,2\}$ we have 
		\begin{equation}\label{eq:iterate4}
			|\mathcal{L}_\alpha^k f|_w \leqc \alpha^{-kp/2}\tnorm{f} \le \alpha^{-kr}\tnorm{f}.
		\end{equation}
		Now fix any $n \in \N$ and write $n = 3m+k$ for $m \in \N \cup \{0\}$ and $k \in \{0,1,2\}$. Then, from \eqref{eq:boundedoperator} and \eqref{eq:iterate3} we obtain 
		\begin{equation}
			\tnorm{\mathcal{L}^n_\alpha f} = \tnorm{\mathcal{L}_\alpha^k (\mathcal{L}_\alpha^{3m}f)} \leqc \tnorm{\mathcal{L}_\alpha^{3m}f} \le \alpha^{-rm}\tnorm{f} \le \alpha^{2r/3}\alpha^{-rn/3}\tnorm{f}= \alpha^{2r/3}\alpha^{-qn}\tnorm{f},
		\end{equation}
		which proves \eqref{eq:spectralgap}. Similarly,  using \eqref{eq:iterate4} and \eqref{eq:iterate3} gives 
		\begin{equation}
			|\mathcal{L}_\alpha^n f|_w \leqc \alpha^{-kr}\tnorm{\mathcal{L}_\alpha^{3m} f} \le \alpha^{-r(k+m)}\tnorm{f}\le \alpha^{-rn/3}\tnorm{f} = \alpha^{-qn}\tnorm{f},
		\end{equation}
		which is \eqref{eq:weakdecay}.
	\end{proof}
	
	\begin{proof}[Proof of Theorem~\ref{prop:UniformMixing}]
		
		For any mean-zero $f \in W^{1,\infty}(\T^2)$, there exists a sequence of mean-zero functions $\{f_n\} \subseteq C^1(\T^2)$ such that $\|f_n - f\|_{W^{1,\infty}} \to 0$ and $\sup_n\|f_n\|_{C^1} \le \|f\|_{W^{1,\infty}}$. Therefore, it suffices to prove the result for $f,g \in C^1(\T^2)$. Let $q > 0$ be as in the statement of Theorem~\ref{thm:exponential} and let
		$W_{x} = \{(x,y):x\in \T\} \in
		\Sigma$ denote the vertical segment of length one with horizontal coordinate $x$. It follows from Fubini's theorem, the definition of the weak stable norm, and \eqref{eq:weakdecay} that 
		\begin{align}
			\left|\int_{\T^2} \mathcal{L}_\alpha^n f(z) g(z) \, \dee z\right|&\le \int_\T \left|\int_{W_x} \mathcal{L}_\alpha^n f \, g \, \dee m_{W_x}\right| \, \dee x \\ 
			& \le |\mathcal{L}_\alpha^n f|_w \|g\|_{C^1} \\ 
			& \leqc \alpha^{-qn} \tnorm{f}\|g\|_{C^1}.
		\end{align}
		Since $\tnorm{f}\leqc \|f\|_{C^1(\T^2)}$, this proves \eqref{eq:uniformmix}. The bound \eqref{eq:uniformmix2} is equivalent to proving the same estimate as above with $\mathcal{L}_\alpha^n f$ replaced by $(\mathcal{L}_\alpha^n f) \circ \phi_{t,0}^\alpha$, and therefore follows from the same argument and Proposition~\ref{prop:StableUnstable}(c). 
	\end{proof}
	
	\subsection{Decomposition of $T^{-1}(W)$ and basic properties of $\mathcal{S}^-$} \label{sec:Sminus}
	
	Henceforth, we suppress the subscript $\alpha$ from the notation. We write $T = T_2 \circ T_1$ and denote the associated transfer operator by $\mathcal{L}$. Recall from Section~\ref{sec:hyperbolic} that the singularity set of $T^{-1}$ can be written as 
	\begin{equation} \label{eq:S-explicit}
		\mathcal{S}^- = \{y = 0\} \cup \{y = 1/2\} \cup T_2(\{x=0\}) \cup T_2(\{x=1/2\}),
	\end{equation}
	and hence consists of a finite union of line segments in $C_u$. In
	addition, the collection $\{\mathcal{C}_i^-\}_{i=1}^4$
	partitions $\T^2 \setminus \mathcal{S}^-$ into open, connected
	components on which $DT^{-1}$ is constant. Plots of
	$\mathcal{S}^-$ for $\alpha = 8, 8.5$ are shown below in
	Figure~\ref{fig:backward8}. The lines $ \{y = 0\}$ and
	$\{y = 1/2\}$ are drawn in red and blue,
	respectively. Both $T_2(\{x=0\})$ and $T_2(\{x=1/2\})$ each consist of two lines and are colored in green and orange, respectively. By the definition of $T_2:\T^2 \to \T^2$, the green and orange lines have slope $\pm \alpha^{-1}$. We will often refer to the orange and green segments that connect the two vertical edges of fundamental domain as \textit{spanning curves} of $\mathcal{S}^{-}$.  A simple computation shows that the perpendicular distance between adjacent, parallel spanning curves is $(2\sqrt{1+\alpha^2})^{-1} \approx \alpha^{-1}$. We also note that the maximal number of line segments comprising $\mathcal{S}^-$ that intersect at a single point is three. Such a triple intersection occurs at four points, regardless of $\alpha$. In the special case that $\alpha$ is an integer, these multi-intersection points occur at $\{(0,0), (1/2,0), (0,1/2), (1/2,1/2)\}$. Generally, the multi-intersections are located at $\{(0,1/2), (1/2,1/2), T_2(0,0), T_2(1/2,0)\}$.
	
	\begin{figure}[h]
		\includegraphics[width=8cm, height=8cm]{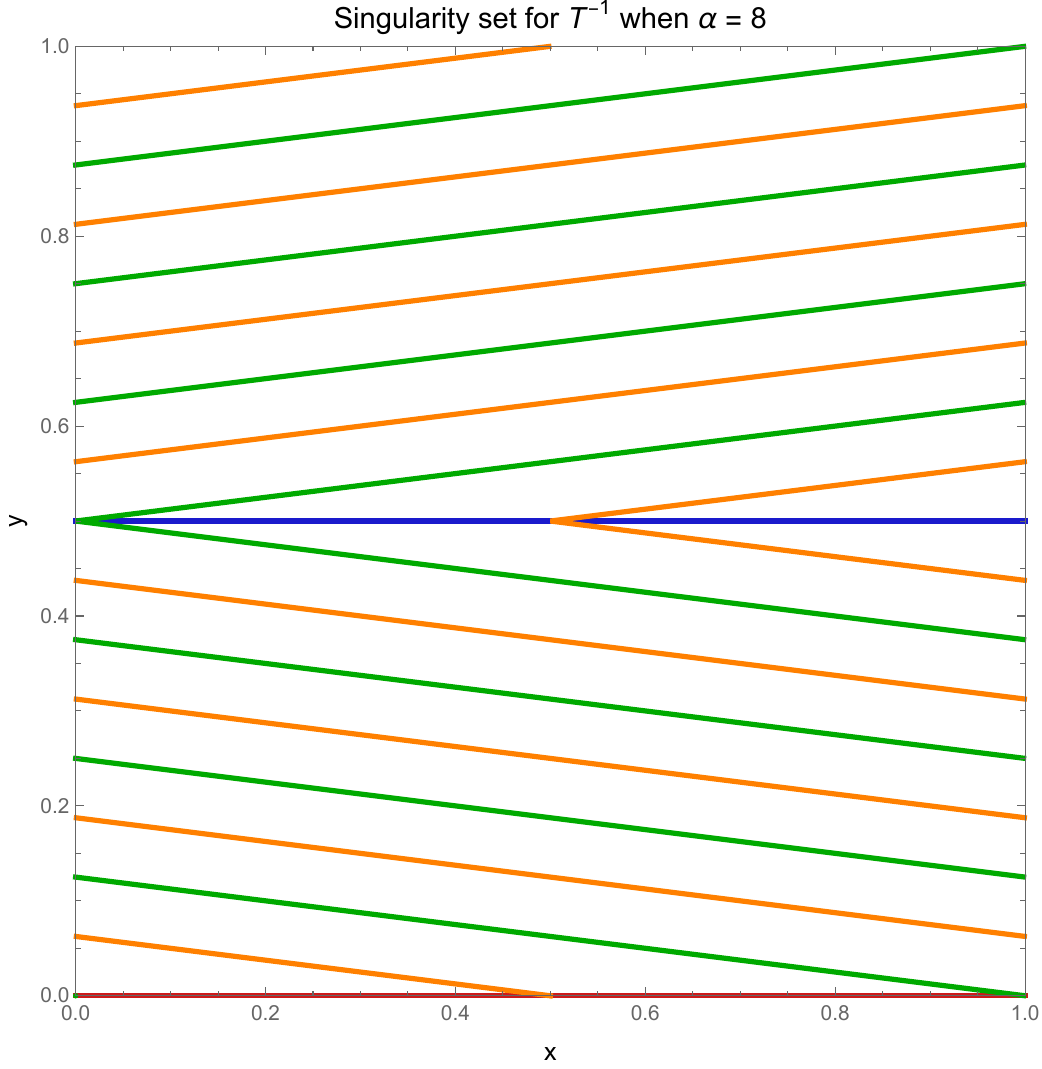}
		\includegraphics[width=8cm, height=8cm]{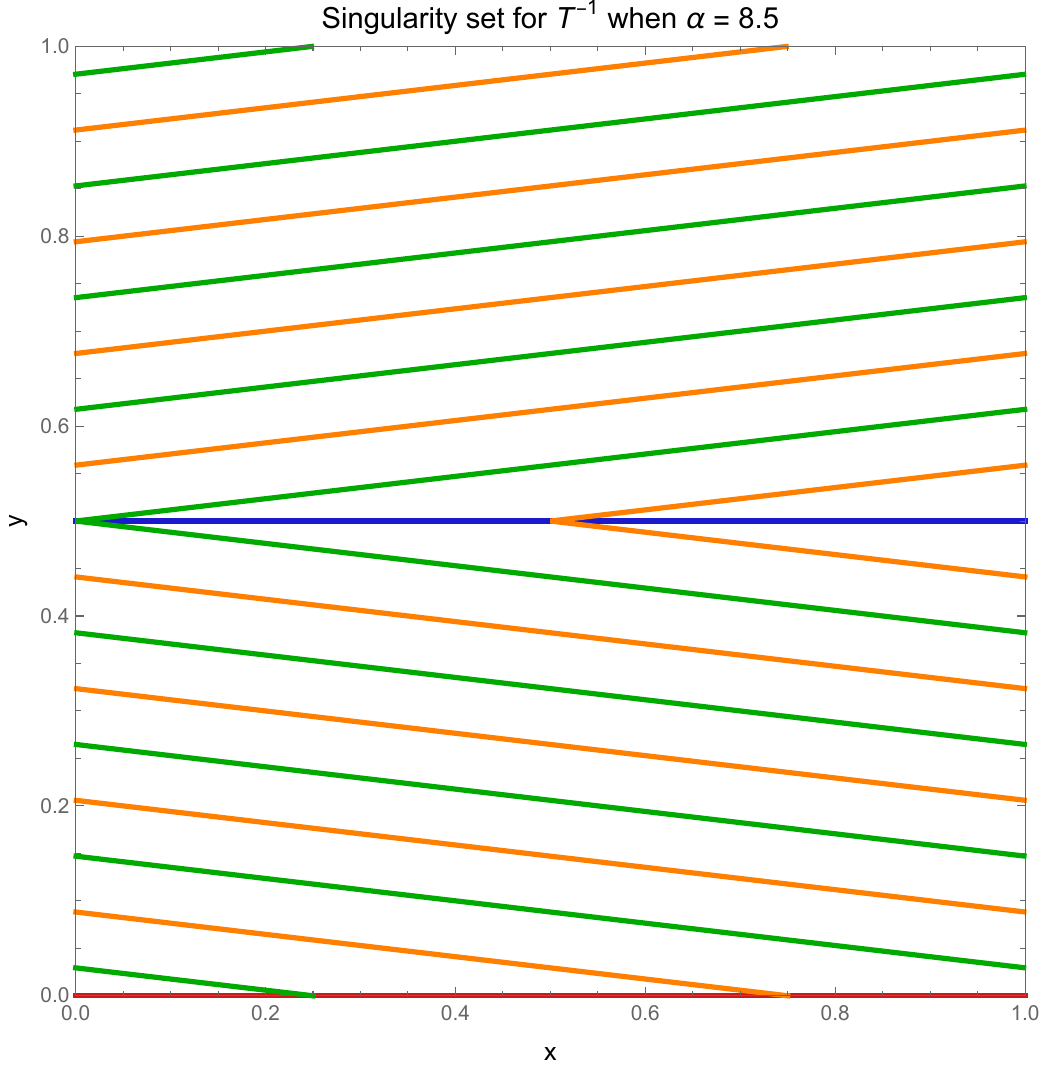}
		\caption{Plots of $\mathcal{S}^-$ for $\alpha = 8$ and $\alpha = 8.5$ obtained from \eqref{eq:S-explicit}. The green and orange lines show $T_2(\{x=0\})$ and $T_2(\{x=1/2\})$, respectively.}
		\label{fig:backward8}
	\end{figure}
	
	Estimating quantities such as $\|\mathcal{L} f\|_s = \|f \circ T^{-1}\|_s$ in terms of the norms of $f$ requires understanding the properties of $T^{-1}(W)$ for $W \in \Sigma$. Given an admissible segment $W$, we define a partition (up to finitely many removed points) of $T^{-1}(W)$ into disjoint line segments $\Psi(W) \subseteq \Sigma$ as follows. Since the tangent vector of $W$ lies in $C_s$ and $\mathcal{S}^-$ is a union of line segments in $C_u$, it follows that $W \cap \mathcal{S}^-$ is a finite set. Let $\{W_i'\}_{i=1}^M$ be the decomposition of $W \setminus \mathcal{S}^-$ into its connected components. For each $1 \le i \le M$, $W_i'$ is contained in just one element of $\{\mathcal{C}_i\}_{i=1}^4$, and hence the image $T^{-1}(W_i')$ consists of a single line segment $W_i$ that is tangent to some vector in $C_s$. If $|W_i| \le 2$, then we declare $W_i \in \Psi(W)$. If instead $|W_i| > 2$, then we subdivide $W_i$ into shorter segments $\{W_{ij}\}$ with $|W_{ij}| \in (1,2)$ and include each $W_{ij}$ in $\Psi(W)$. It will be convenient at times to refer to the collection $\{W_i\}_{i=1}^M$ obtained prior to subdivision in the construction above, which we call the \textit{minimal decomposition} of $T^{-1}(W)$ into line segments.
	
	\begin{remark}\label{rem:psin}
		Having defined $\Psi(W)$ as above, it is natural and often necessary to consider subsequent generations of line segments that provide a decomposition of $T^{-k}(W)$ for $k \ge 2$. These are defined by setting $\Psi_1(W) = \Psi(W)$ and 
		\begin{equation}\label{eq:generationrem}
			\Psi_k(W) = \bigcup_{W_i \in \Psi_{k-1}(W)} \Psi(W_i) \quad \forall k \ge 2.
		\end{equation}
		Since our proof of Proposition~\ref{prop:UniformMixing} is based on establishing one-step estimates for the transfer operator, we will not need these generations now, but they will play a crucial role in Section~\ref{sec:projected}.
	\end{remark}
	
	For $W \in \Sigma$, we further decompose $\Psi(W)$ into short and long curves by defining  
	\begin{equation}\label{eq:shortdef1}
		S(W) = \{W_i \in \Psi(W): |W_i| < 1\}
	\end{equation}
	and
	\begin{equation} 
		L(W) = \Psi(W) \setminus S(W) = \{W_i \in \Psi(W):|W_i| \in [1,2]\}.
	\end{equation}
	The following lemma describes some properties of $\mathcal{S}^-$ and $S(W)$ that are clear from the inspection of Figure~\ref{fig:backward8}. We write $\# U$ to denote the cardinality of a finite set $U$.
	
	\begin{lemma} \label{lem:complexityfact1}
		The statements below hold for all $W \in \Sigma$ and $\alpha$ sufficiently large.
		\begin{itemize}
			\item[(a)] If $|W| \le (4\alpha)^{-1}$, then $\# (W \cap \mathcal{S}^-) \le 3$.
			\vspace{0.1cm}
			\item[(b)] $\#S(W) \le 10$.
		\end{itemize}
	\end{lemma}
	
	\begin{proof}
		The angle between any two directions in $C_s$ and $C_u$ tends to $\pi/2$ as $\alpha \to \infty$, and so Part (a) follows from the fact that the perpendicular distance between the adjacent parallel line segments in Figure~\ref{fig:backward8} is $(2\sqrt{1+\alpha^2})^{-1} > (4\alpha)^{-1}$. We now prove (b). Let $\mathcal{H} = \{y=0\} \cup \{y=1/2\}$ denote the two horizontal lines contained in $\mathcal{S}^-$. Fix $W \in \Sigma$ and let $\{W_i\}_{i=1}^I \subseteq \Sigma$ be the decomposition of $W \setminus \mathcal{H}$ into its connected components. Note that since $|W| \le 2$, we have $I \le 5$. Moreover, from $\mathcal{H} \subseteq \mathcal{S}^-$ and the construction of $\Psi(W)$ it follows that
		\begin{equation} \label{eq:shortcountsum}
			\# S(W) = \sum_{i=1}^I \# S(W_i).
		\end{equation}
		We now estimate $\# S(W_{i_0})$ for a fixed $1 \le i_0 \le I$. Let $\mathcal{S}^-_\ell = \mathcal{S}^{-} \cap \{0 < y < 1/2\}$ and $\mathcal{S}^-_u = \mathcal{S}^{-} \cap \{1/2 < y < 1\}$. For $U \subseteq \T^2$, we define $d_{C_s}(U)$ as the minimal stable-curve length needed to connect two distinct points of $U$. Precisely,
		\begin{equation}\label{eq:unstabledistance}
			d_{C_s}(U) = \inf_{\substack{z_1,z_2 \in U \\ z_1 \neq z_2}}d_s(z_1,z_2), \quad \text{where} \quad d_s(z_1,z_2) = \inf\{|W|: W\in \Sigma, \text{ } z_1,z_2 \in W\}.
		\end{equation}
		We then have 
		\begin{equation} \label{eq:Sunstabledistance}
			d_{C_s}(\mathcal{S}_{u}^-) \ge \frac{1}{2\sqrt{1+\alpha^2}}, \quad d_{C_s}(\mathcal{S}_{\ell}^-) \ge \frac{1}{2\sqrt{1+\alpha^2}}.
		\end{equation}
		By the definition of $\{W_i\}_{i=1}^I$, only one of $W_{i_0} \cap \mathcal{S}_u^-$ and $W_{i_0} \cap \mathcal{S}_\ell^-$ can be nonempty. Thus, by \eqref{eq:Sunstabledistance} and \eqref{eq:backwardinvariant}, all except possibly two segments $V$ in the minimal decomposition of $T^{-1}(W_{i_0})$ satisfy
		$$ |V| \ge \frac{1}{2\sqrt{1+\alpha^2}} \frac{\alpha^2}{2} \ge \frac{\alpha}{8} \ge 1. $$
		Consequently, $\# S(W_{i_0}) \le 2$, with the short segments occurring only from intersections with $\mathcal{S}^-$ near the ends of $W_{i_0}$. Since $I \le 5$ and $i_0$ was arbitrary, $\# S(W) \le 10$ now follows from \eqref{eq:shortcountsum}.
	\end{proof}
	
	By construction, $D T$ exists and is constant along each $W_i \in \Psi(W)$. We can then define the \textit{stable Jacobian} of $T$ along $W_i$ by 
	\begin{equation} \label{eq:StableJacobian}
		|J_{W_i} T|
		:= |(D T)(z) v_i|,
	\end{equation}
	where $z$ is any point on $W_i$ and $v_i \in C_s$ is the unit tangent vector that defines $W_i$. With the definition \eqref{eq:StableJacobian}, for any $W \in \Sigma$, $f \in W^{1,\infty}(\T^2)$, and $\varphi \in C^1(W)$ we have
	$$ \int_W \mathcal{L} f \varphi \, \dee m_W = \sum_{W_i \in \Psi(W)} |J_{W_i} T| \int_{W_i} f \varphi \circ T \, \dee m_{W_i}. $$
	It will thus be necessary to estimate $\sum_{W_i \in \Psi(W)}|J_{W_i}T|$ and related quantities in the proofs of Propositions~\ref{prop:StableUnstable} and~\ref{prop:OneStep}. The bounds we require are recorded in the following lemma.
	
	\begin{lemma} \label{lem:ComplexityOneStep}
		Fix $W \in \Sigma$ and let $0 \le \eta \le 1-\beta$. Then, 
		\begin{equation} \label{eq:ComplexityOneStep1}
			\sum_{W_i \in S(W)} |W_i|^\eta|J_{W_i} T| \leqc \alpha^{-2\beta}|W|^\eta
		\end{equation}
		and
		\begin{equation} \label{eq:ComplexityOneStep2}
			\sum_{W_i \in \Psi(W)} |W_i|^\eta |J_{W_i} T| \leqc \alpha^{-2\beta}|W|^\eta + |W|.
		\end{equation}
	\end{lemma}
	
	\begin{proof}
		By \eqref{eq:forwardcontract0}, for every $W_i \in \Psi(W)$ we have \begin{equation}\label{eq:forwardcontract}
			|J_{W_i} T| \le \left(\frac{\alpha^2}{2}\right)^{-1}.
		\end{equation}
		Using the identity $|W_i| |J_{W_i} T| = |T(W_i)|$, \eqref{eq:forwardcontract}, H\"{o}lder's inequality, and Lemma~\ref{lem:complexityfact1}(b) we obtain
		\begin{align} 
			& \sum_{W_i \in S(W)}|W_i|^\eta |J_{W_i}T| =\sum_{W_i \in S(W)} |T(W_i)|^\eta |J_{W_i} T|^{1-\eta} \le \left(\frac{\alpha^2}{2}\right)^{\eta-1}  \sum_{W_i \in S(W)} |T(W_i)|^\eta \\ 
			& \qquad \le \left(\frac{\alpha^2}{2}\right)^{\eta-1} 10^{1-\eta} \left(\sum_{W_i \in S(W)} |T(W_i)|\right)^\eta \le \left(\frac{\alpha^2}{20}\right)^{-\beta} |W|^\eta \leqc \alpha^{-2\beta}|W|^\eta,
		\end{align}
		which proves \eqref{eq:ComplexityOneStep1}. Next, observe that 
		\begin{equation}\label{eq:ComplexityOneStep3}
			\sum_{W_i \in L(W)}|W_i|^\eta |J_{W_i}T| = \sum_{W_i \in L(W)} |W_i|^{\eta-1} |J_{W_i}T||W_i| \le \sum_{W_i \in L(W)} |T(W_i)| \le |W|,
		\end{equation}
		where in the first inequality we used that $ |W_i|^{\eta-1} \leq 1$ since $\eta-1 < 0$ and $|W_i| \ge 1$ for each $W_i \in L(W)$. Because
		\begin{equation} 
			\sum_{W_i \in \Psi(W)} |W_i|^\eta |J_{W_i} T| = \sum_{W_i \in L(W)}|W_i|^\eta |J_{W_i}T| + \sum_{W_i \in S(W)}|W_i|^\eta |J_{W_i}T|,
		\end{equation}
		the bounds \eqref{eq:ComplexityOneStep3} and \eqref{eq:ComplexityOneStep1} together imply \eqref{eq:ComplexityOneStep2}.
	\end{proof}
	
	\subsection{Proofs of contraction estimates} \label{sec:contractionproofs}
	
	In this section, we prove Propositions~\ref{prop:StableUnstable} and~\ref{prop:OneStep}.
	
	\subsubsection{Estimate of the stable norms}
	
	\begin{proof}[Proof of Proposition~\ref{prop:StableUnstable}(a)]
		Let $f \in W^{1,\infty}(\T^2)$, $W \in \Sigma$, and $\varphi \in C^\beta(W)$ satisfy $|\varphi|_{C^\beta(W)} \le |W|^{-p}$. For each $W_i \in \Psi(W)$, we define 
		\begin{equation} 
			\bar{\varphi_i} \circ T = \frac{1}{|W_i|} \int_{W_i} \varphi \circ T \, \dee m_{W_i}.
		\end{equation}
		Then, we have
		\begin{align}
			\int_{W} \mathcal{L} f \varphi \dee m_W & = \sum_{W_i \in \Psi(W)} |J_{W_i} T|\int_{W_i} f \varphi \circ T \, \dee m_{W_i} \\ 
			& = \sum_{W_i \in \Psi(W)}|J_{W_i}T| \left(\int_{W_i} f (\varphi - \bar{\varphi_i})\circ T \, \dee m_{W_i} + \int_{W_i} f \bar{\varphi_i}\circ T \, \dee m_{W_i}\right) \\ 
			&:=I_1 + I_2\label{eq:Stable1}
		\end{align}
		We will bound $I_1$ using the strong stable norm, and therefore must estimate $|(\varphi - \bar{\varphi_i})\circ T|_{C^\beta(W_i)}$. To this end, first observe that since $\bar{\varphi_i}\circ T$ is constant on $W_i$, by \eqref{eq:forwardcontract0} we have 
		\begin{align}
			[(\varphi-\bar{\varphi_i})\circ T)]_{\beta,W_i} & = \sup_{z, z_2 \in W_i} \frac{|\varphi\circ T(z_1) - \varphi \circ T(z_2)|}{d_{W_i}(z_1,z_2)^\beta} \\ 
			& = \sup_{z_1, z_2 \in W_i} \frac{|\varphi\circ T(z_1) - \varphi \circ T(z_2)|}{d_{W}(T(z_1),T(z_2))^\beta}\left(\frac{d_{W}(T(z_1),T(z_2))}{d_{W_i}(z_1,z_2)}\right)^\beta \\
			& \le |\varphi|_{C^\beta(W)}\left(\frac{\alpha^2}{2}\right)^{-\beta}. \label{eq:Stable2}
		\end{align}
		Moreover, 
		\begin{align}
			\sup_{z \in W_i}|(\varphi - \bar{\varphi_i})\circ T(z)| & = \sup_{z \in W_i} \frac{1}{|W_i|} \left|\int_{W_i} (\varphi \circ T(z) - \varphi \circ T(\bar{z}))\dee m_{W_i}(\bar{z})\right| \\
			& \le \sup_{z,\bar{z} \in W_i}|\varphi \circ T(z) - \varphi \circ T(\bar{z})| \le |\varphi|_{C^\beta(W)} \sup_{z,\bar{z} \in W_i}d_{W}(T(z),T(\bar{z}))^{\beta} \\ 
			& \le |W_i|^\beta \left(\frac{\alpha^2}{2}\right)^{-\beta} |\varphi|_{C^\beta(W)} \le 2\left(\frac{\alpha^2}{2}\right)^{-\beta}|\varphi|_{C^\beta(W)}. \label{eq:Stable3}
		\end{align}
		Combining \eqref{eq:Stable2} and \eqref{eq:Stable3} we have shown that 
		\begin{equation} \label{eq:AveError1}
			|(\varphi - \bar{\varphi_i})\circ T|_{C^\beta(W_i)} \le 3|W|^{-p} \left(\frac{\alpha^2}{2}\right)^{-\beta}. 
		\end{equation}
		Therefore, by the definition of the strong stable norm and \eqref{eq:ComplexityOneStep2} applied with $\eta = p$, we have 
		\begin{align} 
			I_1 = \sum_{W_i \in \Psi(W)} |J_{W_i} T| \int_{W_i} f (\varphi - \bar{\varphi_i}) \circ T \, \dee m_{W_i} &\le 3 |W|^{-p} \left(\frac{\alpha^2}{2}\right)^{-\beta}\|f\|_{s} \sum_{W_i \in \Psi(W)} |W_i|^p |J_{W_i}T| \\ 
			& \leqc \alpha^{-2\beta}\|f\|_{s}. \label{eq:I1}
		\end{align}
		For $I_2$, first observe that since $\bar{\varphi_i}\circ T$ is constant with 
		$$|\bar{\varphi_i}\circ T| \le \sup_{z \in W_i}|\varphi \circ T(z)|\le |\varphi|_{C^\beta(W)} \le |W|^{-p},$$ we have
		\begin{align} 
			I_2 = \sum_{W_i \in \Psi(W)} |J_{W_i} T| &\int_{W_i} f \bar{\varphi_i}\circ T \, \dee m_{W_i} \le |W|^{-p} \sum_{W_i \in \Psi(W)} |J_{W_i} T| \left|\int_{W_i} f \, \dee m_{W_i}\right| \\ 
			& = |W|^{-p} \sum_{W_i \in S(W)} |J_{W_i} T| \left|\int_{W_i} f \, \dee m_{W_i}\right| + |W|^{-p} \sum_{W_i \in L(W)} |J_{W_i} T| \left|\int_{W_i} f \, \dee m_{W_i}\right| \\ 
			& \le |W|^{-p} \|f\|_s \sum_{W_i \in S(W)} |W_i|^p |J_{W_i}T| + |W|^{-p} |f|_w \sum_{W_i \in L(W)} |J_{W_i} T|.
		\end{align} 
		Applying \eqref{eq:ComplexityOneStep1} with $\eta = p$ in the first sum and \eqref{eq:ComplexityOneStep3} with $\eta = 0$ to the second sum, we obtain 
		\begin{equation} \label{eq:I2}
			I_2 \leqc \alpha^{-2\beta}\|f\|_{s} + |W|^{1-p}|f|_w \leqc \alpha^{-2\beta}\|f\|_{s} + |f|_w.
		\end{equation}
		Putting \eqref{eq:I1} and \eqref{eq:I2} into \eqref{eq:Stable1} completes the proof.
	\end{proof}
	
	We next prove Proposition~\ref{prop:StableUnstable}(b). Let us first note that the bound would be immediate from the computations above if $f \circ \phi_{t,0}^\alpha$ were replaced with $\mathcal{L}_\alpha f$. Indeed, by \eqref{eq:ComplexityOneStep2} applied with $\eta = 0$, for $W \in \Sigma$ and $\varphi \in C^1(W)$ we have 
	\begin{align}
		\int_W \mathcal{L}f \varphi \, \dee m_W & = \sum_{W_i \in \Psi(W)} |J_{W_i} T| \int_{W_i} f \varphi \circ T \, \dee m_{W_i} \\ 
		& \leqc |f|_w \sup_{W_i \in \Psi(W)} |\varphi \circ T|_{C^1(W_i)} \sum_{W_i \in \Psi(W)} |J_{W_i} T| 
		\leqc |f|_w \sup_{W_i \in \Psi(W)} |\varphi \circ T|_{C^1(W_i)}.
	\end{align}
	The same argument used to establish \eqref{eq:Stable2} gives 
	$$[\varphi \circ T]_{1,W_i} \le |\varphi|_{C^1(W)} \left(\frac{\alpha^2}{2} \right)^{-1},$$
	so that $\sup_{W_i \in \Psi(W)}|\varphi \circ T|_{C^1(W_i)} \leqc |\varphi|_{C^1(W)}$ and $|\mathcal{L}f|_w \leqc |f|_w$ follows. The main point in the proof below is that since $t-1/2 \ge \alpha^{-1/2}$, the map $\phi_{t,0}^\alpha$ is still a composition of two shears that have large amplitude. While the amplitudes are not the same, this ensures that $\phi_{t,0}^\alpha$ has the same essential features as $T_\alpha^{-1}$.
	
	\begin{proof}[Proof of Proposition~\ref{prop:StableUnstable}(b)]
		Fix $t \in [1/2 + \alpha^{-1/2}, 1]$ and let $\tilde{\alpha} = 2\alpha(t-1/2) \ge 2\sqrt{\alpha} \gg 1$. Then, we have 
		$$\phi_{t,0}^\alpha = T_1^{-1} \circ \widetilde{T}_2^{-1},$$
		where 
		$$\widetilde{T}_2(x,y) = \begin{pmatrix}
			x+\tilde{\alpha}|y-1/2| \\ 
			y
		\end{pmatrix}.
		$$   
		The singularity set for $\phi_{t,0}^\alpha$ is given by 
		$$\widetilde{\mathcal{S}}^- = \{y=0\}\cup\{y=1/2\}\cup \widetilde{T}_2(\{x=0\}) \cup \widetilde{T}_2(\{x=1/2\}) $$
		and has the same qualitative features as $\mathcal{S}^{-}$. Moreover, one can check that
		\begin{equation}
			(D\phi_{0,t}^\alpha)(z)C_s \subseteq C_s \quad \forall z \in \T^2 \setminus \widetilde{\mathcal{S}}^-, \qquad \inf_{v \in C_s} \inf_{z \in \T^2 \setminus\widetilde{\mathcal{S}}^-} \frac{|(D\phi_{0,t}^\alpha)(z)v|}{|v|} \ge \frac{\alpha^{3/2}}{2}
		\end{equation}
		for $\alpha$ sufficiently large. Following the same construction of $\Psi(W)$ for $W \in \Sigma$, one can define a decomposition of $\phi_{t,0}^\alpha(W)$ that is adapted to the singularity set $\widetilde{S}^-$. The same arguments from Lemmas~\ref{lem:complexityfact1} and~\ref{lem:ComplexityOneStep} apply to show that analogue of \eqref{eq:ComplexityOneStep2} (for $\eta = 0$) holds with constants independent of $t$, and Proposition~\ref{prop:StableUnstable}(b) then follows from the argument immediately preceding this proof. 
	\end{proof}
	
	\subsubsection{Estimate of strong unstable norm}
	
	In this section, we will prove Proposition~\ref{prop:StableUnstable}(c). To estimate
	\begin{equation} \label{eq:unstableLHS}
		d_{\Sigma}(W_1,W_2)^{-p} \left|\int_{W_1} \mathcal{L} f \dee m_{W_1} - \int_{W_2} \mathcal{L}f \dee m_{W_2}\right| 
	\end{equation}
	for general $(W_1,W_2) \in \Sigma_u$ satisfying
	$d_{\Sigma}(W_1,W_2) \le \alpha^{-1}$, it is convenient to
	first treat the case where $|W_i| \le (4\alpha)^{-1}$. To this
	end, we begin with a technical lemma that provides us with
	useful decomposition of an admissible pair $(W_1,W_2) \in
	\Sigma_u$ satisfying $|W_1| = |W_2| \le
	(4\alpha)^{-1}$. We will decompose $(W_1,W_2)$ into ``well-paired'' segments $(W_{1,j},W_{2,j})$, which coincide up to translation along the local stable direction of $DT^{-1}$, and remainder segments $V_{i,j}$ whose total length is $\mathcal{O}(d_\Sigma(W_1,W_2))$.
	
	\begin{lemma} \label{lem:pairdecomp}
		Let $(W_1,W_2) \in \Sigma_u$ be such that $|W_1| = |W_2| \le (4\alpha)^{-1}$. Set $d_{\Sigma}(W_1,W_2) = \epsilon \le \alpha^{-1}$. For $\alpha$ sufficiently large, there exists a pair of decompositions
		\begin{equation} \label{eq:pairdecomp1}
			W_i = \bigcup_{j=1}^n W_{i,j} \cup \bigcup_{j=1}^{m_i} V_{i,j}, \qquad i=1,2
		\end{equation}
		with the following properties:
		\begin{itemize}
			\item $m_1, m_2, n \le 7$;
			\vspace{0.15cm}
			\item Each pair $(W_{1,j},W_{2,j})$ is in $\Sigma_u$ and $\sup_{1\le j \le n} d_{\Sigma}(W_{1,j},W_{2,j}) \le 2\epsilon$;
			\vspace{0.1cm}
			\item For each $1 \le j \le n$, their exist lifts $\widetilde{W}_{1,j},
			\widetilde{W}_{2,j} \subseteq \R^2$ of   ${W}_{1,j}$ and
			${W}_{2,j}$ from $\T^2$ to
			its covering space $\R^2$
			such that
			$\widetilde{W}_{1,j}$ and
			$\widetilde{W}_{2,j}$ form opposite sides of a parallelogram $\mathcal{P}_j$ for which $\pi (\mathcal{P}_j)$ (where $\pi:\R^2 \to \T^2$ denotes the canonical projection) is contained in a single connected component $\mathcal{C}$ of $\T^2 \setminus \mathcal{S}^-$ where $D T^{-1}$ is constant. Moreover, the two remaining sides of $\mathcal{P}_j$ are parallel to the stable eigenvector of $DT^{-1}$ in $\mathcal{C}$;
			\vspace{0.1cm}
			\item For each $i = 1,2$, the remainder pieces $V_{i,j}$ belong to $\Sigma$ and $\sup_{1 \le j\le m_i}|V_{i,j}| \le C \epsilon$ for a constant $C \ge 1$ independent of $(W_1,W_2)$, $\alpha$, and $\epsilon$. 
		\end{itemize}
	\end{lemma}
	
	\begin{proof}
		For $z \in W_1 \setminus \mathcal{S}^-$, let $e_s(z)$ denote a unit eigenvector of $DT^{-1}(z)$ with corresponding eigenvalue inside the unit circle. Recall from Section~\ref{sec:hyperbolic} that $e_s(z) \in C_u$, and hence it is nearly perpendicular to $W_1$ and $W_2$ for $\alpha \gg 1$. We fix the sign of $e_s(z)$ so that, if the line $t \mapsto z+t e_s(z)$ meets $W_2$ for $|t|\le 2\epsilon$, then the intersection occurs for $t>0$. Let $\gamma_z(t) = z+t e_s(z)$ and define
		$$ I = \{z \in W_1\setminus \mathcal{S}^-: \exists t_0 \in (0,2\epsilon] \text{ such that }\gamma_z(t_0) \in W_2 \text{ and } \gamma_z([0,t_0]) \cap \mathcal{S}^- = \emptyset\}. $$
		Thus, $I$ consists of the points $z \in W_1 \setminus \mathcal{S}^-$ such that $t \mapsto \gamma_z(t)$ intersects $W_2$ before $\mathcal{S}^-$. The pairs $(W_{1,j},W_{2,j})$ will be obtained by matching the connected components of $I \subseteq W_1$ with the associated intersections of $\gamma_z$ with $W_2$.
		
		To pair segements as described above, it is necessary to estimate the size of $W_1 \setminus I$. To this end, let $$K = (W_1 \cap \mathcal{S}^-) \cup \partial W_1,$$ where $\partial W_1$ is the set of endpoints of $W_1$. We claim that there exists a constant $C \ge 1$, independent of $(W_1,W_2)$, $\epsilon$, and $\alpha$, such that 
		\begin{equation} \label{eq:Icompbound}
			z \in W_1 \setminus I \implies d_{W_1}(z,K) \le C \epsilon,
		\end{equation}
		where as earlier $d_{W_1}$ denotes the arclength distance along $W_1$. Indeed, if $z \in W_1 \setminus I$, then either $\gamma_z([0,2\epsilon]) \cap W_2 = \emptyset$ or $\gamma_z([0,2\epsilon]) \cap \mathcal{S}^{-} \neq \emptyset$. If the former holds, then $d_\Sigma(W_1,W_2) \le \epsilon$ and $\gamma_z' = e_s(z) \in C_u$ together imply that $d_{W_1}(z,\partial W_1) \leqc \epsilon$. For the latter case, first observe that any point $z_0 \in \mathcal{S}^-$ lies on some piecewise linear, continuous curve $L_{z_0}$ that is contained in $\mathcal{S}^-$ and connects the two horizontal edges of the fundamental domain $[0,1)^2$ (see Figure~\ref{fig:backward8}). If $t \mapsto \gamma_z(t)$ intersects $\mathcal{S}^-$ at some time $t_0 \in [0,2\epsilon]$, then we can trace back along $L_{\gamma_z(t_0)}$ from $\gamma_z(t_0)$ to its potential intersection with $W_1$. When such an intersection exists we necessarily have $d_{W_1}(z, W_1 \cap \mathcal{S}^-) \leqc \epsilon$, while if there is no intersection then $d_{W_1}(z, \partial W_1) \leqc \epsilon$. This proves \eqref{eq:Icompbound}. 
		
		From Lemma~\ref{lem:complexityfact1}(a), we have $\#K \le 5$ and hence $K = \{z_j\}_{j=1}^{m_1} \subseteq W_1$ for some $m_1 \le 5$. We then define 
		\begin{equation} \label{eq:VijDef}
			V_{1,j} = \{z \in W_1: d_{W_1}(z,z_j) \le C \epsilon\},
		\end{equation}
		where $C$ is as in \eqref{eq:Icompbound}, and let $\{W_{1,j}\}_{j=1}^n$ be the set of connected components of $W_1 \setminus \cup_{j=1}^m V_{1,j}$. Observe that since $m_1 \le 5$, we have $n \le 6$. To construct the associated decomposition of $W_2$, we define pairs $(W_{1,j},W_{2,j})\in \Sigma_u$ by using that \eqref{eq:Icompbound} implies that each $W_{1,j}$ is contained in $I$. In particular, for every $z \in W_{1,j}$, there exists $t_z \in [0,2\epsilon]$ such that $\gamma_z(t_z) \in W_2$ and $\gamma_z([0,t_z]) \cap \mathcal{S}^- = \emptyset$, and hence we may define 
		$$ W_{2,j} = \bigcup_{z \in W_{1,j}} \gamma_z(t_z). $$
		We then let $\{V_{2,j}\}_{j=1}^{m_2}$ be the set of connected components of $W_2 \setminus \cup_{j=1}^n W_{2,j}$ and note that, as before, $n \le 6$ implies $m_2 \le 7$. 
		
		Since $m_1 \le 5$, $n \le 6$, and $m_2 \le 7$, the
		first bullet point in the lemma statement holds. The
		second and third points are immediate from the
		construction above, while the fourth follows from
		\eqref{eq:VijDef} and the fact that
		$\sum_{j=1}^{m_2}|V_{2,j}| = \sum_{j=1}^{m_1}|V_{1,j}|
		\le 5 C \epsilon$. 
	\end{proof}
	
	The estimate of \eqref{eq:unstableLHS} in the case where $|W_1| = |W_2| \le (4\alpha)^{-1}$ is provided by the following lemma. 
	\begin{lemma} \label{lem:UnstableShortCase}
		Let $(W_1,W_2) \in \Sigma_u$ satisfy $|W_1| = |W_2| \le (4\alpha)^{-1}$ and set $d_{\Sigma}(W_1,W_2) = \epsilon \le \alpha^{-1}$ There exists a constant $C \ge 1$ that is independent of $(W_1,W_2)$,  $\alpha \gg 1$, and $\epsilon>0$ such that for every $f \in W^{1,\infty}(\T^2)$, we have
		\begin{equation}
			\frac{1}{\epsilon^p} \left|\int_{W_1} \mathcal{L} f \, \dee m_{W_1} - \int_{W_2} \mathcal{L}f \,\dee m_{W_2}\right| \le C (\alpha^{-2p}\|f\|_{u} + \|f\|_s).
		\end{equation}
	\end{lemma}
	
	\begin{proof}
		Fix $(W_1,W_2) \in \Sigma_u$ with $|W_1| = |W_2| \le (4\alpha)^{-1}$ and let $\{W_{i,j}\}_{i=1}^n \cup \{V_{i,j}\}_{j=1}^{m_i}$ be the decomposition of $W_i$ guaranteed by Lemma~\ref{lem:pairdecomp}. Splitting the integral according to this decomposition, we have 
		\begin{equation} \label{eq:UnstableSum1}
			\begin{aligned}
				\frac{1}{\epsilon^p} \left|\int_{W_1} \mathcal{L}f \, \dee m_{W_1} - \int_{W_2} \mathcal{L}f \,\dee m_{W_2}\right| & \le \sum_{j=1}^n \frac{1}{\epsilon^p}\left|\int_{W_{1,j}} \mathcal{L} f \, \dee m_{W_{1,j}} - \int_{W_{2,j}} \mathcal{L} f \, \dee m_{W_{2,j}}\right| \\ 
				& + \sum_{i=1}^2 \sum_{j=1}^{m_i}\frac{1}{\epsilon^p}\left|\int_{V_{i,j}} \mathcal{L} f \, \dee m_{V_{i,j}}\right|.
			\end{aligned}
		\end{equation}
		We first estimate the terms involving $V_{i,j}$ using the strong stable norm. From Proposition~\ref{prop:StableUnstable}(a), it follows that $\|\mathcal{L} f\|_{s} \leqc \|f\|_s$. Therefore, by the definition of the strong stable norm and the fourth property in Lemma~\ref{lem:pairdecomp}, we have 
		$$ \left|\int_{V_{i,j}} \mathcal{L} f \,  \dee m_{V_{i,j}}\right| \le |V_{i,j}|^p\|\mathcal{L}f\|_s \leqc \epsilon^p \|f\|_s. $$
		Since $m_i \le 7$ for each $i$, we conclude 
		\begin{equation} \label{eq:RemainderEstimate1}
			\sum_{i=1}^2 \sum_{j=1}^{m_i} \frac{1}{\epsilon^p}\left|\int_{V_{i,j}} \mathcal{L} f \, \dee m_{V_{i,j}}\right| \leqc \|f\|_s.
		\end{equation}
		We now turn to the terms in \eqref{eq:UnstableSum1} that involve $W_{i,j}$. Fix $1 \le j \le n$. From the third point in Lemma~\ref{lem:pairdecomp}, there are lifts of $T^{-1}W_{1,j}$ and $T^{-1} W_{2,j}$ to $\R^2$ that form opposite sides of a parallelogram $\mathcal{P}$ whose remaining two sides are tangent to $C_u$ and have length $\leqc \epsilon \alpha^{-2}$. If $|T^{-1}W_{i,j}| \le 2$, then $(T^{-1}W_{1,j}, T^{-1}W_{2,j}):=(U_1,U_2) \in \Sigma_u$ is such that $d_{\Sigma}(U_1,U_2) \leqc \epsilon \alpha^{-2}$ and from the definition of the unstable norm we obtain 
		\begin{align}
			\frac{1}{\epsilon^p}\left|\int_{W_{1,j}} \mathcal{L} f \dee m_{W_{1,j}} - \int_{W_{2,j}} \mathcal{L} f \dee m_{W_{2,j}}\right| & = \frac{1}{\epsilon^p}|J_{U_1} T| \left|\int_{U_1} f \dee m_{U_1} - \int_{U_2} f \dee m_{U_2}\right| \\ 
			& \leqc \frac{1}{\epsilon^p} d_{\Sigma}(U_1,U_2)^p \|f\|_u 
			\leqc \alpha^{-2p}\|f\|_u. \label{eq:Paired1}
		\end{align}
		Above, we noted that $|J_{U_1} T| = |J_{U_2}T|$ since $T^{-1}$ is a single affine map on the parallelogram containing $W_{1,j}$ and $W_{2,j}$. In the event that $|T^{-1}W_{i,j}| > 2$, by splitting $\mathcal{P}$ into parallelograms with the same interior angles and projecting the long sides back down to $\T^2$, we obtain decompositions
		\begin{equation} 
			T^{-1}W_{i,j} = \bigcup_{k} U_{i,k}, \quad i=1,2
		\end{equation}
		such that $d_{\Sigma}(U_{1,k},U_{2,k}) \leqc \epsilon \alpha^{-2}$ and $|U_{i,k}| \in (1,2]$ for each $k$. Using computations similar to those in \eqref{eq:Paired1} and \eqref{eq:ComplexityOneStep3}, we then find
		\begin{align}
			\frac{1}{\epsilon^p}\left|\int_{W_{1,j}} \mathcal{L} f \,\dee m_{W_{1,j}} - \int_{W_{2,j}} \mathcal{L} f \, \dee m_{W_{2,j}}\right| &\le \frac{1}{\epsilon^p}\sum_{k}|J_{U_{1,k}}T|\left|\int_{U_{1,k}} f \, \dee m_{U_{1,k}} - \int_{U_{2,k}} f \,\dee m_{U_{2,k}}\right| \\
			& \leqc \alpha^{-2p}\|f\|_{u}\sum_{k}|J_{U_{1,k}}T| 
			\leqc \alpha^{-2p}\|f\|_{u} |W_{1,j}|. \label{eq:Paired2}
		\end{align}  
		In view of \eqref{eq:Paired1} and \eqref{eq:Paired2}, we have shown that 
		\begin{equation}\label{eq:Paired3}
			\frac{1}{\epsilon^p}\left|\int_{W_{1,j}} \mathcal{L} f \,\dee m_{W_{1,j}} - \int_{W_{2,j}} \mathcal{L} f \, \dee m_{W_{2,j}}\right| \leqc \alpha^{-2p}\|f\|_u.
		\end{equation}
		Putting \eqref{eq:RemainderEstimate1} and \eqref{eq:Paired3} into \eqref{eq:UnstableSum1} and recalling that $n \le 6$ completes the proof.
	\end{proof}
	
	We are now ready to complete the estimate of the unstable norm. Recall from the beginning of Section~\ref{sec:Sminus} that by the \textit{spanning curves} of $\mathcal{S}^{-}$ we mean the orange and green lines in Figure~\ref{fig:backward8} that connect the two vertical sides of the fundamental domain. We will call two spanning curves \textit{adjacent} if they are parallel and no other singularity curve lies between them. Note that in the proof below the notations $U_{i,j}$, $V_{i,j}$, and $W_{i,j}$ are reused and do not have exactly the same meaning as they did above. 
	
	\begin{proof}[Proof of Proposition~\ref{prop:StableUnstable}(c)]
		Let $d_{\Sigma}(W_1,W_2) = \epsilon \le \alpha^{-1}$. By
		Lemma~\ref{lem:UnstableShortCase}, it only remains to bound
		\eqref{eq:UnstableSum1} in the case where $|W_1| = |W_2| >
		(4\alpha)^{-1}$. In this situation, we split $(W_1,W_2)$ into
		pairs $(W_{1,j}, W_{2,j})$ bounded by adjacent spanning curves of $\mathcal{S}^-$, pairs $(U_{1,j}, U_{2,j})$ that satisfy the hypotheses of
		Lemma~\ref{lem:UnstableShortCase}, and short remainder
		pieces $V_{i,j}$. Precisely, we can find decompositions 
		\begin{equation} \label{eq:DecompTriple}
			W_i = \Big(\bigcup_{j=1}^n W_{i,j}\Big) \cup \Big(\bigcup_{j=1}^m U_{i,j} \Big)\cup \Big(\bigcup_{j=1}^{m_i} V_{i,j}\Big), \quad i = 1,2
		\end{equation}
		with the following properties:
		\begin{itemize}
			\item For each $1\le j \le n$, the pair $(W_{1,j}, W_{2,j})$ belongs to $\Sigma_u$, satisfies $d_{\Sigma}(W_{1,j},W_{2,j}) \le 2 \epsilon$, and $W_{i,j}$ is the portion of $W_i$ lying between two adjacent spanning curves of $\mathcal{S}^-$;
			\vspace{0.1cm}
			\item $m, m_1, m_2 \le C$ for an absolute constant $C \ge 1$ independent of $(W_1,W_2)$ and $\alpha$;
			\vspace{0.1cm}
			\item Each pair $(U_{1,j}, U_{2,j})$ is contained in
			$\Sigma_u$ and satisfies $d_{\Sigma}(U_{1,j},
			U_{2,j}) \le 2\epsilon$ and $|U_{1,j}| = |U_{2,j}|
			\le (4\alpha)^{-1}$;
			\vspace{0.1cm}
			\item Each remainder piece $V_{i,j}$ belongs to $\Sigma$ and satisfies $|V_{i,j}| \leqc \epsilon$.
		\end{itemize}
		The ``bulk'' of $(W_1,W_2)$ is comprised of the pairs $(W_{1,j},W_{2,j})$, which lie between intersections of $W_i$ with adjacent spanning curves of $\mathcal{S}^-$. The pairs $(U_{1,j}, U_{2,j})$ mainly deal with the portions of $W_i$ that lie inside of the sideways ``V'' shapes of $\mathcal{S}^-$, but they also come from the parts of $W_i$ near its endpoints that may not end on a spanning curve. The remainders $V_{i,j}$ are needed to ensure that $|U_{1,j}| = |U_{2,j}|$ for each $j$. They satisfy $|V_{i,j}| \leqc \epsilon$ because $d_\Sigma(W_1,W_2) = \epsilon$.  A typical pair $(W_1, W_2)$ and some representative elements of its associated decomposition are shown in Figure~\ref{fig:pair} below. Note that the $U_{i,j}$, displayed in yellow in Figure~\ref{fig:pair}, may need to be subdivided to ensure that they have length less than $(4\alpha)^{-1}$.  
		\begin{figure}[h]
			\begin{center}
				\includegraphics[width=12cm, height=8cm]{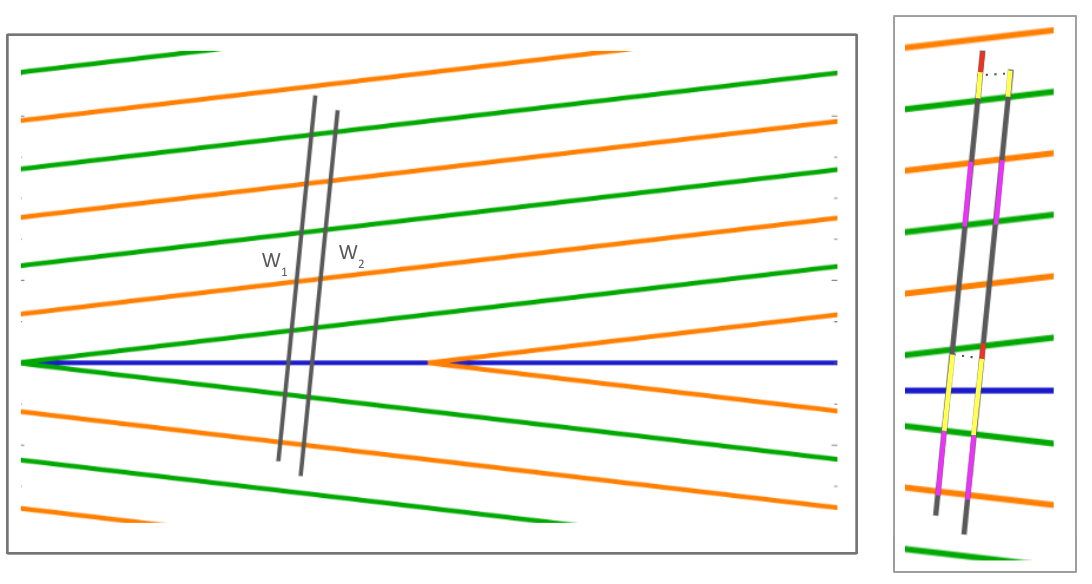}
			\end{center}
			\caption{A typical pair $(W_1,W_2)$ is shown on the left and some representative elements of the decomposition \eqref{eq:DecompTriple} are colored on the right. Two pairs $(W_{1,j}, W_{2,j})$ are shown in pink. The segments in each pink pair have length on the order of $\alpha^{-1}$, while the distance between them is  on the order of $\epsilon$. The vertical coordinates of the upper endpoints on the curves differ by roughly $\epsilon \alpha^{-1}$. Examples of the $U_{i,j}$ and $V_{i,j}$ are displayed in yellow and red, respectively.}
			\label{fig:pair}
		\end{figure}
		
		Splitting \eqref{eq:unstableLHS} according to \eqref{eq:DecompTriple}, we have 
		\begin{align}
			\frac{1}{\epsilon^p}\left|\int_{W_1} \mathcal{L}f  \, \dee m_{W_1} - \int_{W_2} \mathcal{L}f \, \dee m_{W_2}\right| & \le \frac{1}{\epsilon^p}\sum_{j=1}^n \left|\int_{W_{1,j}} \mathcal{L}f  \, \dee m_{W_{1,j}} - \int_{W_{2,j}} \mathcal{L}f \, \dee m_{W_{2,j}}\right|\\
			& + \frac{1}{\epsilon^p}\sum_{j=1}^m \left|\int_{U_{1,j}} \mathcal{L}f \,  \dee m_{U_{1,j}} - \int_{U_{2,j}} \mathcal{L}f \, \dee m_{U_{2,j}}\right| \\ 
			& + \frac{1}{\epsilon^p} \sum_{i=1}^2 \sum_{j=1}^{m_j}\left|\int_{V_{i,j}}\mathcal{L}f \, \dee m_{V_{i,j}}\right|.
		\end{align}
		The integrals involving $U_{i,j}$ can be bounded using Lemma~\ref{lem:UnstableShortCase} and those involving $V_{i,j}$ are controlled by the strong stable norm as in \eqref{eq:RemainderEstimate1}. Proceeding this way, it follows that
		\begin{equation} \label{eq:DecompTriple2}
			\frac{1}{\epsilon^p}\left|\int_{W_1} \mathcal{L}f  \dee m_{W_1} - \int_{W_2} \mathcal{L}f \dee m_{W_2}\right| \leqc \alpha^{-2p}\|f\|_u + \|f\|_s + \frac{1}{\epsilon^p}\sum_{j=1}^n \left|\int_{W_{1,j}} \mathcal{L}f  \dee m_{W_{1,j}} - \int_{W_{2,j}} \mathcal{L}f \dee m_{W_{2,j}}\right|,
		\end{equation}
		and thus it remains to estimate the integrals involving $W_{i,j}$. Fix some $1\le j \le n$. Since the orange and green lines in Figure~\ref{fig:pair} arise from the images of $\{x = 0\}$ and $\{x = 1/2\}$ under $T_2$, we see that $T_2^{-1}W_{1,j}$ and $T_2^{-1}W_{2,j}$ form opposite sides of a parallelogram whose remaining two sides are exactly vertical segments of length $\leqc \epsilon \alpha^{-1}$. As $T^{-1} = T_1^{-1} \circ T_2^{-1}$ and both the length and orientation of a vertical segment is preserved by $T_1^{-1}$, it follows that $|T^{-1}W_{i,j}| \approx \alpha$ and we may decompose $T^{-1}W_{i,j}$ into segments $\widetilde{W}_{i,j}$ such that $|\widetilde{W}_{i,j}| \in (1,2]$ and  $d_{\Sigma}(\widetilde{W}_{1,j},\widetilde{W}_{2,j}) \leqc \epsilon \alpha^{-1}$. Proceeding as in the proof of \eqref{eq:Paired2} then results in the bound 
		\begin{equation} \label{eq:Paired4}
			\frac{1}{\epsilon^p}\left|\int_{W_{1,j}} \mathcal{L} f \, \dee m_{W_{1,j}} - \int_{W_{2,j}} \mathcal{L} f \, \dee m_{W_{2,j}}\right| \leqc \alpha^{-p}|W_{1,j}| \|f\|_{u}.
		\end{equation}
		Putting \eqref{eq:Paired4} into \eqref{eq:DecompTriple2} and noting that $\sum_{j=1}^n |W_{1,j}| \le |W_1| \le 2$ completes the proof. 
	\end{proof}

	\subsubsection{Strong to weak decay estimate}
	
	We now turn to the proof of Proposition~\ref{prop:OneStep}. The key
	lemma that we need is stated below. Its proof is based on the fact
	that if $W \in \Sigma$ is bounded by adjacent spanning curves in
	$\mathcal{S}^{-}$, then the image $T^{-1}W$ can be understood
	explicitly, and in particular distributes itself uniformly over one
	half of the torus. 
	
	\begin{lemma} \label{lem:onestepkey}
		Let $W \in \Sigma$ connect two parallel spanning curves. Suppose that $W$ can be decomposed as $W = W_1 \cup W_2$ for line segments $W_1, W_2 \in \Sigma$ that are each bounded by adjacent spanning curves (see Figure~\ref{fig:adjacent}). There exists a constant $C \ge 1$ independent of $W$ and $\alpha \gg 1$ such that for all mean-zero $f \in W^{1,\infty}(\T^2)$ we have 
		\begin{equation} \label{eq:onestepkey}
			\sup_{|\varphi|_{C^1(W)} \le 1} \int_W \mathcal{L} f \varphi \, \dee m_{W} \le C \alpha^{-p} |W| \tnorm{f}.
		\end{equation}
	\end{lemma}
	
	\begin{figure}
		\begin{center}
			\includegraphics[width=10cm, height=6cm]{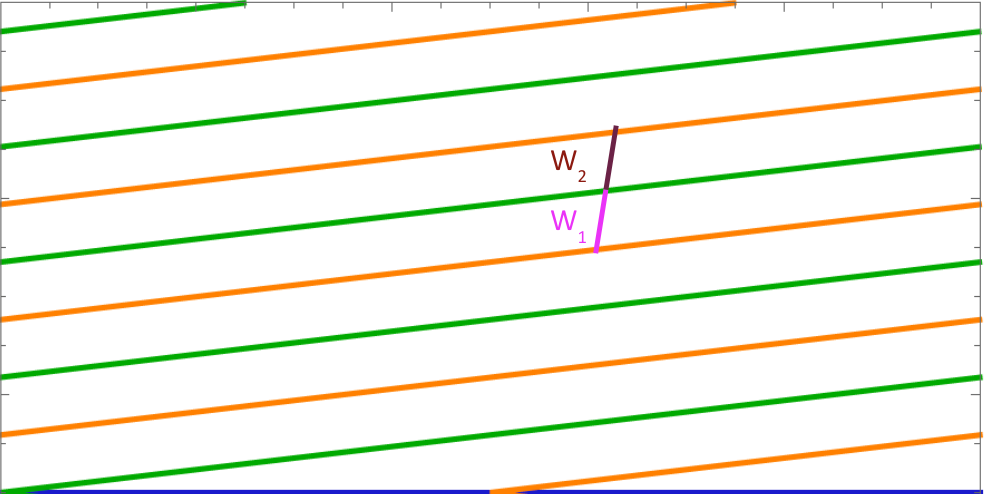}
		\end{center}
		\caption{An example of $W_1$ and $W_2$ as in the statement of Lemma~\ref{lem:onestepkey}. For the configuration shown, $T^{-1}W_1$ is contained in the right half of the fundamental domain and $T^{-1}W_2$ in the left half. The situation would be reversed if the point shared by $W_1$ and $W_2$ were instead on an orange spanning curve.}
		\label{fig:adjacent}
	\end{figure}

	\begin{proof}
		Let $W_1 \cup W_2 = W \in \Sigma$ be as in the statement of the lemma, $f \in W^{1,\infty}(\T^2)$ have zero mean, and $\varphi \in C^1(W)$ satisfy $|\varphi|_{C^1(W)} \le 1$. Define $V_i = T^{-1}(W_i)$ and $V = T^{-1}(W)$. Since $|J_{V_i}T| |V_i| = |W_i| = |W|/2$, we have
		\begin{align} \label{eq:key1}
			\int_{W} \mathcal{L} f \varphi \, \dee m_{W} = \sum_{i=1}^2\int_{W_i} \mathcal{L} f \varphi \, \dee m_{W_i}  & = \sum_{i=1}^2 |J_{V_i} T| \int_{V_i} f \varphi \circ T \, \dee m_{V_i} = \frac{|W|}{2} \sum_{i=1}^2 \dashint_{V_i} f \varphi \circ T \, \dee m_{V_i},
		\end{align} 
		where we have used the notation $\dashint_{V_i} f \varphi \circ T \,  \dee m_{V_i} := |V_i|^{-1} \int_{V_i} f\varphi \circ T \, \dee m_{V_i}.$ Let $$\bar{\varphi} = \frac{1}{|V_1| + |V_2|} \sum_{i=1}^2 \int_{V_i}\varphi \circ T \, \dee m_{V_i},$$ 
		so that by \eqref{eq:key1} we have
		\begin{equation} \label{eq:key1added}
			\left|\int_W \mathcal{L} f \varphi \, \dee m_W \right| \leqc |W| \sum_{i=1}^2 \left|\dashint_{V_i} f \bar{\varphi} \, \dee m_{V_i}\right| + |W| \sum_{i=1}^2 \left|\dashint_{V_i} f (\bar{\varphi} - \varphi \circ T)\, \dee m_{V_i}\right|.
		\end{equation}
		Arguing as in the proof of \eqref{eq:AveError1}, it can be shown that 
		\begin{equation} \label{eq:AveError2}
			|\bar{\varphi} - \varphi \circ T|_{C^1(V_i)} \leqc \alpha^{-1}|\varphi|_{C^1(W)} \le \alpha^{-1}
		\end{equation}
		for each $i = 1,2$. Since $\bar{\varphi}$ is constant and satisfies $|\bar{\varphi}| \le |\varphi|_{C^1(W)} \le 1$, it follows by the definition of the strong stable norm (splitting each $V_i$ into approximately $|V_i|$ admissible segments) and \eqref{eq:key1added} that
		\begin{equation}\label{eq:key2}
			\left| \int_W\mathcal{L} f \varphi \, \dee m_W\right| \leqc |W| \left|\sum_{i=1}^2 \dashint_{V_i} f \, \dee m_{V_i}\right| + |W| \alpha^{-1} \|f\|_s.
		\end{equation}
		
		Since $W_1, W_2 \in \Sigma$ are bounded by adjacent spanning curves, both $T_2^{-1}(W_1)$ and $T_2^{-1}(W_2)$ are line segments tangent to $C_u$ that connect one edge of the fundamental domain with the line $\{x=1/2\}$. Without loss of generality, we assume that the point shared by $W_1$ and $W_2$ lies on a green spanning curve, as in Figure~\ref{fig:adjacent}. Then, $T_2^{-1} (W_1) \subseteq \overline{R}_1$ and $T_2^{-1} (W_2) \subseteq \overline{R}_2$, where $R_1 = \{0 \le x < 1/2\}$ and $R_2 = \{1/2 \le x < 1\}$ denote the left and right halves of the fundamental domain, respectively. As $\overline{R}_1$ is preserved by $T_1^{-1}$, we see that $T_1^{-1}(T_2^{-1}W_1) = V_1 \in \mathcal{W}_s$ is a line segment contained in $\overline{R}_1$ with length approximately $\alpha$ that connects a point on the left edge of the fundamental domain and a point on $\{x = 1/2\}$. A similar statement holds for $V_2$. The images $V_1$ and $V_2$ when $W$ is configured as in Figure~\ref{fig:adjacent} are shown below in Figure~\ref{fig:adjacentimage}. We will prove below that for each $i = 1,2$ we have the bound
		\begin{equation} \label{eq:key3}
			\left|\dashint_{V_i} f \, \dee m_{V_i} - \dashint_{R_i} f(z) \, \dee z\right| \leqc \alpha^{-p}\tnorm{f}.
		\end{equation}
		Combining \eqref{eq:key3} with \eqref{eq:key2} completes the proof since 
		$$ \sum_{i=1}^2 \dashint_{R_i} f(z) \, \dee z = 2 \int_{\T^2} f(z)\, \dee z = 0. $$
		
		\begin{figure}[h]
			\includegraphics[width=7cm, height=7cm]{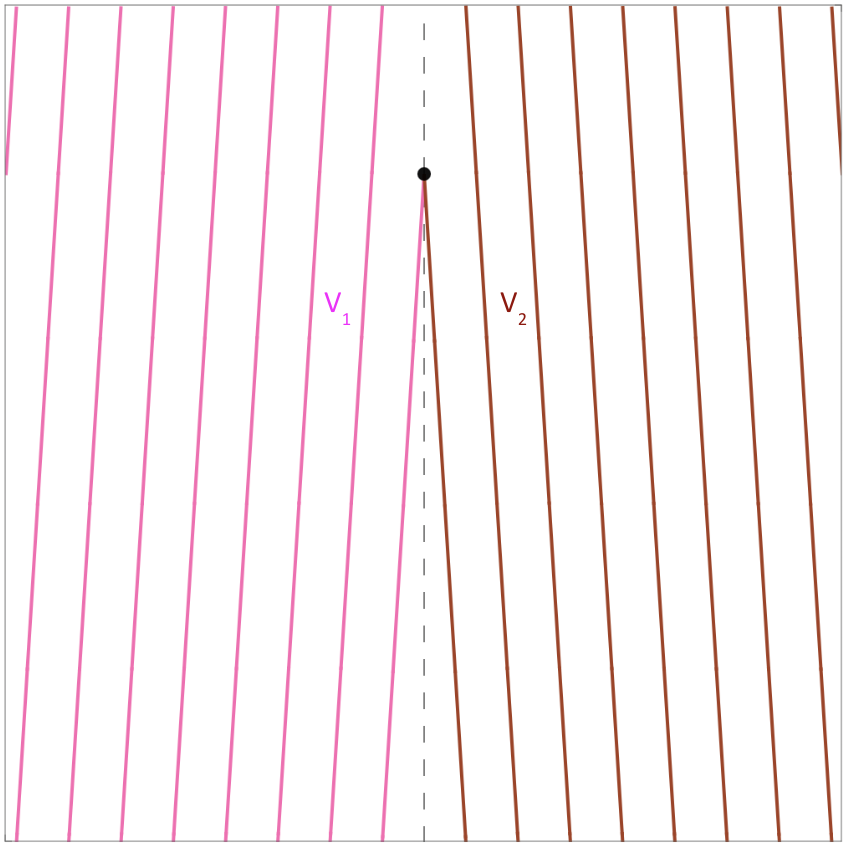}
			\hspace{0.5cm}
			\includegraphics[width=7cm, height=7cm]{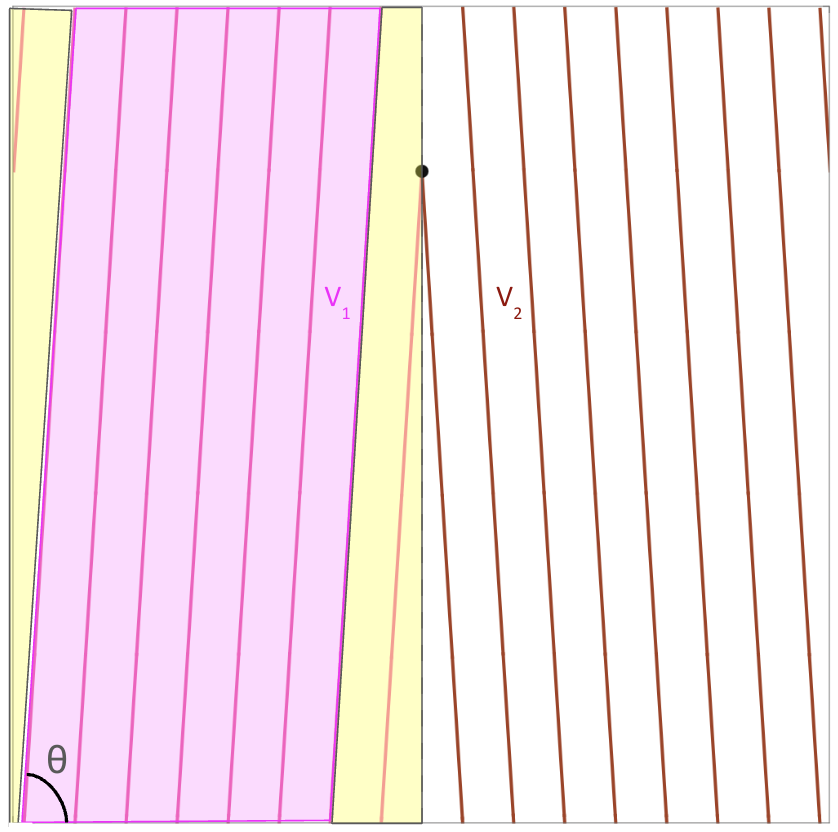}
			\caption{A typical example of the pre-images $V_1 = T^{-1}(W_1)$ and $V_2 = T^{-1}(W_2)$ for $\alpha = 16$ and $W_1$ and $W_2$ configured as in Figure~\ref{fig:adjacent}. Each $V_i$ consists of approximately $\alpha$ segments in $\Sigma$ that span $\T^2$ vertically, together with two segments that may not span. The parallelogram $P$ is shaded in pink and the remainder set $E = R_1 \setminus P$ is shaded in yellow. The angle $\theta$ satisfies $|\pi/2 - \theta| \approx \alpha^{-1}$ since the slope of each pink line is roughly $\alpha$.}
			\label{fig:adjacentimage}
		\end{figure}
		
		It remains to prove \eqref{eq:key3}. We set $i = 1$, as the proof is the same in both cases. We first reduce the proof of \eqref{eq:key3} to the estimate
		\begin{equation}\label{eq:key7}
			\left|\dashint_{V_1} f \dee m_{V_1} - \dashint_P f(z) \dee z\right| \leqc \alpha^{-p}\tnorm{f},
		\end{equation}
		where $P \subseteq \T^2$ is a  parallelogram
		satisfying $|R_1 \setminus P| \leqc \alpha^{-1}$ to be
		defined below. From Figure~\ref{fig:adjacentimage} and the fact that $|V_1| \approx \alpha$, there is $n \approx \alpha$ such that we may decompose $V_1$ as 
		\begin{equation} \label{eq:V1decomp}
			V_1 = \bigcup_{k=0}^{n+1} V_1^k,
		\end{equation}
		where each $V_1^k$, $1 \le k \le n$ is a line segment in $\Sigma$ that spans $\T^2$ vertically (labeled from left to right by increasing $k$) and has slope $\approx \alpha$, $V_1^0 \in \Sigma$ connects the left and top edges of the fundamental domain, and $V_1^{n+1} \in \Sigma$ connects the bottom edge of the fundamental domain and a point on $\{x = 1/2\}.$ Let $P$ be the parallelogram bounded by $V_1^1$, $V_1^n$, and the two horizontal edges of fundamental domain, and define $E = R_1 \setminus P$. Both $P$ and $E$ are displayed in the second image in Figure~\ref{fig:adjacentimage}. There exists $I \subseteq \T$ with $|I| \leqc \alpha^{-1}$ and a family of disjoint line segments $\{W_x\}_{x \in I} \subseteq \Sigma$ such that $E = \cup_{x \in I} W_x$. By Fubini's theorem, we then have 
		\begin{equation} \label{eq:key4}
			\left|\int_E f(z) \dee z\right| = \left| \int_{I} \left(\int_{W_x} f \, \dee m_{W_x}\right) \, \dee x\right|  \le |I| \|f\|_{s} \leqc \alpha^{-1} \|f\|_s.
		\end{equation}
		A similar argument shows that 
		\begin{equation} \label{eq:key5}
			\left|\int_P f(z) \dee z\right| \leqc \|f\|_{s}.
		\end{equation}
		From the identity
		$$  \dashint_{R_1} f(z) \, \dee z  = \dashint_P f(z) \, \dee z - \frac{|E|}{|R_1|} \dashint_P f(z) \, \dee z + \frac{1}{|R_1|} \int_E f(z) \, \dee z$$
		combined with \eqref{eq:key4}, \eqref{eq:key5}, and $|E| \leqc \alpha^{-1}$, we obtain
		\begin{equation} \label{eq:key6}
			\left|\dashint_{R_1} f(z) \, \dee z - \dashint_P f(z) \, \dee z\right| \leqc \alpha^{-1}\|f\|_s.
		\end{equation}
		It follows from \eqref{eq:key6} that \eqref{eq:key7} is enough to complete the proof of \eqref{eq:key3}.
		
		We now turn to the proof of \eqref{eq:key7}. For $1 \le k \le n-1$, let $P_k \subseteq P$ be the parallelogram bounded by $V_1^k$, $V_1^{k+1}$, and the two horizontal edges of the fundamental domain. Then, $|P_k \cap P_j| = 0$ for $k \neq j$ and $P = \cup_{k=1}^{n-1}P_k$. Let $\theta \in [0,\pi/2]$ denote the angle between $V_1^1$ and the line $\{y=0\}$ (see Figure~\ref{fig:adjacentimage}) and $h > 0$ be the distance between the intersections of $V_1^k$ and $V_1^{k+1}$, for $1 \le k \le n-1$, with $\{y=0\}$ and note that $h \leqc \alpha^{-1}$. For fixed $1 \le k \le n-1$ and $0 \le x \le h$, let $W_x = V_1^{k} + (x,0)$. Then, $d_\Sigma(V_1^k,W_x) \le h$ and therefore by the definition of the strong unstable norm we have
		\begin{equation} \label{eq:key8} \left|\int_{V_1^k} f \, \dee m_{V_1^k} - \int_{W_x} f \, \dee m_{W_x}\right| \le h^p \|f\|_u \leqc \alpha^{-p} \|f\|_u.
		\end{equation}
		Since
		\begin{align}\int_{P_k} f(z) \dee z &= \sin \theta \int_0^h \left(\int_{W_x} f \, \dee m_{W_x}\right) \dee x \\ 
			& = h \sin \theta \int_{V_1^k} f \, \dee m_{V_1^k} + \sin \theta \int_0^h \left(\int_{W_x} f \, \dee m_{W_x} - \int_{V_1^k} f \, \dee m_{V_1^k}  \right),    
		\end{align}
		it follows from \eqref{eq:key8} and $h \leqc \alpha^{-1}$ that 
		\begin{equation} \label{eq:key9}
			\left|\int_{P_k} f(z) \dee z - h\sin\theta \int_{V_1^k}f \dee m_{V_1^k} \right| \leqc \alpha^{-(1+p)}\|f\|_{u}.
		\end{equation}
		Summing this bound over $1 \le k \le n-1$ and defining $\widetilde{V}_1 = \cup_{k=1}^{n-1} V_1^k$, we obtain 
		\begin{equation} \label{eq:key10}
			\left| \int_P f(z) \, \dee z - h\sin \theta \int_{\widetilde{V}_1} f \, \dee m_{\widetilde{V}_1} \right| \leqc \alpha^{-p} \|f\|_u.
		\end{equation}
		Notice now that 
		$$|P| = h(n-1) |V_1^1| \sin\theta = h|\widetilde{V}_1| \sin \theta,$$
		and hence dividing \eqref{eq:key10} by $|P|$ gives 
		\begin{equation} \label{eq:key11}
			\left| \dashint_P f(z) \, \dee z - \dashint_{\widetilde{V}_1} f \, \dee m_{\widetilde{V}_1} \right| \leqc \alpha^{-p} \|f\|_u.
		\end{equation}
		Finally, since $\left|\int_{V_1^k} f \, \dee m_{V_1^k}\right| \le \|f\|_s$ for each $k$, $|V_1| \gtrsim \alpha$, and $\left|1-\dfrac{|\widetilde{V}_1|}{|V_1|}\right| \leqc \alpha^{-1}$, it follows from \eqref{eq:key11} and the identity
		\begin{align}
			\dashint_P f(z) \, \dee z - \dashint_{V_1} f \, \dee m_{V_1}  & = \dashint_P f(z) \, \dee z - \dashint_{\widetilde{V}_1} f \, \dee m_{\widetilde{V}_1} \\ 
			& + \left(1-\frac{|\widetilde{V}_1|}{|V_1|}\right)\dashint_{\widetilde{V}_1} f \, \dee m_{\widetilde{V}_1} - \frac{1}{|V_1|}\sum_{k\in \{0,n,n+1\}} \int_{V_1^k} f \, \dee m_{V_1^k} 
		\end{align}
		that 
		$$\left| \dashint_P f(z) \, \dee z - \dashint_{V_1} f \, \dee m_{V_1}\right| \leqc \alpha^{-p}\|f\|_u + \alpha^{-1}\|f\|_s \leqc \alpha^{-p}\tnorm{f}.$$
		This is \eqref{eq:key7} and thus completes the proof. 
	\end{proof}
	
	\begin{proof}[Proof of Proposition~\ref{prop:OneStep}]
		Let $f \in W^{1,\infty}(\T^2)$ have zero mean. Fix $W \in \Sigma$ and $\varphi \in C^1(W)$ with $|\varphi_{C^1(W)} \le 1$. Quite similar to \eqref{eq:DecompTriple}, we may decompose $W$ as 
		$$W = \bigcup_{i=1}^N W_i \cup \bigcup_{i=1}^M V_i,$$
		where each $W_i \in \Sigma$ satisfies the hypotheses of Lemma~\ref{lem:onestepkey}, $M \le C$ for an absolute constant $C$ independent of $\alpha$, and the remainder pieces $\{V_i\}_{i=1}^M \subseteq \Sigma$ satisfy $|V_i| \leqc \alpha^{-1}$. Then, by Lemma~\ref{lem:onestepkey}, we have 
		\begin{align}
			\int_W \mathcal{L} f \varphi \, \dee m_W &= \sum_{i=1}^N \int_{W_i} \mathcal{L} f \varphi \, \dee m_{W_i} + \sum_{i=1}^M \int_{V_i} \mathcal{L} f \varphi \, \dee m_{W_i} \\ 
			& \leqc \alpha^{-p}\tnorm{f}\sum_{i=1}^N |W_i| + M |V_i|^p \|f\|_s \\ 
			& \leqc \alpha^{-p} |W| \tnorm{f} + \alpha^{-p}\|f\|_s \\ 
			& \leqc \alpha^{-p} \tnorm{f}.
		\end{align}
	\end{proof}
	
	\section{Projected exponential mixing} \label{sec:projected}
	
	In this section we will prove Theorem~\ref{prop:Projected}, which recall consists of the mixing estimate
	\begin{equation} \label{eq:projectedrepeated}
		\left|\int_{\T^2} P_{\le N} (f \circ T^{-n})(z) P_{\le N}(g\circ T^{-m})(z) \, \dee z\right| \le C e^{-(m-n)c \log \alpha}\|f\|_{C^1} \|g\|_{C^1}
	\end{equation}
	for all integers $m,n \ge 0$, frequencies $N \ge 1$, and mean-zero $f,g \in C^1(\T^2)$. In Section~\ref{sec:mixingstatements}, we stated the result for $f,g \in W^{1,\infty}(\T^2)$, but by the same density argument used in the proof of Theorem~\ref{prop:UniformMixing} it is sufficient to consider functions in $C^1(\T^2)$. Recall also that in the above
	$$P_{\le N} f(z) = \int_{\T^2} \varphi_N(z-\bar{z}) f(\bar{z}) \dee \bar{z},$$
	where 
	$$ \varphi_N(z) = \sum_{m \in \Z^2} N^2 \varphi(N(z+m))$$
	and $\varphi:\R^2 \to [0,\infty)$ is a smooth, $L^1$-normalized bump function supported in $[-1,1]^2$. 
	
	By the symmetry of \eqref{eq:projectedrepeated}, we may assume without loss of generality that $m > n$. The proof will be split into the two cases $n \le c_0 (m-n)$ and $n >
	c_0(m-n$) for some fixed $c_0 > 0$. A first observation, already discussed briefly in Section~\ref{sec:mixingstatements}, is that the former case is a corollary of
	Theorem~\ref{prop:UniformMixing}. The lemma below makes this precise.
	
	\begin{lemma}\label{lem:easycase}
		There exist $c_0 \in (0,1)$ and $C,c > 0$ such that for all $\alpha$ sufficiently large, $m > n$ satisfying $n \le c_0(m-n)$, and mean-zero $f,g \in W^{1,\infty}(\T^2)$ we have
		\begin{equation}
			\left|\int_{\T^2} P_{\le N} (f \circ T^{-n})(z) P_{\le N}(g\circ T^{-m})(z) \, \dee z\right| \le C e^{-(m-n)c \log \alpha}\|f\|_{W^{1,\infty}} \|g\|_{W^{1,\infty}}.
		\end{equation}
	\end{lemma}
	
	\begin{proof}
		Let $c > 0$ be as in Theorem~\ref{prop:UniformMixing}. Then, 
		\begin{equation} \label{eq:alpha^qdecay}
			\left|\int_{\T^2} f_1 \circ T^{-k}(z) f_2(z) \, \dee z\right| \leqc \alpha^{-k c} \|f_1\|_{W^{1,\infty}}\|f_2\|_{W^{1,\infty}} 
		\end{equation} 
		holds for all $\alpha$ sufficiently large, $k \in \N$, and mean-zero $f_1,f_2 \in W^{1,\infty}(\T^2)$. Since $P_{\le N}$ is convolution against an even, normalized kernel, it is symmetric on $L^2$ and for any $h \in W^{1,\infty}(\T^2)$ we have $\|P_{\le N} h\|_{W^{1,\infty}} \le \|h\|_{W^{1,\infty}}$. Therefore, by \eqref{eq:alpha^qdecay} we have
		\begin{align}
			\left|\int_{\T^2} P_{\le N} (f \circ T^{-n})(z) P_{\le N}(g\circ T^{-m}(z)) \, \dee z\right| & = \left|\int_{\T^2} P^2_{\le N} (f \circ T^{-n})(z) g\circ T^{-m}(z)\,  \dee z\right| \\ 
			& \leqc \alpha^{-mc}\|g\|_{W^{1,\infty}}\|P_{\le N}^2 (f\circ T^{-n})\|_{W^{1,\infty}} \\ 
			& \leqc  \alpha^{-mc} \|g\|_{W^{1,\infty}} \|f \circ T^{-n}\|_{W^{1,\infty}} \\ 
			& \leqc 4^n \alpha^{-mc} \alpha^{2n} \|g\|_{W^{1,\infty}}\|f\|_{W^{1,\infty}},
			\label{eq:trivialcase1}
		\end{align}
		where in the last inequality we noted that 
		$$\|f \circ T^{-n}\|_{W^{1,\infty}} \le (2\alpha)^{2n} \|f\|_{W^{1,\infty}}$$
		by the explicit formula for the matrices in \eqref{eq:matrices}. Now, $n \le c_0 (m-n)$ implies
		$$ n \le \frac{c_0}{1+c_0} m \le \frac{c_0}{2} m. $$
		Thus, taking $c_0 < c/2$ gives $n \le mc/4$, which when put into \eqref{eq:trivialcase1} yields
		$$ \left|\int_{\T^2} P_{\le N} (f \circ T^{-n})(z) P_{\le N}(g\circ T^{-m}(z) \, \dee z\right| \leqc  2^m \alpha^{-\frac{mc}{2}}  \|f\|_{W^{1,\infty}}\|g\|_{W^{1,\infty}} \leqc \alpha^{-\frac{mc}{4}}\|f\|_{W^{1,\infty}}\|g\|_{W^{1,\infty}},$$
		where the last inequality holds for $\alpha$ sufficiently large. This completes the proof since $m \ge m-n$.
	\end{proof}
	
	The remainder of this section is devoted to proving \eqref{eq:projectedrepeated} for $n > c_0(m-n)$. By Lemma~\ref{lem:easycase}, this is sufficient to obtain Theorem~\ref{prop:Projected}. 
	
	\subsection{Discussion of the proof} \label{sec:projecteddiscussion}
	
	When $n > c_0 (m-n)$, the argument in Lemma~\ref{lem:easycase} fails because there is an essential mismatch between the exponents $c$ and $2$ in \eqref{eq:trivialcase1} and we can no longer leverage the smallness $n \ll m$ to absorb the factor of $\alpha^{2n}$. We begin here by elaborating on the challenges in treating the case $n < c_0(m-n)$ and providing an overview of our proof.
	
	Since the goal is to obtain decay in $|m-n|$, it is natural to express \eqref{eq:projectedrepeated} in a way that involves $f \circ T^{-(m-n)}$. This can be done using the symmetry of $P_{\le N}$ and a change of variables:
	\begin{align}
		\int_{\T^2} P_{\le N}(f\circ T^{-m})(z) P_{\le N}(g \circ T^{-n})(z)\, \dee z  & = \int_{\T^2} f\circ T^{-m}(z) P^2_{\le N}(g \circ T^{-n})(z) \, \dee z \\ 
		& = \int_{\T^2} f \circ T^{-(m-n)}(z) (P^2_{\le N}(g \circ T^{-n}) \circ T^n)(z) \, \dee z \\ 
		& = \int_{\T^2} f \circ T^{-(m-n)}(z) g_n(z) \, \dee z, \label{eq:easycase2}
	\end{align}
	where we have defined 
	\begin{equation} \label{eq:Gv2}
		g_n :=P^2_{\le N}(g \circ T^{-n}) \circ T^n.
	\end{equation}
	Equivalently, $g_n$ can be written as 
	\begin{equation}\label{eq:G}
		g_n(z) = \int_{\T^2} \psi_N(T^n(\bar{z})-T^n(z)) g(\bar{z}) \, \dee \bar{z},
	\end{equation}
	where 
	\begin{equation} \label{eq:psiN}
		\psi_N(z) = \int_{\T^2} \varphi_N(z - \bar{z})\varphi_N(\bar{z}) \, \dee \bar{z}  
	\end{equation}
	is the kernel associated with $P_{\le N}^2$. In obtaining \eqref{eq:G}, we have changed variables in the convolution and noted that $\psi_N$ is even because $\varphi_N$ is even.
	
	If we had $\|g_n\|_{W^{1,\infty}} \leqc \|g\|_{C^1}$ uniformly with respect to $n$ and $N$, then \eqref{eq:projectedrepeated} would follow from \eqref{eq:easycase2} and Theorem~\ref{prop:UniformMixing}. Unfortunately, such control on the regularity of $g_n$ is far from true. Indeed, unless $N$ is large enough that the averaging in $P_{\le N}(g \circ T^{-n})$ resolves the small scales of $g \circ T^{-n}$, which would require at least $N^{-1} \leqc \alpha^{-2n}$, then it is clear from \eqref{eq:G} that $|\grad g_n \cdot v|$ will in general grow exponentially fast in $n$ for some vectors $v \in C_u$. Our proof of Theorem~\ref{thm:discrete} only required \eqref{eq:projectedrepeated} for $N^{-1} \leqc \alpha^{-nc}$, where $c$ is the exponent from Theorem~\ref{prop:UniformMixing} (see \eqref{eq:nstar} and \eqref{eq:diagonal3}). However, it is necessary to bridge the gap between $N \sim \alpha^{nc}$ and $N \sim \alpha^{2n}$, and in this intermediate regime it seems impossible to treat \eqref{eq:projectedrepeated} as a corollary of Theorem~\ref{prop:UniformMixing}. This again reflects an inability to match the mixing rate with the worst possible growth rate of $DT^{-n}$. In particular, for $\alpha^{nc} \leqc N \leqc \alpha^{2n}$, the frequency $N$ is simultaneously too small to justify $P_{\le N} (g \circ T^{-n}) \approx g \circ T^{-n}$ and too large to prove that $\|P_{\le N} (g \circ T^{-n})\|\ll \|g\|_{L^2}$.
	
	While $|\grad g_n|$ can certainly be quite large, \eqref{eq:G} suggests that it is possible, at least in principle, for $g_n$ to be regular along admissible curves $W \in T^{-n}(\Sigma) \cap \Sigma$. This motivates decomposing \eqref{eq:projectedrepeated} into an integral along segments in $T^{-n}(\Sigma) \cap \Sigma$ that can be estimated via the decay of $|f \circ T^{-(m-n)}|_w$ provided by \eqref{eq:weakdecay}. In Section~\ref{sec:reduction}, we show that
	\begin{equation} \label{eq:discussionreduction}
		\left|\int_{\T^2} g_n(z) f \circ T^{-(m-n)}(z)\dee z\right| \le \int_{\T} \sum_{W_i \in \Psi_n(W_x)}|J_{W_i}T^n| \left|\int_{W_i} g_n \, f \circ T^{-(m-n)} \dee m_{W_i}\right| \dee x,
	\end{equation}
	where $W_x = \{(x,y) \in \T^2: y \in \T\} \in \Sigma$ is the vertical segment with horizontal coordinate $x$, $\Psi_n(W_x) \subseteq T^{-n}(\Sigma) \cap \Sigma$ is the decomposition of $T^{-n}(W_x)$ defined in Remark~\ref{rem:psin}, and $|J_{W_i} T^n|$ is the stable Jacobian of $T^n$ along $W_i$ (see \eqref{eq:nStableJacobian}). Using ideas from the ``complexity estimates'' in \cite{DemersLiverani2d}, it can be shown that 
	\begin{equation}
		\sup_{x \in \T}\sum_{W_i \in \Psi_n(W_x)} |J_{W_i}T^n| \leqc 1,
	\end{equation}
	and hence \eqref{eq:projectedrepeated} would follow from \eqref{eq:weakdecay} if one had $|g_n|_{C^1(W_i)} \leqc \|g\|_{C^1}$ for every $x \in \T$ and $W_i \in \Psi_n(W_x)$.
	
	An obstruction to the idea above is that, even for $W \in T^{-n}(\Sigma) \cap \Sigma$, it is possible that $|g_n|_{C^1(W)}$ is quite large. One can see this by changing variables in \eqref{eq:G} to rewrite $g_n$ as 
	\begin{equation} \label{eq:Gv3}
		g_n(z) = \int_{\T^2} \psi_N(\bar{z})g(T^{-n}(\bar{z} + T^n(z))) \, \dee \bar{z}.
	\end{equation}
	From \eqref{eq:Gv3}, it follows that $|\grad g_n (z)| \gg \|g\|_{C^1}$ may occur when $\mathcal{U} = \{\bar{z} \in \T^2: d_{\T^2}(\bar{z},z) \leqc N^{-1}\}$ contains points where the unstable eigenvalues of $(DT^{-n})(T^n(z))$ and $(DT^{-n})(\bar{z} + T^n(z))$ differ dramatically. This cannot be ruled out whenever $\mathcal{U}$ contains a point where many singuarity curves of $T^{-n}$ intersect. 
	
	Our proof of \eqref{eq:projectedrepeated} is based on \eqref{eq:discussionreduction} and \eqref{eq:weakdecay}, combined with a suitable decomposition of $\Psi_n(W_x)$ and a careful analysis of the singularity curves of $T^{-n}$. We prove a quantitative bound on the number of points where multiple singularity curves intersect. This allows us to restrict the line integrals in \eqref{eq:discussionreduction} to a ``regular'' portion $W_i^{(r)} \subseteq W_i$ of each curve, at the cost of neglecting a set with measure that vanishes like $\alpha^{-n} \le \alpha^{-c_0(m-n)}$ as $n \to \infty$. Along these regular segments, the growth of $|g_n|_{C^1(W^{(r)}_i)}$ can be controlled in a manner consistent with the estimate
	\begin{equation} \label{eq:gnregdiscussion}
		\sup_{x \in \T}\sum_{W_i \in \Psi_n(W_x)} |J_{W_i}T^n| |g_n|_{C^1(W_i^{(r)})} \leqc \|g\|_{C^1},
	\end{equation}
	which together with \eqref{eq:weakdecay} will be enough to complete the proof.
	
	The remainder of this section is organized as follows. In Section~\ref{sec:nstep}, we define the $n$-step backward singularity set $\mathcal{S}_n^{-}$ and a decomposition of $\Psi_n(W)$ into ``good'' and ``bad'' curves that extends the splitting of $\Psi(W)$ into ``long'' and ``short'' segments used in Section~\ref{sec:Sminus}. We also state our counting estimate, Lemma~\ref{lem:multipoints}, concerning the number of multi-intersection points in $\mathcal{S}_n^{-}$ and prove some preliminary facts that will be needed to control the regularity of $g_n$ away from such points. In Section~\ref{sec:reduction}, we reduce the proof to \eqref{eq:gnregdiscussion}, which is then established in Section~\ref{sec:regularity}. Finally, Lemma~\ref{lem:multipoints} is proven in Section~\ref{sec:multipoints}.
	
	\subsection{Higher order singularity sets and generations of curves}\label{sec:nstep}
	
	\subsubsection{$n$-step singularity set and matrix products} \label{sec:nstepsingularity}
	Recall that $\{\mathcal{C}_i^-\}_{i=1}^4$ partitions $\T^2 \setminus \mathcal{S}^-$ into open, connected components on which $T^{-1}$ has constant derivative. Let us label the inverses of the four possible derivative matrices from \eqref{eq:matrices} as $\{A_i\}_{i=1}^4$, such that $DT^{-1}(z) = A_i$ for $z \in \mathcal{C}_i^-$. For each $k \ge 0$ and $i \in \{1,2,3,4\}$, let $\mathcal{R}_{k,i} = T^k(\mathcal{C}_i^-)$. Then, $\{\mathcal{R}_{k,i}\}_{i=1}^4$ partitions $\T^2 \setminus T^k(\mathcal{S}^-)$ into open, connected components on which $DT^{-1}\circ T^{-k}$ is constant. In particular, $DT^{-1} \circ T^{-k}$ is given by $A_i$ on the set $\mathcal{R}_{k,i}$. We define the \textit{$n$-step backward singularity set} $\mathcal{S}_n^-$ by 
	\begin{equation}\label{eq:nstepsingdef}
		\mathcal{S}_n^- = \bigcup_{k=0}^{n-1}T^k(\mathcal{S}^-).
	\end{equation}
	For a word $\sigma = (\sigma_0, \ldots, \sigma_{n-1}) \in [4]^n:=\{1,2,3,4\}^n$, define 
	\begin{equation}
		\mathcal{C}_\sigma^{-} = \bigcap_{k=0}^{n-1}\mathcal{R}_{k,i}.
	\end{equation}
	These sets are open, but may be either disconnected or empty. By their definition, we have that 
	\begin{equation} \label{eq:nprod}
		DT^{-n}(z) = A_{\sigma_{n-1}} A_{\sigma_{n-2}}\ldots A_{\sigma_0}:= A_\sigma \quad \forall z \in \mathcal{C}_\sigma^-
	\end{equation}
	and it is straightforward to check that 
	\begin{equation}
		\mathcal{S}_n^{-} = \bigcup_{\sigma \in [4]^n} \partial \mathcal{C}_\sigma^-.
	\end{equation}
	
	Before proceeding, we record a lemma regarding matrix products of the form \eqref{eq:nprod}.
	
	\begin{lemma} \label{lem:equivexpansion}
		Fix $n \in \N$ and let $\sigma,\sigma' \in [4]^n$ be words that differ in only one entry. There exists a constant $C_0 \ge 1$ independent of $n$ and $\alpha$ such that for all $v \in C_s$ we have 
		\begin{equation} \label{eq:equivexpansion}
			\frac{|A_\sigma v|}{|A_{\sigma'}v|} \le C_0.
		\end{equation}
	\end{lemma}
	
	\begin{proof}
		Let $0 \le k \le n-1$ be such that $\sigma_m = \sigma_m'$ for $m \neq k$. Define
		$$ w = A_{\sigma_k} \left(\prod_{m = 0}^{k-1} A_{\sigma_m}\right) v \quad \text{and} \quad w' = A_{\sigma'_k} \left(\prod_{m = 0}^{k-1} A_{\sigma_m}\right) v.$$
		Then, $w,w' \in C_s$ and it follows from \eqref{eq:leadingorder} that
		\begin{equation} \label{eq:ww'}
			\frac{1}{2}|w| \le |w'| \le 2 |w|
		\end{equation}
		whenever $\alpha$ is sufficiently large. If $k = n-1$, then there is nothing left to show. If instead $k \le n-2$, then we define $M = \prod_{m={k+1}}^{n-1} A_{\sigma_m}$ and observe that
		\begin{equation}\label{eq:Mratio}
			\frac{|A_\sigma v|}{|A_{\sigma'}v|} = \frac{|M w|}{|Mw'|}.
		\end{equation}
		Being a product of hyperbolic matrices with common forward and backward invariant cones $C_s$ and $C_u$, respectively, $M$ has  eigenvalue and eigenvector pairs $(\lambda_s, e_s)$, $(\lambda_u, e_u)$ that satisfy $e_s \in C_s$, $e_u \in C_u$, $|\lambda_s| \ge (\alpha^2/2)^{n-k-1}$, and $|\lambda_u|  \le (\alpha^2/2)^{-(n-k-1)}$. Let $w = w_s e_s + w_u e_u$ be the decomposition of $w$ in this eigenbasis and similarly write $w' = w_s' e_s + w_u' e_u$. Since $w \in C_s$, we have $|w_u| \leqc \alpha^{-1/2} |w_s|$ and $|w| \approx |w_s|$, as well as the same bounds for $w'$. Thus, 
		$$ \frac{|M w|}{|Mw'|} \le \frac{|\lambda_s| |w_s| + |\lambda_u| |w_u|}{|\lambda_s| |w'_s| - |\lambda_u| |w'_u|} \leqc \frac{|\lambda_s| |w_s|}{|\lambda_s| |w_s'|} \leqc \frac{|w|}{|w'|} \leqc 1, $$ 
		where in the last inequality we used \eqref{eq:ww'}. In view of \eqref{eq:Mratio}, this completes the proof.
	\end{proof}
	
	\subsubsection{$n$-step generations of curves} \label{sec:nstepgen}
	Recall the definition of $\Psi(W)$ from Section~\ref{sec:Sminus}, which decomposes $T^{-1}(W)$, up to finitely many removed points, into a family of line segments $\{W_i\} \subseteq \Sigma$. The construction of $\Psi(W)$ guarantees that for each $W_i$, there is a unique $j \in \{1,2,3,4\}$ such that  $T(W_i) \subseteq \mathcal{C}_j^{-}$. Let $\Psi_0(W) = \{W\}$ and for $k \ge 1$ define 
	\begin{equation}
		\Psi_k(W) = \bigcup_{W_i \in \Psi_{k-1}(W)} \Psi(W_i).
	\end{equation}  
	Then, for $W \in \Sigma$ and $n \in \N$, $\Psi_n(W)$ decomposes $T^{-n}(W)$ into a collection $\{W_i\}\subseteq \Sigma$ with the property that for each $W_i$ there is a unique word $\sigma \in [4]^n$ such that $T^n(W_i) \subseteq \mathcal{C}_\sigma^-$. Generalizing \eqref{eq:StableJacobian}, we define
	\begin{equation}\label{eq:nStableJacobian}
		|J_{W_i}T^n| = |(DT^n)(z) v_i|,
	\end{equation} 
	where $z$ is any point on $W_i$ and $v_i \in C_s$ is the unit tangent vector that defines $W_i$. As in Section~\ref{sec:Sminus}, for $n \ge 0$, we define 
	$$S_n(W) = \{W_i \in \Psi_n(W): |W_i| < 1\}$$
	and 
	$$L_n(W) = \Psi_n(W) \setminus S_n(W) = \{W_i \in \Psi_n(W): |W_i| \in [1,2]\}.$$
	
	For the proof Theorem~\ref{prop:Projected}, we will require a splitting of $\Psi_n(W)$ into two disjoint families slightly different than the ``long'' and ``short'' decomposition defined above. Since $T^{-n}(W)$ is homeomorphic to an interval, for $W_i \in \Psi_n(W)$ there are at most two segments $W_{i,j} \in \Psi_n(W)$ other than $W_i$ such that $\overline{W}_i \cap \overline{W}_{i,j}$ is nonempty. We refer to these curves as the \textit{neighbors} of $W_i$. We define $G_n(W)$ as the set of long segments $W_i \in L_n(W)$ such that $W_i$ has two neighbors that are both also contained in $L_n(W)$. We then set $B_n(W) = \Psi_n(W)\setminus G_n(W)$. Here, the labels ``G'' and ``B'' stand for ``good'' and ``bad,'' respectively. Note that the definitions above set the convention that $W \in B_0(W)$, as $W$ has no neighbors in $\Psi_0(W) = \{W\}$. 
	
	A slight modification to the argument used to prove Lemma~\ref{lem:complexityfact1}(b) yields the following bound, which will allow us to control the rate at which bad curves are generated. 
	
	\begin{lemma} \label{lem:complexityfact}
		For $k \ge 0$ and $W_i \in \Psi_k(W)$, let
		$$ B_{k+1,k}(W_i) = \{V \in B_{k+1}(W): T(V) \subseteq W_i\} $$
		denote the set of bad curves in $\Psi_{k+1}(W)$ that are generated by $W_i$. There exists a constant $C_0 \ge 1$ independent $\alpha$, $W_i$, and $k$ such that
		$$\# B_{k+1,k} \le C_0. $$
	\end{lemma}
	
	The basic observation behind Lemma~\ref{lem:complexityfact} is that bad curves can be generated by $W_i$ (away from its endpoints) only by intersections with singularity curves contained in the sideways ``V'' shapes of Figure~\ref{fig:backward8} and the most nearby spanning curves. Regions of $W_i$ in the ``bulk'' of $\mathcal{S}^-$ that cross many parallel spanning curves do not generate bad segments. We omit the details of the proof for the sake of brevity.
	
	Fix $W \in \Sigma$ and $n \in \N$. Since $W \in B_0(W)$, for each $W_i \in \Psi_n(W)$ one and only of the following three conditions are met:
	\begin{itemize}
		\item $W_i \in G_n(W)$;
		\vspace{0.1cm}
		\item $W_i \in B_n(W)$ and for every $1 \le k \le n$ there exists $V \in B_{n-k}(W)$ such that $T^k(W_i) \subseteq V$;
		\vspace{0.1cm}
		\item $W_i \in B_n(W)$ has a unique most recent good ancestor. That is, there exists $1 \le k \le n-1$ such that $T^k W_i \subseteq V$ for some $V \in G_{n-k}(W)$ and for every $0 \le \ell \le k-1$, $T^\ell W_i \subseteq V'$ for some $V' \in B_{n-\ell}(W)$.
	\end{itemize}
	
	Let $I_n(W)$ denote the set of $W_i \in \Psi_n(W)$ such that the second condition above holds, and for $1 \le k \le n-1$ and $V \in G_{n-k}(W)$ define 
	$$ I_{n,k}(V) = \{W_i \in B_n(W): T^k(W_i) \subseteq V \text{ and }\forall \, 0 \le \ell \le k-1, \, T^\ell(W_i) \subseteq V' \text{ for some } V' \in B_{n-\ell}(W)\},$$
	which represents the set of bad curves at step $n$ whose most recent good ancestor is $V$. With the definitions above, we can write $\Psi_n(W)$ as the disjoint union
	\begin{equation} \label{eq:projecteddecomp}
		\Psi_n(W) =  G_n(W) \bigcup I_{n}(W) \bigcup \left(\bigcup_{1 \le k \le n-1} \bigcup_{V \in G_{n-k}(W)} I_{n,k}(V)\right).
	\end{equation}
	Moreover, it follows immediately from Lemma~\ref{lem:complexityfact} that $I_n(W)$ and $I_{n,k}(V)$ satisfy the uniform-in-$\alpha$ exponential growth bounds stated below for any $W \in \Sigma$ and $V \in G_{n-k}(W)$.
	
	\begin{lemma} \label{lem:badcount}
		Fix $W \in \Sigma$ and $n \in \N$. Then, for every $1 \le k \le n-1$ and $V \in G_{n-k}(W)$ we have  $\# I_{n,k}(V) \le C_0^{k}$ and
		$\# I_n(W) \le C_0^n$, where $C_0 \ge 1$ is the constant from Lemma~\ref{lem:complexityfact}.
	\end{lemma}
	
	\subsubsection{Multi-intersection points}\label{sec:muliPoints}

	Recall from Section~\ref{sec:hyperbolic} that $\mathcal{S}^-$
	can be written as a finite union of closed line segments on
	$\T^2$ that are each tangent to the unstable cone $C_u$. The
	same is then true of $T^k(\mathcal{S}^-)$ for any $k \ge
	0$. Let $\mathcal{E}_k = \{E_{k,j}\}_{j=1}^{N_k}$ be the
	\textit{minimal decomposition} of $T^k(\mathcal{S}^-)$ into a
	finite union of closed line segments. The lines in
	$\mathcal{E}_k$ are pairwise disjoint, except possibly at
	their endpoints. Looking back at the definition of
	$\mathcal{S}_n^-$ from \eqref{eq:nstepsingdef}, we have that
	\begin{align*}
		\mathcal{S}_n^- =\bigcup_{k=0}^{n-1} \bigcup_{j=1}^{N_k} E_{k,j}\,.
	\end{align*}
	
	We define the $n$-step set of \textit{multi-intersection
		points} $\mathcal{M}_n^-$ to be the set of
	$z \in \mathcal{S}_n^-$ such that
	$z \in E_{k,j} \cap E_{\ell,m}$ for some $0\le k,\ell \le n-1$
	with $(k,j) \neq (\ell,m)$. In other words,
	$z \in \mathcal{M}_n^-$ if it is either a shared endpoint of
	two segments in $\mathcal{E}_k$ or lies on both
	a segment in    $\mathcal{E}_\ell$ and a segment $\mathcal{E}_k$ for some $k \neq
	\ell$. Recalling the definition of $A_\sigma$ from
	\eqref{eq:nprod} and the rest of the set-up from
	Section~\ref{sec:nstepsingularity}, we obtain the following
	result readily from the definition of $\mathcal{M}_n^-$.
	
	\begin{lemma} \label{lem:onecross}
		Suppose that $z, z' \in \T^2 \setminus \mathcal{S}_n^-$ can be connected by a continuous curve $\gamma$ that intersects $\mathcal{S}_n^-$ at a single point $z_0$. If $z_0 \not \in \mathcal{M}_n^-$, then $DT^{-n}(z) = A_{\sigma}$ and $D T^{-n}(z') = A_{\sigma'}$ for words $\sigma, \sigma' \in [4]^n$ that differ in at most one entry.
	\end{lemma}
	
	\begin{proof}
		Since $z_0 \not \in \mathcal{M}_n^-$, there exists a
		unique $0 \le k \le n-1$ such that $z_0 \in T^k(\mathcal{S}^-)$. Thus, $\gamma \cap T^\ell(\mathcal{S}^-) = \emptyset$ for every $\ell \neq
		k$ with $0\le \ell \le n-1$. Since $\{\mathcal{R}_{\ell, i}\}_{i=1}^4$ is a partition of $\T^2 \setminus T^\ell(\mathcal{S}^-)$ into open sets, it follows from the connectdness of $\gamma$ that for every $0 \le \ell \le n-1$ other than $k$, there exists a
		unique $i_\ell \in \{1,2,3,4\}$ such that $\gamma \subseteq \mathcal{R}_{\ell,i_\ell}$. Because $z,z'
		\in \gamma$, we see that for every such $\ell\neq k$, we have $(DT^{-1} \circ T^{-\ell})(z) = (DT^{-1} \circ T^{-\ell})(z')$. Therefore, $DT^{-n}(z) = A_\sigma$ and $DT^{-n}(z') = A_{\sigma'}$ for words $\sigma, \sigma' \in [4]^n$ satisfying $\sigma_\ell = \sigma'_\ell$ for every $\ell \neq k$.
	\end{proof}
	
	\begin{remark} \label{rem:onecross}
		From the proof we see that it is not necessary for $\gamma$ to only intersect $\mathcal{S}_n^-$ at a single point. The same result holds under the weaker assumption that $\gamma$ has intersections with $T^k(\mathcal{S}^-)$ for at most one $0 \le k \le n-1$.
	\end{remark}
	
	For $r > 0$ and $z \in \T^2$, let $C_r(z)$ denote the open cube with side length $2r$ centered at $z$.
	
	\begin{lemma} \label{lem:topology}
		Fix $z \in \T^2$ and let $r \in (0,1/4)$ be such that $C_{2r}(z) \cap \mathcal{M}_n^{-} = \emptyset$. Suppose that there exists $0 \le k \le n-1$ and $E_{k,j} \in \mathcal{E}_k$ that intersects $C_r(z)$ at some point $z_0$. Then, there exists a closed line segment $L \subseteq \overline{C_{2r}(z)} \cap E_{k,j}$ which contains $z_0$ and connects the two vertical edges of $\overline{C_{2r}(z)}$.
	\end{lemma}
	
	\begin{proof}
		From \eqref{eq:ExplicitS1} and the fact that $T_1:\T^2 \to \T^2$ preserves vertical lines, we can express the singularity set as  
		$$\mathcal{S}^- = \{y=0\} \cup \{y=1/2\} \cup T(\{x=0\}) \cup T(\{x=1/2\}).$$
		Thus, 
		\begin{equation}\label{eq:ExplicitS2}
			T^k(\mathcal{S}^-) = T^k(\{y=0\}) \cup T^k(\{y=1/2\}) \cup T^{k+1}(\{x=0\}) \cup T^{k+1}(\{x=1/2\})
		\end{equation}
		is a union of four continuous closed loops on $\T^2$. It follows that the endpoints of every $E \in \mathcal{E}_k$ are in $\mathcal{M}_n^-$. In particular, since $C_{2r}(z) \cap \mathcal{M}_n^{-} = \emptyset$, the endpoints of $E_{k,j}$ are not contained in $C_{2r}(z)$. Thus, the connected component of $E_{k,j} \cap \overline{C_{2r}(z)}$ that contains $z_0$ is the desired line $L$. That this line segment connects the vertical edges of  $\overline{C_{2r}(z)}$ follows from the fact that $z_0 \in C_r(z)$ and $L$ is tangent to $C_u$.
	\end{proof}
	
	A key technical lemma needed in what follows states that $\mathcal{M}_n^- $ is a discrete set and gives an upper bound on the number of points it contains. We defer the proof to Section~\ref{sec:multipoints}.
	
	\begin{lemma} \label{lem:multipoints}
		For any $n \in \N$, the set $\mathcal{M}_n^-$ is
		discrete. Moreover, there exists a constant $C > 0$
		independent of $n$ and $\alpha$ such that $\# \mathcal{M}_n^- \le C(2\alpha)^{2n}$.
	\end{lemma}
	
	\subsection{Reduction to \eqref{eq:gnregdiscussion}} \label{sec:reduction}
	
	Recall from Section~\ref{sec:projecteddiscussion} that 
	$$\int_{\T^2} P_{\le N}(f \circ T^{-m})(z) P_{\le N}(g\circ T^{-n})(z) \, \dee z = \int_{\T^2} f \circ T^{-(m-n)}(z) g_n(z) \, \dee z,$$
	where $g_n$ is given by \eqref{eq:G}. In this section, we define the ``regular'' curves $W_i^{(r)}$ mentioned at the end of Section~\ref{sec:projecteddiscussion} and use Lemma~\ref{lem:multipoints} to reduce the proof of Theorem~\ref{prop:Projected} to \eqref{eq:gnregdiscussion}.
	
	First, we let $W_{x} = \{(x,y) \in \T^2: y \in \T\} \in \Sigma$ and change variables to rewrite the right-hand side of the expression above as an integral over segments in $T^{-n}(\Sigma) \cap \Sigma$: 
	\begin{align}
		& \left|\int_{\T^2} f \circ T^{-(m-n)}(z)\, g_n(z) \, \dee z\right|
		=\left|\int_{\T^2} f \circ T^{-m}(z)\, g_n\circ T^{-n}(z)\, \dee z \right| \\ 
		& \qquad  = \left|\int_{\T}\left(\int_{W_{x}} f \circ T^{-m}\, g_n\circ T^{-n} \, \dee m_{W_x}\right)\dee x \right| \nonumber \\ 
		& \qquad = \left|\int_{\T} \sum_{W_i \in \Psi_n(W_x)} |J_{W_i} T^n|
		\int_{W_i} f \circ T^{-(m-n)}\, g_n \, \dee m_{W_i} \, \dee x\right| \\
		& \qquad  \le \int_{\T} \sum_{W_i \in \Psi_n(W_{x})}|J_{W_i}T^n|
		\left|\int_{W_i} f \circ T^{-(m-n)}\, g_n \, \dee m_{W_i}\right|\dee x, \label{eq:Gint2}
	\end{align}
	where $\Psi_n(W_{x})$ is the decomposition of $T^{-n}(W_x)$ introduced in Section~\ref{sec:nstepgen} and $|J_{W_i} T^n|$ is as defined in \eqref{eq:nStableJacobian}.
	
	Recall from Section~\ref{sec:muliPoints} above that $\mathcal{M}_n^{-}$ is the set of multi-intersection points of $\mathcal{S}_n^-$ and for $r > 0$ we write $C_r(z)$ to the denote the open cube of side length $2r$ centered at $z \in \T^2$. Let $r_n = \alpha^{-3n/2}$ and define
	$$ \mathring{\mathcal{M}}^-_n = \bigcup_{z \in \mathcal{M}_n^-} C_{r_n}(z). $$
	By Lemma~\ref{lem:multipoints}, we have 
	\begin{equation} \label{eq:Bcircvolume}
		|\mathring{\mathcal{M}}^-_n| \le (\# \mathcal{M}^-_n) r_n^2 \leqc (2\alpha)^{2n} \alpha^{-3n} = \left(\frac{\alpha}{4}\right)^{-n}.
	\end{equation}
	We define the sections $\mathring{\mathcal{M}}^-_{n,x} = \mathring{\mathcal{M}}_n^- \cap W_{x}$ and split each $W_i \in \Psi_n(W_{x})$ into pieces according to how $T^n (W_i)$ intersects $\mathring{\mathcal{M}}^-_{n,x}$. Let $$W_i^{(b)} = T^{-n}(T^n W_i \cap \mathring{\mathcal{M}}^-_{n,x}) \quad \text{and} \quad W_i^{(r)} = T^{-n}(T^n W_i \cap (W_{x}\setminus \mathring{\mathcal{M}}^-_{n,x})).$$ Since $T^n (W_i)$ is a vertical line segment, which by \eqref{eq:forwardcontract0} satisfies
	$$ |T^n W_i| \le |W_i|\left(\frac{\alpha^2}{4}\right)^{-n} \le 2(4^n) \alpha^{-2n} \ll r_n,$$
	and $\mathring{\mathcal{M}}^-_{n,x}$ is a finite union of intervals each with length at least $r_n$, $W_i^{(r)}$ contains at most one connected component. Thus, $W_i^{(r)} \in \Sigma$. 
	
	With the definitions above in place, we start from \eqref{eq:Gint2} and split the line integrals of $W_i$ between $W_i^{(r)}$ and $W_i^{(b)}$:  
	\begin{align}
		&\int_{\T} \sum_{W_i \in \Psi_n(W_{x})}|J_{W_i}T^n|
		\left|\int_{W_i} f \circ T^{-(m-n)}\, g_n \, \dee m_{W_i}\right|\dee x \\ 
		& \qquad \le \int_{\T} \sum_{W_i \in \Psi_n(W_{x})}|J_{W_i}T^n|
		\left|\int_{W_i^{(b)}} f \circ T^{-(m-n)}\, g_n \, \dee m_{W_i}\right|\dee x \\
		& \qquad + \int_{\T} \sum_{W_i \in \Psi_n(W_x)}|J_{W_i}T^n|
		\left|\int_{W_i^{(r)}} f \circ T^{-(m-n)}\, g_n\, \dee m_{W_i}(y)\right|\dee x. \label{eq:W_i^g}
	\end{align}
	Since 
	$$ \sup_{z \in \T^2}|g_n(z)||f \circ T^{-(m-n)}(z)| \le \|f\|_{L^\infty}\|g\|_{L^\infty}, $$
	for the first term above we have the estimate
	\begin{align}
		&\int_{\T} \sum_{W_i \in \Psi_n(W_{x})}|J_{W_i}T^n| \left|\int_{W_i^{(b)}} f \circ T^{-(m-n)}\, g_n \, \dee m_{W_i}\right|\dee x  \\ 
		& \qquad \le  \|f\|_{W^{1,\infty}}\|g\|_{W^{1,\infty}} \int_\T \sum_{W_i \in \Psi_n(W_{x})} |J_{W_i} T^n| |W_i^{(b)}| \, \dee x \\ 
		& \qquad =  \|f\|_{W^{1,\infty}}\|g\|_{W^{1,\infty}} \int_\T \sum_{W_i \in \Psi_n(W_{x})} |T^n W_i \cap \mathring{\mathcal{M}}^-_{n,x}|\, \dee x \\
		& \qquad  =  \|f\|_{C^1}\|g\|_{C^1} \int_{\T} |\mathring{\mathcal{M}}^-_{n,x}| \, \dee x \\
		& \qquad =  \|f\|_{C^1}\|g\|_{C^1}|\mathring{\mathcal{M}}_n^-| \leqc  \|f\|_{C^1}\|g\|_{C^1}\left(\frac{\alpha}{4}\right)^{-c_0(m-n)},
	\end{align}
	where in the last inequality we used \eqref{eq:Bcircvolume} and the fact that we consider here the case where $n > c_0(m-n)$. Combining this estimate with \eqref{eq:W_i^g} and using \eqref{eq:weakdecay} of Theorem~\ref{thm:exponential}, we have shown that 
	\begin{align}
		&\left|\int_{\T^2}P_{\le N} (f\circ T^{-n})(z) P_{\le N}(g \circ T^{-m})(z) \, \dee z\right| \\ 
		&\qquad \leqc \|f\|_{C^1}\|g\|_{C^1} \left(\frac{\alpha}{4}\right)^{-c_0(m-n)} + \alpha^{-q(m-n)}\|f\|_{C^1} \sup_{z_1 \in \T} \sum_{W_i \in \Psi_n(W_{x})}|J_{W_i} T^n| |g_n|_{C^1(W_i^{(r)})}.
	\end{align}
	Thus, to complete the proof of Theorem~\ref{prop:Projected}, assuming Lemma~\ref{lem:multipoints}, it is sufficient to show that for every $W_{x}$ we have 
	\begin{equation}\label{eq:reducedgoal}
		\sum_{W_i \in \Psi_n(W_{x})}|J_{W_i} T^n| |g_n|_{C^1(W_i^{(r)})} \le C \|g\|_{C^1}
	\end{equation}
	for a constant $C \ge 1$ that does not depend on $\alpha$, $n$, $N$, or $x$, which is simply a slightly more precise restatement of \eqref{eq:gnregdiscussion}.
	
	\subsection{Regularity of $g_n$} \label{sec:regularity}
	
	We now turn to the goal of proving \eqref{eq:reducedgoal}. Throughout, we fix some $x \in \T$ and to simplify notation set $W = W_{x}$. We proceed by decomposing the sum over  $W_i \in \Psi_n(W)$ in \eqref{eq:reducedgoal} according to \eqref{eq:projecteddecomp}, which gives
	\begin{align} \sum_{W_i \in \Psi_n(W)}|J_{W_i}T^n| |g_n|_{C^1(W^{(r)}_i)} &= \sum_{W_i \in I_{n}(W)} |J_{W_i}T^n| |g_n|_{C^1(W^{(r)}_i)} \label{eq:firstsum}  \\ 
		&  + \sum_{W_i \in G_n(W)}|J_{W_i}T^n| |g_n|_{C^1(W^{(r)}_i)} \label{eq:secondsum} \\ 
		& + \sum_{k=1}^{n-1} \sum_{V \in G_{n-k}(W)} \sum_{W_i \in I_{n,k}(V)} |J_{W_i}T^n| |g_n|_{C^1(W^{(r)}_i)}. \label{eq:thirdsum}
	\end{align}
	
	We will estimate each of these sums separately below, but first we record some preliminary facts. For any vector $v \in \R^2$, differentiating \eqref{eq:G} and \eqref{eq:Gv2} gives the two equivalent formulas
	\begin{align} 
		& \grad g_n(z) \cdot v =  -\int_{\T^2} \grad \psi_N(T^n(\bar{z}) - T^n(z))\cdot  DT^n(z) v \, g(\bar{z}) \, \dee \bar{z}, \label{eq:gradG1} \\ 
		& \grad g_n(z) \cdot v = \int_{\T^2} \psi_N(\bar{z}) \, \grad g(T^{-n}(\bar{z}+T^n(z))) \cdot DT^{-n}(\bar{z}+T^n(z)) DT^n(z)v  \, \dee \bar{z}. \label{eq:gradG2} 
	\end{align}
	Note that by the definition of $\varphi_N$ and \eqref{eq:psiN}, the kernel $\psi_N$ above lives at the scale $N^{-1}$ in the sense that $\psi_N(z) = 0$ for $z \not \in C_{4/N}(0)$. Moreover, $\int_{\T^2} \psi_N(z) \, \dee z = 1$ and $\|\grad \psi_N\|_{L^1(\T^2)}\leqc N$. Lastly, we record that from iterating \eqref{eq:leadingorder2}, we have
	\begin{equation} \label{eq:forwardcontract2}
		|J_{W_i} T^n| \le (\alpha^2 - 10\alpha^{3/2})^{-n} \le  \left(\frac{\alpha^2}{2}\right)^{-n}
	\end{equation}
	for any $n \in \N$ and $W_i \in \Psi_n(W)$.

	\subsubsection{Estimate of \eqref{eq:firstsum}}
	
	We will first bound \eqref{eq:firstsum}, as this term is by far the easiest and does not require any facts about $\mathcal{S}_n^{-}$ or $\mathcal{M}_n^-$. Fix any $W_i \in \Psi_n(W)$ and let $v_i \in C_s$ denote its unit tangent vector. Since $T^n(W_i)$ is contained in $W$, and hence a vertical segment, we know that $DT^n(z)v_i \in C_s$. Thus, it follows from \eqref{eq:leadingorder} and \eqref{eq:forwardcontract2} that for any $\bar{z} \in \T^2$ and $z \in W_i$ we have 
	\begin{equation}\label{eq:trivialgrad}
		|DT^{-n}(\bar{z} + T^n(z))DT^n(z) v_i| \le (\alpha^2 + 10\alpha^{3/2})^n|DT^n(z) v_i| \le \left(\frac{\alpha^2 + 10\alpha^{3/2}}{\alpha^2 - 10 \alpha^{3/2}}\right)^n \le 2^n
	\end{equation}
	whenever $\alpha$ is sufficiently large. Since
	\begin{equation} \label{eq:gC1}
		|g|_{C^1(W_i^{(r)})} \le \|g\|_{L^\infty} + \sup_{z \in W^{(r)}_i}|\grad g_n(z) \cdot v_i|,
	\end{equation}
	using \eqref{eq:gradG2} and \eqref{eq:trivialgrad} we obtain 
	\begin{equation} \label{eq:sum1_1}
		|g_n|_{C^1(W_i^{(r)})} \leqc 2^n \|g\|_{C^1}.
	\end{equation}
	Thus, by \eqref{eq:forwardcontract2} and Lemma~\ref{lem:badcount}, we have
	\begin{equation} \label{eq:firstsumfinal}
		\sum_{W_i \in I_n(W
			)} |J_{W_i}T^n| |g_n|_{C^1(W^{(r)}_i)}\leqc \# I_n(W) \left(\frac{\alpha^2}{2}\right)^{-n} 2^n \|g\|_{C^1} \le (4C_0)^n \alpha^{-2n} \|g\|_{C^1} \leqc \|g\|_{C^1},
	\end{equation}
	where $C_0$ is as in the statement of Lemma~\ref{lem:badcount} and $\alpha$ is taken sufficiently large in the final inequality.
	
	\subsubsection{Estimate of \eqref{eq:secondsum} and \eqref{eq:thirdsum}}
	
	We now turn to the sums \eqref{eq:secondsum} and \eqref{eq:thirdsum}. The following lemma will play a crucial role in their analysis and is the key motivation behind our decomposition of $\Psi_n(W)$ into good and bad curves, as well as the further splitting \eqref{eq:projecteddecomp}. Roughly speaking, for $W_i \in \Psi_n(W)$ and $z \in W^{(r)}_i$, the lemma restricts how the neighborhood $C_{4/N}(T^n(z))$ can intersect the backward singularity set. It controls intersections with $\mathcal{S}_n^-$ when $W_i$ is good, and with $\mathcal{S}_{n-k}^-$ when $W_i$ is a bad curve whose most recent good ancestor lies in $\Psi_{n-k}(W)$. 
	
	\begin{lemma} \label{lem:goodlemma2}
		Fix $0 \le k \le n-1$ and $V \in G_{n-k}(W)$. Let $W_i \in \Psi_n(W)$ satisfy $$T^k(W_i) \subseteq V.$$
		There exists a universal constant $C_1 \ge 1$ such that, if
		\begin{equation} \label{eq:maincaselemma}
			N^{-1} \le \frac{|J_{W_i} T^n|}{C_1},
		\end{equation}
		then for any $z \in W_i^{(r)}$ the neighborhood $C_{4/N}(T^n(z))$ can intersect $T^\ell(\mathcal{S}^-)$ for at most one $0 \le \ell \le n-k-1$.
	\end{lemma}
	
	\begin{remark}
		If $k = 0$, then $V = W_i \in G_n(W)$, while if $k > 0$, then $V \in G_{n-k}(W)$ and $W_i \in I_{n,k}(V)$. 
	\end{remark}
	
	\begin{proof}
		Let $V \in G_{n-k}(W)$, $W_i \in \Psi_n(W)$ satisfy $T^k(W_i) \subseteq V$, and $z \in W_i^{(r)} \subseteq W_i$. Since $V \in G_{n-k}(W)$, by definition it has two long neighbors $V_1, V_2 \in L_{n-k}(W)$. Let
		$$\widetilde{V} = \overline{T^{n-k}(V) \cup T^{n-k}(V_1) \cup T^{n-k}(V_2)}.$$
		Then, $\widetilde{V}$ is a vertical, closed line segment contained in $W$.  Moreover, $T^n(z) \in T^{n-k}(V) \subseteq \widetilde{V}$ because $$T^n(z) = T^{n-k}(T^k(z)) \in T^{n-k}(T^k W_i) \subseteq T^{n-k}(V).$$ 
		Recalling the discussion at the beginning of Section~\ref{sec:nstepgen}, there are words $\sigma, \sigma_1, \sigma_2 \in [4]^{n-k}$ such that $T^{n-k}(V) \subseteq \mathcal{C}_\sigma^-$, $T^{n-k}(V_1) \subseteq \mathcal{C}_{\sigma_1}^-$, and $T^{n-k}(V_2) \subseteq \mathcal{C}_{\sigma_2}^-$. In particular, $\widetilde{V}$ can intersect $\mathcal{S}_{n-k}^-$ only at the endpoints of the closures of $T^{n-k}(V)$, $T^{n-k}(V_1)$, and $T^{n-k}(V_2)$. Thus, if $\widetilde{W}$ is a vertical line segment that intersects the shared endpoint $\overline{T^{n-k}(V)} \cap \overline{T^{n-k}(V_i)}$ for some $i \in \{1,2\}$, then
		\begin{equation}\label{eq:contradiction2}
			|\widetilde{W}| \le \frac{1}{2}\min(|T^{n-k}(V)|, |T^{n-k}(V_{i})|) \implies \# (\widetilde{W} \cap \mathcal{S}_{n-k}^-) \le 1.
		\end{equation}
		The remainder of the proof is based on showing that if the lemma is false, then we can construct such a line segment $\widetilde{W}$ for which the implication in \eqref{eq:contradiction2} fails.
		
		Let $\mathcal{U}_1= C_{4/N}(T^n(z))$, $\mathcal{U}_2 = C_{8/N}(T^n(z))$, and define $\widetilde{W} = \mathcal{U}_2 \cap W$.  Since $T^n(z) \in T^{n-k}(V)$, $\widetilde{W}$ is a vertical line segment of length $16N^{-1}$ that intersects $T^{n-k}(V)$ and connects the two horizontal edges of $\mathcal{U}_2$. Suppose for the sake of contradiction that there exist $0 \le \ell_1,\ell_2 \le n-k-1$ with $\ell_1 \neq \ell_2$ such that
		\begin{equation} \label{eq:contradiction3}
			T^{\ell_1}(\mathcal{S}^-) \cap \mathcal{U}_1 \neq \emptyset \quad \text{and} \quad T^{\ell_2}(\mathcal{S}^-) \cap \mathcal{U}_1 \neq \emptyset.
		\end{equation} 
		Since $z \in W_i^{(r)}$ and 
		$$N^{-1} \le |J_{W_i}T^n| \le \left(\frac{\alpha^2}{2}\right)^n \ll \alpha^{-3n/2} = r_n$$ 
		by \eqref{eq:maincaselemma} and \eqref{eq:forwardcontract2}, it follows from the definition of $W_i^{(r)}$ that $\mathcal{U}_2 \cap \mathcal{M}_n^- = \emptyset$. Therefore, by \eqref{eq:contradiction3} and Lemma~\ref{lem:topology} there exist line segments $L_1 \subseteq T^{\ell_1}(\mathcal{S}^-)$ and $L_2 \subseteq T^{\ell_2}(\mathcal{S}^-)$ that each connect the two vertical edges of $\mathcal{U}_2$. In particular, $L_1$ and $L_2$ intersect $\widetilde{W}$ at some points $z_1$ and $z_2$, respectively. Since $\widetilde{W} \subseteq \mathcal{U}_2 \subseteq \T^2 \setminus \mathcal{M}_n^{-}$ and $\ell_1 \neq \ell_2$, by the definition of the multi-intersection set it follows that $z_1 \neq z_2$ and we conclude that 
		\begin{equation} \label{eq:contradiction4}
			\#(\widetilde{W} \cap \mathcal{S}_{n-k}^{-}) \ge 2.
		\end{equation}
		
		Because $T^{n-k}(V) \subseteq \mathcal{C}_\sigma^- \subseteq \T^2 \setminus \mathcal{S}_{n-k}^-$ and $\widetilde{W}$ is a vertical segment containing $T^{n}(z) \in T^{n-k}(V)$, it is clear from \eqref{eq:contradiction4} that $\widetilde{W}$ must contain at least one endpoint of $\overline{T^{n-k}(V)}$. Without loss of generality, let us supposed that $\widetilde{W}$ contains the shared endpoint $\overline{T^{n-k}(V)} \cap \overline{T^{n-k}(V_1)}$. Since $\widetilde{W} \subseteq \T^2 \setminus \mathcal{M}_n^{-} \subseteq \T^2 \setminus \mathcal{M}_{n-k}^-$, this shared endpoint is not in $\mathcal{M}_{n-k}^-$. As $T^{n-k}(V)$ and $T^{n-k}(V_1)$ both do not intersect $\mathcal{S}_{n-k}^-$, it follows then from Lemma~\ref{lem:onecross} that the words $\sigma, \sigma_1 \in [4]^{n-k}$ differ in at most one entry. Thus, by Lemma~\ref{lem:equivexpansion} and the fact that $V$ and $V_1$ are both long, we have 
		\begin{equation} \label{eq:lengthequiv}
			\frac{|T^{n-k}(V)|}{|T^{n-k}(V_1)|} = \frac{|V| |A_{\sigma_1} e_2|}{|V_1| |A_\sigma e_2|} \le 2 C_0,
		\end{equation}
		where $C_0$ is as in Lemma~\ref{lem:equivexpansion}. Observe now that $$|J_{W_i}T^n| = |J_{W_i}T^k| |J_{V} T^{n-k}| \le |J_V T^{n-k}| = \frac{|T^{n-k}(V)|}{|V|} \le |T^{n-k}(V)|, $$
		and therefore \eqref{eq:maincaselemma} and \eqref{eq:lengthequiv} imply
		\begin{equation} \label{eq:maincasev2}
			N^{-1} \le \frac{|T^{n-k}(V)|}{C_1} \le \frac{2C_0}{C_1} |T^{n-k}(V_1)|.
		\end{equation}
		Thus, 
		\begin{equation}
			|\widetilde{W}| = 16 N^{-1} \le \frac{32C_0}{C_1} \min(|T^{n-k}(V)|, |T^{n-k}(V_1)|) 
		\end{equation}
		and \eqref{eq:contradiction4} contradicts \eqref{eq:contradiction2} for $C_1 > 64 C_0$.
	\end{proof}
	
	Lemma~\ref{lem:goodlemma2} allows us to control the size of $|g_n|_{C^1(W_i^{(r)})}$ when either $W_i \in G_n(W)$ or $W_i \in I_{n,k}(V)$ for some $V \in G_{n-k}(W)$. Precisely, we have the following.
	
	\begin{lemma}\label{lem:intermediatesum}
		If $W_i \in G_n(W)$, then 
		\begin{equation} \label{eq:gnreggood}
			|g_n|_{C^1(W_i^{(r)})} \leqc \|g\|_{C^1}
		\end{equation}
		and if $W_i \in I_{n,k}(V)$ for some $1 \le k \le n-1$ and $V \in G_{n-k}(W)$, then 
		\begin{equation} \label{eq:gnregintermediate}
			|g_n|_{C^1(W_i^{(r)})} \leqc 2^k \|g\|_{C^1}.
		\end{equation}
		The implicit constants above are independent of $W_i$, $W$, $k$, $n$, $V$, and $\alpha$.
	\end{lemma}
	\begin{proof}
		Let $0 \le k \le n-1$, $V \in G_{n-k}(W)$, and $W_i \in \Psi_n(W)$ satisfy $T^k(W_i) \subseteq V$. We will prove that
		\begin{equation} \label{eq:gnregcombined}
			|g_n|_{C^1(W_i^{(r)})} \leqc 2^k \|g\|_{C^1},
		\end{equation}
		which treats both \eqref{eq:gnreggood} and \eqref{eq:gnregintermediate} simultaneously. Indeed, the $k = 0$ case corresponds to \eqref{eq:gnreggood}, while the $1 \le k \le n-1$ is the situation of \eqref{eq:gnregintermediate}.
		
		Let $v_i$ denote the unit tangent vector of $W_i$. As in the estimate of \eqref{eq:firstsum}, to prove \eqref{eq:gnregcombined} it is sufficient to show 
		\begin{equation} \label{eq:combinedreggoal}
			\sup_{z \in W_i^{(r)}} |\grad g_n(z) \cdot v_i| \leqc 2^k \|g\|_{C^1}.
		\end{equation}
		Fix $z \in W_i^{(r)}$. First observe that from \eqref{eq:gradG1}, we have
		\begin{equation}
			|\grad g_n(z) \cdot v_i| \le |J_{W_i}T^n| \|g\|_{L^\infty} \|\grad \psi_N\|_{L^1} \leqc N |J_{W_i} T^n| \|g\|_{C^1},
		\end{equation}
		where we used that $|D T^n(z) v_i| = |J_{W_i} T^n|$. Thus, \eqref{eq:combinedreggoal} is immediate when $N |J_{W_i} T^n| \le C_1$, where $C_1$ is as in the statement of Lemma~\ref{lem:goodlemma2}, and it suffices to consider the case that \eqref{eq:maincaselemma} holds. In this case, Lemma~\ref{lem:goodlemma2} applies and we estimate $|\grad g_n(z) \cdot v_i|$ using \eqref{eq:gradG2}, re-written below with $v = v_i$ for convenience: 
		\begin{equation}\label{eq:integral}
			\grad g_n(z) \cdot v_i = \int_{\T^2} \psi_N(\bar{z}) \, \grad g(T^{-n}(\bar{z}+T^n(z))) \cdot DT^{-n}(\bar{z}+T^n(z)) DT^n(z)v_i  \, \dee \bar{z}.
		\end{equation}
		For any $\bar{z} \in \T^2$ where all relevant derivatives are defined, we have the formula
		\begin{equation} \label{eq:derivativefactor}
			\begin{aligned} 
				& DT^{-n}(\bar{z}+T^n(z))DT^n(z) v_i \\ 
				& \qquad  = DT^{-k}(T^{-(n-k)}(\bar{z} + T^n(z)))DT^{-(n-k)}(\bar{z}+T^n(z))(DT^{-(n-k)}(T^n(z)))^{-1}DT^k(z) v_i. 
			\end{aligned}
		\end{equation}
		Suppose that $\bar{z} \in \supp(\psi_N) \subseteq C_{4/N}(0)$. Then, by Lemma~\ref{lem:goodlemma2}, there exists a continuous curve connecting $T^n(z)$ and $T^n(z) + \bar{z}$ that intersects $T^{\ell}(\mathcal{S}^-)$ for at most $0 \le \ell \le n-k-1$. Consequently, from Lemma~\ref{lem:onecross} (see Remark~\ref{rem:onecross}), we have that $DT^{-(n-k)}(T^n(z)) = A_\sigma$ and $DT^{-(n-k)}(\bar{z} + T^n(z)) = A_{\sigma'}$ for words $\sigma, \sigma' \in [4]^{n-k}$ that differ in at most one entry. Let $v \in C^s$ be the unit vector in the direction of $DT^k(z) v_i$. Then, we can rewrite \eqref{eq:derivativefactor} as
		\begin{equation} \label{eq:derivativefactor2}
			DT^{-n}(\bar{z}+T^n(z))DT^n(z) v_i = \frac{|DT^k(z) v|}{|A_\sigma (A_\sigma^{-1}v)|} \left[DT^{-k}(T^{-(n-k)}(\bar{z} + T^n(z)))\right]\left(A_{\sigma'}(A_\sigma^{-1}v)\right).
		\end{equation}
		By \eqref{eq:leadingorder2},
		\begin{equation} \label{eq:randombound1}
			|DT^k(z)v_i| \le (\alpha^2 - 10\alpha^{3/2})^{-k}.
		\end{equation}
		Moreover, $A_\sigma^{-1} v \parallel A_{\sigma}^{-1} DT^k(z) v_i = DT^n(z) v_i \parallel e_2 \in C_s$, and hence $A_{\sigma'} A_\sigma^{-1} v \in C_s$, from which it follows by \eqref{eq:leadingorder2} that
		\begin{equation} \label{eq:randombound2}
			\left|\left[DT^{-k}(T^{-(n-k)}(\bar{z} + T^n(z)))\right]\left(A_{\sigma'}(A_\sigma^{-1}v)\right)\right| \le (\alpha^2 + 10\alpha^{3/2})^k |A_{\sigma'}(A_\sigma^{-1}v)|.
		\end{equation}
		Putting \eqref{eq:randombound1} and \eqref{eq:randombound2} into \eqref{eq:derivativefactor2} and applying Lemma~\ref{lem:equivexpansion} yields
		$$\left| DT^{-n}(\bar{z}+T^n(z))DT^n(z) v_i \right| \le  \frac{(\alpha^2+10\alpha^{3/2})^k}{(\alpha^2 - 10\alpha^{3/2})^k} \frac{|A_{\sigma'}(A_\sigma^{-1}v)|}{|A_\sigma (A_\sigma^{-1} v)|} \leqc 2^k.$$
		This bound holds for Lebesgue almost every $\bar{z}$ where the integrand in \eqref{eq:integral} is nonzero, from which \eqref{eq:combinedreggoal} follows immediately.
	\end{proof}
	
	We are now finally prepared to estimate \eqref{eq:secondsum} and \eqref{eq:thirdsum}. For \eqref{eq:secondsum} we use \eqref{eq:gnreggood} and the fact that each $W_i \in G_n(W)$ satisfies $|W_i| \ge 1$ to obtain
	\begin{equation} \label{eq:secondsumfinal}
		\sum_{W_i \in G_n(W)}|J_{W_i}T^n| |g_n|_{C^1(W^{(r)}_i)} \leqc \|g\|_{C^1}\sum_{W_i \in G_n(W)}|J_{W_i}T^n||W_i| \le \|g\|_{C^1}|W| \le \|g\|_{C^1}. 
	\end{equation}
	For \eqref{eq:thirdsum}, applying \eqref{eq:gnregintermediate}, using Lemma~\ref{lem:badcount} to bound $\# I_{n,k}$, and taking $\alpha$ sufficiently large, we have
	\begin{align}
		\sum_{k=1}^{n-1} \sum_{V \in G_{n-k}(W)} \sum_{W_i \in I_{n,k}(V)} |J_{W_i}T^n| |g_n|_{C^1(W^{(r)}_i)} & = \sum_{k=1}^{n-1} \sum_{V \in G_{n-k}(W)} |J_V T^{n-k}| \sum_{W_i \in I_{n,k}(V)} |J_{W_i} T^{k}| |g_n|_{C^1(W_i^{(r)})} \\ 
		& \leqc \sum_{k=1}^{n-1} \sum_{V \in G_{n-k}(W)} |J_V T^{n-k}| \,  \#I_{n,k}(V) \left(\frac{\alpha^2}{2}\right)^{-k} 2^k \|g\|_{C^1} \\
		& \leqc \|g\|_{C^1} \sum_{k=1}^{n-1} \sum_{V \in G_{n-k}(W)} |J_V T^{n-k}| \left(\frac{\alpha^2}{4C_0}\right)^{-k} \\ 
		& \leqc \|g\|_{C^1} \sum_{k=1}^{n-1} \alpha^{-k} \sum_{V \in G_{n-k}(W)} |J_V T^{n-k}| |V|\\ 
		& \leqc \|g\|_{C^1} \sum_{k=1}^{n-1} \alpha^{-k} \leqc \|g\|_{C^1}. \label{eq:thirdsumfinal}
	\end{align}
	In the last three lines above, we have noted that $V \in G_{n-k}(W)$ implies that $V$ is long, so that 
	$$\sum_{V \in G_{n-k}(W)} |J_V T^{n-k}| \leqc \sum_{V \in G_{n-k}(W)} |J_V T^{n-k}| |V| = \sum_{V \in G_{n-k}(W)} |T^{n-k}(V)| \le |W| \le 1. $$
	Combining \eqref{eq:firstsumfinal}, \eqref{eq:secondsumfinal}, and \eqref{eq:thirdsumfinal}, we have proven \eqref{eq:reducedgoal}.
	
	\subsection{Proof of Lemma~\ref{lem:multipoints}} \label{sec:multipoints} 
	Section~\ref{sec:reduction} reduced the proof of Theorem~\ref{prop:Projected} to \eqref{eq:reducedgoal}, assuming Lemma~\ref{lem:multipoints}, and Section~\ref{sec:regularity} established \eqref{eq:reducedgoal}. Thus, to complete the proof of Theorem~\ref{prop:Projected} it remains to prove Lemma~\ref{lem:multipoints}. This will be carried out in the present section.   
	
	Recall from Section~\ref{sec:nstep} that \begin{equation} \label{eq:singrecall}
		\mathcal{S}_n^{-} = \bigcup_{k=0}^{n-1} T^k(\mathcal{S}^-) = \bigcup_{k=0}^{n-1} \bigcup_{j=1}^{N_k} E_{k,j},  
	\end{equation}
	where $\mathcal{E}_k = \{E_{k,j}\}_{j=1}^{N_k}$ is the minimal decomposition of $T^k(\mathcal{S}^-)$ into closed line segments. The set of multi-intersection points was defined as 
	\begin{equation} \label{eq:multiexplicit}
		\mathcal{M}_n^- = \{z \in \mathcal{S}_n^{-}: z \in E_{k,j} \cap E_{\ell,m} \,\, \text{for} \,\, (k,j) \neq (\ell,m)\}.
	\end{equation}
	That is, $\mathcal{M}_n^{-}$ consists of the shared endpoints between segments in $\mathcal{E}_k$ for $0 \le k \le n-1$, together with the points lying in both $T^k(\mathcal{S}^-)$ and $T^\ell(\mathcal{S}^-)$ for distinct $0\le k,\ell \le n-1$. Equivalently,
	\begin{equation}\label{eq:Mnequiv}
		\MM_n^- = \left(\bigcup_{k=0}^{n-1} \partial \mathcal{E}_k\right) \bigcup \left(\bigcup_{0 \le j < k \le n-1} T^j(\mathcal{S}^-) \cap T^k(\mathcal{S}^-) \right),
	\end{equation}
	where $\partial \mathcal{E}_k$ denotes the set of endpoints of the segments in $\mathcal{E}_k$. Here we have used the fact, pointed out in the proof of Lemma~\ref{lem:topology}, that the topological structure of $T^k(\mathcal{S}^-)$ forces every endpoint in $\mathcal{E}_k$ to be shared. Our goal is to show that $\mathcal{M}_n^-$ is discrete and  $\# \mathcal{M}_n^{-} \le C(2\alpha)^{2n}$ for a constant $C$ independent of $n$ and $\alpha$.
	
	To clarify the definitions above, we first describe explicitly the minimal decomposition of $\mathcal{S}^-$ and its set of multi-intersection points $\mathcal{M}_1^-$. Define the horizontal lines $\mathcal{H}_1 = \{y=0\}$, $\mathcal{H}_2 = \{y=1/2\}$ and the vertical lines $\mathcal{V}_1 = \{x=0\}$, $\mathcal{V}_2 = \{x=1/2\}$. As noted in Section~\ref{sec:hyperbolic}, the singularity set can be written as 
	\begin{equation} \label{eq:HVrep}
		\mathcal{S}^- = \mathcal{H}_1 \cup \mathcal{H}_2 \cup T_2(\mathcal{V}_1) \cup T_2(\mathcal{V}_2).
	\end{equation}
	Both $\mathcal{V}_1$ and $\mathcal{V}_2$ can be decomposed into two segments that each lie in a single region where $T_2$ is smooth. In particular, if we define $$\mathcal{V}_i^-= \mathcal{V}_i \cap \{(x,y): 0 \le y \le 1/2 \} \quad \text{and}  \quad \mathcal{V}_i^+ = \mathcal{V}_i \cap \{(x,y): 1/2 \le y \le 1\},$$
	then $T_2(\mathcal{V}_i^\pm)$ consists of a single closed line segment and we see from \eqref{eq:HVrep} that the minimal decomposition of $\mathcal{S}^-$ contains six lines:
	\begin{equation}\label{eq:Minimal0}
		\mathcal{S}^- = \mathcal{H}_1 \cup \mathcal{H}_2 \cup T_2(\mathcal{V}_1^+) \cup T_2(\mathcal{V}_1^-) \cup T_2(\mathcal{V}_2^+) \cup T_2(\mathcal{V}_2^-). 
	\end{equation}
	The set $\mathcal{M}_1^-$ consists of the shared endpoints of these six segments. Since $\mathcal{H}_1 \cap \mathcal{H}_2 = \mathcal{V}_1 \cap \mathcal{V}_2 = \emptyset$, such points arise only from intersections of the form $T_2(\mathcal{V}_i^\pm) \cap \mathcal{H}_j$ for $i,j \in \{1,2\}$. Each of these eight intersections contains a single point, but only four are distinct (see Figure~\ref{fig:backward8}). In particular, 
	\begin{equation} \label{eq:M1}
		\# \mathcal{M}_1^- = 4, \quad \mathcal{M}_1^- = \{(0,1/2), (1/2,1/2), T_2(0,0), T_2(1/2,0)\}.
	\end{equation}
	Before proceeding, it is also useful to record a convenient representation of $\mathcal{S}^-$ that follows immediately from \eqref{eq:HVrep}. Specifically, since $T_1(\mathcal{V}_i) = \mathcal{V}_i$, we have $T(\mathcal{V}_i) = T_2(\mathcal{V}_i)$, and therefore
	\begin{equation} \label{eq:HVrep2}
		\mathcal{S}^- = \mathcal{H} \cup T(\mathcal{V}),
	\end{equation}
	where $\mathcal{H} = \mathcal{H}_1 \cup \mathcal{H}_2$ and $\mathcal{V} = \mathcal{V}_1 \cup \mathcal{V}_2$.

	\subsubsection{Forward singularity set and generations of unstable curves}
	
	Let $\mathcal{W}_u$ denote the set of line segments in $\T^2$ whose tangent vectors lie in the unstable cone $C_u$. As discussed in Section~\ref{sec:hyperbolic} and evident from \eqref{eq:HVrep}, each segment in the minimal decomposition of $\mathcal{S}^-$ belongs to $\mathcal{W}_u$.  Since the characterization of $\MM_n^-$ in \eqref{eq:Mnequiv} involves the minimal decompositions of the sets $T^k(\mathcal{S}^-)$, we must study images of unstable curves under $T$. To this end, we first briefly describe the forward singularity set $\mathcal{S}^+$, and then introduce a decomposition of $T^k(W)$ for $W \in \mathcal{W}_u$ analogous to the decompositions $\Psi_k$ of pre-images of stable curves defined in Section~\ref{sec:Sminus}, with the role of $T^{-1}$ replaced by $T$. 
	
	The same reasoning that leads to \eqref{eq:HVrep} shows that $\mathcal{S}^+$ can be written as 
	\begin{equation} \label{eq:SplusHV}
		\mathcal{S}^+ = T^{-1}(\mathcal{H}) \cup \mathcal{V}.
	\end{equation}
	A plot of $\mathcal{S}^+$ is shown in Figure~\ref{fig:S1+}. It has the same geometric features as $\mathcal{S}^-$, but now consists of line segments tangent to $C_s$. Given $W \in \mathcal{W}_u$, we cut $W$ along $\mathcal{S}^+$ into subsegments, each contained in a region where $T$ is affine. We then push these subsegments through $T$, and subdivide any resulting segments of length greater than two into subsegments whose lengths lie in $(1,2]$. This produces a decomposition of $T(W)$ into segments in $\mathcal{W}_u$, which we denote by $\Psi^+(W)$. We then define $\Psi_1^+(W) = \Psi^+(W)$, and for $k \ge 2$,
	$$\Psi_k^+(W) = \bigcup_{W_i \in \Psi_{k-1}^+(W)} \Psi^+(W_i).$$
	We also set $\Psi_0^+(W) = \{W\}$.
	
	\begin{figure}[h]
		\centering 
		\includegraphics[ height=9cm, width=9cm]{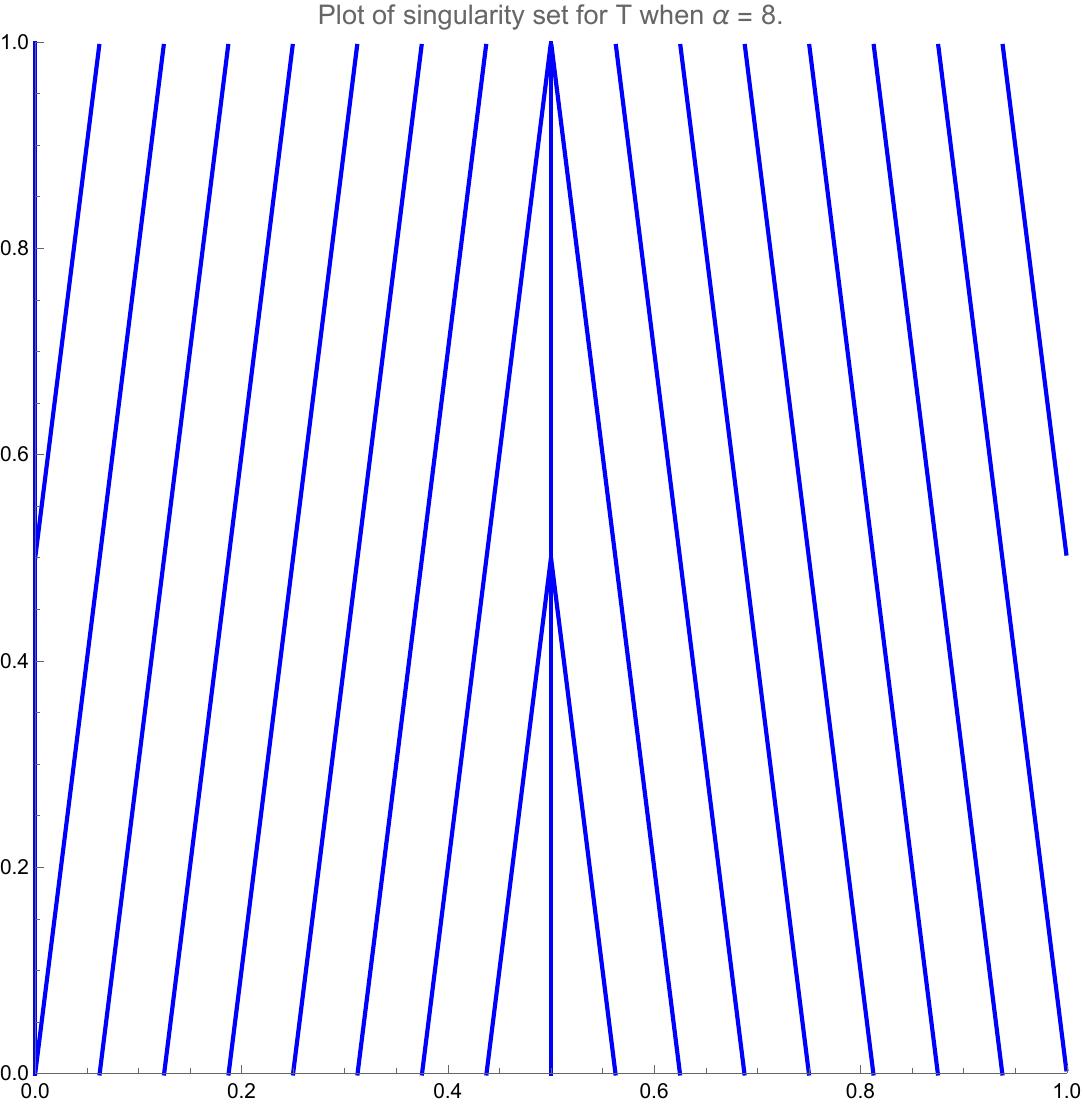}
		\caption{The singularity set $\mathcal{S}^+$ of $T$ when $\alpha = 8$.}
		\label{fig:S1+}
	\end{figure}
	
	The lemma below collects two facts following from the definition of $\Psi_k^+$ that will be necessary to estimate $\# \MM_n^-$ using \eqref{eq:Mnequiv}.
	
	\begin{lemma} \label{lem:psiplusfacts}
		Let $\{W_i\}_{i=1}^m \subseteq \mathcal{W}_u$ be any decomposition of $\mathcal{S}^-$. Then, for each $k \ge 0$, we have
		\begin{equation} \label{eq:psiplusfacts1}
			\# \partial \mathcal{E}_k \le 2\sum_{i=1}^m \# \Psi^+_k(W_i)
		\end{equation}
		and
		\begin{equation} \label{eq:psiplusfacts2}
			\# (\mathcal{S}^+ \cap T^k(\mathcal{S}^-)) \le 2\sum_{i=1}^m \# \Psi^+_{k+1}(W_i).
		\end{equation}
	\end{lemma}
	
	\begin{remark} \label{rem:discrete}
		The set on the left-hand side of \eqref{eq:psiplusfacts2} is necessarily discrete, since the intersection $V_1 \cap V_2$ is finite for any $V_1 \in \mathcal{W}_s$ and $V_2 \in \mathcal{W}_u$.
	\end{remark}
	
	\begin{proof}
		Since $\cup_{i=1}^m \Psi_k^+(W_i)$ and $\mathcal{E}_k$ are both decompositions of $T^k(\mathcal{S}^-)$ into line segments, and $\mathcal{E}_k$ is the minimal such decomposition, we have 
		$$\# \mathcal{E}_k \le \#\left(\bigcup_{i=1}^m \Psi_k^+(W_i)\right) \le \sum_{i=1}^m \# \Psi_k^+(W_i). $$
		The bound \eqref{eq:psiplusfacts1} then follows from the fact that the number of endpoints in $\mathcal{E}_k$ cannot exceed $2 \, \# \mathcal{E}_k$.
		
		We now turn to the second bound. Let $z \in \mathcal{S}^+ \cap T^k(\mathcal{S}^-)$. Then $z \in W$ for some $1 \le i_0 \le m$ and $W \in \Psi_k^+(W_{i_0})$. Since $z \in \mathcal{S}^+$, the construction of $\Psi^+(W)$ cuts $W$ at $z$ before applying $T$. In particular, in the decomposition $\Psi^+(W)$ of $T(W)$, there is at least one segment having $T(z)$ as an endpoint. Choose one such segment and denote it by $W_z$. Then, $W_z \in \Psi_{k+1}^+(W_{i_0}) \subseteq \cup_{i=1}^m \Psi_{k+1}^+(W_i)$. This defines a map
		$$\zeta: \mathcal{S}^+ \cap T^k(\mathcal{S}^-) \to \bigcup_{i=1}^m \Psi_{k+1}^+(W_i), \quad \zeta(z)=W_z$$
		with the property that $T(z)$ is an endpoint of $\zeta(z)$. Since each segment in the codomain of $\zeta$ has only two endpoints, the mapping $\zeta$ is at most two-to-one. It follows that 
		$$\#(\mathcal{S}^+ \cap T^k(\mathcal{S}^-)) \le 2 \, \# \left(\bigcup_{i=1}^m \Psi_{k+1}^+(W_i)\right)\le 2 \sum_{i=1}^m \# \Psi_{k+1}^+(W_i),$$
		which is \eqref{eq:psiplusfacts2}.
	\end{proof}
	
	The essential ingredient in the proof of Lemma~\ref{lem:multipoints} is the following counting estimate for $\Psi_k^+$.
	
	\begin{lemma} \label{lem:longcount}
		Let $W \in \mathcal{W}_u$ satisfy $|W| \in [1,2]$. Then, for $\alpha$ sufficiently large, it holds that $\# \Psi_n^+(W) \le 4(2\alpha)^{2n}$ for every $n \ge 1$.
	\end{lemma}
	
	\subsubsection{Conclusion of the proof}
	
	We momentarily postpone the proof of Lemma~\ref{lem:longcount} and use it to conclude the proof of Lemma~\ref{lem:multipoints}.
	
	\begin{proof}[Proof of Lemma~\ref{lem:multipoints}]
		
		Observe that \eqref{eq:HVrep2} and \eqref{eq:SplusHV} imply that the forward and backward singularity sets are related by 
		$$\mathcal{S}^+ = T^{-1}(\mathcal{S}^-).$$
		Thus, for any $0 \le j < k \le n-1$ we have 
		$$T^j(\mathcal{S}^-) \cap T^k(\mathcal{S}^-) = T^{j+1}(\mathcal{S}^+ \cap T^{k-j-1}(\mathcal{S}^-)).$$
		Since $k-j-1 \ge 0$, we see from Remark~\ref{rem:discrete} that $T^j(\mathcal{S}^-) \cap T^k(\mathcal{S}^-)$ is discrete and
		\begin{equation} \label{eq:discrete2}
			\#(T^j(\mathcal{S}^-) \cap T^k(\mathcal{S}^-)) = \# (\mathcal{S}^+ \cap T^{k-j-1}(\mathcal{S}^-)). 
		\end{equation}
		Using \eqref{eq:discrete2} in \eqref{eq:Mnequiv}, we obtain
		\begin{equation} \label{eq:Mnequivbound}
			\# \mathcal{M}_n^- \le \sum_{k=0}^{n-1} \# \partial \mathcal{E}_k \, \, \, + \sum_{0 \le j < k \le n-1}  \# (\mathcal{S}^+ \cap T^{k-j-1}(\mathcal{S}^-)).
		\end{equation}
		
		To conclude, we bound both terms in \eqref{eq:Mnequivbound} using Lemmas~\ref{lem:psiplusfacts} and~\ref{lem:longcount}. Each segment in the minimal decomposition of $\mathcal{S}^-$ from \eqref{eq:Minimal0} has length greater than or equal to one and belongs to $\mathcal{W}_u$. Therefore, we can find a family $\{W_i\}_{i=1}^m \subseteq \mathcal{W}_u$ such that $$\mathcal{S}^-=\bigcup_{i=1}^m W_i \quad \text{and} \quad |W_i| \in [1,2]  \quad \forall \, 1\le i \le m.$$ 
		Moreover, since 
		$$|T_2(\mathcal{V}_i^\pm)| = \sqrt{\alpha^2 + \frac{1}{4}},$$
		the total length of the segments in \eqref{eq:Minimal0} does not exceed $5\alpha$ for large $\alpha$. Thus, $\{W_i\}_{i=1}^m$ may be chosen so that $m \le 5\alpha$. For the first term in \eqref{eq:Mnequivbound}, we apply  Lemmas~\ref{lem:psiplusfacts} and~\ref{lem:longcount}, together with $m \le 5\alpha$, to estimate
		\begin{equation}
			\# \partial \mathcal{E}_k \le 2 \sum_{i=1}^m \# \Psi_k^+(W_i) \le \sum_{i=1}^m 8 (2\alpha)^{2k} \le 8m(2\alpha)^{2k}  \le 40(2\alpha)^{2k+1}. 
		\end{equation}
		Thus, 
		\begin{equation} \label{eq:endpointcount}
			\sum_{k=0}^{n-1} \# \partial \mathcal{E}_k \leqc \sum_{k=0}^{n-1} (2\alpha)^{2k+1} = \frac{(2\alpha)^{2n+1} - 2\alpha}{(2\alpha)^2 - 1} \leqc (2\alpha)^{2n}
		\end{equation}
		For the second term in \eqref{eq:Mnequivbound}, using the same two lemmas again we obtain, for any $0 \le j < k \le n-1$,
		\begin{align*}
			\#(\mathcal{S}^+ \cap T^{k-j-1}(\mathcal{S}^-)) \le 2 \sum_{i=1}^m \# \Psi_{k-j}^+(W_i) \le \sum_{i=1}^m 8 (2\alpha)^{2(k-j)} \le 80 (2\alpha)^{2(k-j)}.
		\end{align*}
		Thus,
		\begin{align}
			\sum_{0 \le j < k \le n-1} \#(\mathcal{S}^+ \cap T^{k-j-1}(\mathcal{S}^-)) &\le 80\alpha \sum_{0 \le j < k \le n-1} (2\alpha)^{2(k-j)} 
			= 80\alpha \sum_{k=1}^{n-1} (2\alpha)^{2k} \sum_{j=0}^{k-1} (2\alpha)^{-2j} \\ 
			& \leqc \sum_{k=1}^{n-1} (2\alpha)^{2k+1} 
			\leqc (2\alpha)^{2n}. \label{eq:extracount}
		\end{align}
		Substituting \eqref{eq:endpointcount} and \eqref{eq:extracount} into \eqref{eq:Mnequivbound} completes the proof.
	\end{proof}

	\begin{proof}[Proof of Lemma~\ref{lem:longcount}]
		Fix $W \in \mathcal{W}_u$ with $|W| \in [1,2]$.  We decompose $\Psi_k^+(W)$ into long and short segments:
		$$S_k^+(W) = \{W_i \in \Psi_k^+(W): |W_i| < 1\}, \quad L_k^+(W) = \Psi_k^+(W) \setminus S_k^+(W).$$
		Since $\mathcal{S}^-$ and $\mathcal{S}^+$ are structurally the same after reversing the roles of $C_u$ and $C_s$, it follows from the proof of Lemma~\ref{lem:complexityfact1}(b) that each member of $\Psi_{k}^+(W)$ can generate at most ten short curves in the next generation. Thus, 
		\begin{equation} \label{eq:shortcount+}
			\# S^+_{k+1}(W) \le 10 \, \#\Psi^+_k(W) = 10(\# S^+_k(W) + \# L^+_k(W)).
		\end{equation}
		We may similarly obtain a lower bound on $L_{k+1}^+(W)$. Indeed, it follows from \eqref{eq:forwardinvariant} that $|T(V)| \ge |V|  (\alpha^2/2) \ge \alpha^2/2$ for any $V \in L^+_k(W)$. Since $\Psi_1^+(V)$ contains at most ten short curves,  we see that
		$$ 
		\frac{\alpha^2}{2} \le |T(V)| \le 2 \, \#  L_1^+(V) + S_1^+(V) \le 2 \, \#  L_1^+(V)+10.
		$$
		Thus, $\# L_1^+(V) \ge \alpha^2/8$ for large $\alpha$ and by summing over $V \in L_k^+(W)$ we conclude
		\begin{equation} \label{eq:longcount+}
			\# L_{k+1}(W) \ge \frac{\alpha^2}{8} \# L_{k}(W).
		\end{equation}
		
		Define $$r_k = \frac{\# S^+_k(W)}{\# L^+_k(W)}.$$ Together, \eqref{eq:shortcount+} and \eqref{eq:longcount+} imply that
		\begin{equation} \label{eq:rk}
			r_{k+1} \le \frac{80}{\alpha^2} r_k + \frac{8}{\alpha^2}
		\end{equation}
		for any $k \ge 0$. Note that $r_0 = 0$ since $|W| \ge 1$. Therefore, we find by iterating \eqref{eq:rk} that 
		\begin{equation}  \label{eq:rkbound}
			r_n \le \sum_{k =1}^n \left(\frac{80}{\alpha^2}\right)^k \le \frac{\frac{80}{\alpha^2}}{1-\frac{80}{\alpha^2}} \le \frac{160}{\alpha^2} \le 1
		\end{equation}
		for $\alpha$ sufficiently large. From \eqref{eq:rkbound} we have 
		\begin{equation} \label{eq:2longcount}
			\# \Psi_n^+(W)  = \# L_n^+(W) + \# S_n^+(W) \le 2 \, \#L_n^+(W).
		\end{equation}
		Now, by the control on $DT$ provided by \eqref{eq:matrices}, it follows that 
		$$\sum_{W_i \in \Psi_n^+(W)} |W_i| \le (2\alpha)^{2n} |W| \le 2 (2\alpha)^{2n}.$$
		Since each long segment has length at least one, we deduce
		\begin{equation} \label{eq:longcount3}
			\#L^+_n(W) \le \sum_{W_i \in L^+_n(W)} |W_i| \le 2 (2\alpha)^{2n}.
		\end{equation}
		Putting \eqref{eq:longcount3} into \eqref{eq:2longcount} completes the proof.
	\end{proof}

	\bibliographystyle{plainnat}
	\bibliography{Batchelor}
	
\end{document}